\documentclass[12pt]{article}
\usepackage{amsmath,amssymb,graphicx,color} 

\usepackage[margin=1in]{geometry} 
\usepackage{amsthm}
\usepackage[backend=bibtex,style=numeric]{biblatex}
\usepackage[toc,page]{appendix}
\usepackage{xcolor}
\usepackage{comment}
\usepackage[colorlinks=true,linkcolor=blue]{hyperref}

\newcommand{\footremember}[2]{%
	\footnote{#2}
	\newcounter{#1}
	\setcounter{#1}{\value{footnote}}%
}

\author{Illya Koval \footremember{email}{Illya.Koval@ist.ac.at}}
\title{
 Local strong Birkhoff conjecture and local spectral rigidity of almost every ellipse}
\date{%
	Institute of Science and Technology Austria\\ %
	\today
}

\begin{document}
	
	\maketitle
	
	\newtheorem*{oldtheorem}{Theorem}
	\newtheorem{theorem}{Theorem}
	\newtheorem{lemma}{Lemma}
	\newtheorem{remark}{Remark}
	\newtheorem{proposition}{Proposition}
	\newtheorem{definition}{Definition}
	\newtheorem*{conjecture}{Conjecture}
	\newtheorem{example}{Example}

	\numberwithin{equation}{section}
	\numberwithin{proposition}{section}
	\numberwithin{lemma}{section}
	\numberwithin{definition}{section}

	\newcommand{\reale}{\operatorname{Re}}
	\newcommand{\imag}{\operatorname{Im}}

	\begin{abstract}
		The Birkhoff conjecture says that the boundary of a strictly convex integrable billiard 
		table is necessarily an ellipse. In this article, we consider a stronger notion of 
		integrability, namely, integrability close to the boundary, and prove a local version 
		of this conjecture: a small perturbation of almost every ellipse that preserves 
		integrability near the boundary, is itself an ellipse. We apply this result to study local 
		spectral rigidity of ellipses using the connection between the wave trace of the Laplacian 
		and the dynamics near the boundary and establish rigidity for almost all of them.
	\end{abstract}
	
	\graphicspath{ {./figures/} }

	\section{Introduction}
	A mathematical billiard is a dynamical system, first proposed by G.D. Birkhoff in \cite{birk} 
	as a playground, where “the formal side, usually so formidable in dynamics, almost 
	completely disappears and only the interesting qualitative questions need to be considered”.
	
	Let $\Omega$ be a strictly convex $C^r$ domain in $\mathbb{R}^2$ with $r>3$. Let $x$ be a point in 
	the boundary $\partial \Omega$ and $\varphi$ is angle of a direction $V$ with 
	the clockwise tangent to $\partial \Omega$ at $x$. Let 
	$M := \left\lbrace \right (x, \varphi): x\in \partial \Omega, \varphi \in (0, \pi)\rbrace $. Then, one can consider 
	a billiard map $f : M \rightarrow M$, where $M$ consists of unit vectors with foot $x$ 
	on $\partial \Omega$ and with inward direction $v$. The map reflects the ray from the boundary of 
	the domain elastically, i.e. the angle of the incidence equals the angle of reflection.
	
	This dynamical system has simple local dynamics, however, its study turns out to be really complex and 
	has many important open questions. One group of ``direct'' questions is to pick domains 
	and analyse the properties of the billiard in them. For example, can they be chaotic, have 
	a positive metric entropy, or an open set of periodic points\footnote{Some recent progress was done in \cite{keagan}}, etc? A different way 
	to study the billiards is an indirect one, see e.g. \cite{gutkin}. Given some property of 
	the mathematical billiard in some $\Omega$, can something be said about the shape of $\Omega$?  
	
	In this paper we analyse so called integrable billiards. For example, if $\Omega$ is 
	an ellipse, then the billiard map is integrable, meaning all of dynamics can be described in 
	a relatively simple way using action-angle coordinates, see \cite{chang}. A natural 
	question then arises, the one asked by Birkhoff in \cite{birk} and by Poritsky in \cite{pori} 
	and formulated in the following conjecture:
	
	\begin{conjecture} 
		There are no other examples of integrable billiards.
	\end{conjecture}
	
	Despite its simple-looking statement, the question still remains open. Various methods 
	were developed to attack this problem. For example, in \cite{math} the author has 
	proven, that if the curvature of the domain vanishes at one point, then it cannot be integrable.
	
	The two strongest non-perturbative results are due to Bialy in \cite{bialy} and Bialy and Mironov in \cite{bialymironov}:
	\begin{oldtheorem} (\cite{bialy})
		If the phase space of the  
		billiard ball map is globally foliated  by  continuous invariant  curves which  are  
		not  null-homotopic, then it corresponds to a billiard in a disc.
	\end{oldtheorem}
	\begin{oldtheorem} (\cite{bialymironov})
	A centrally symmetric domain with an integrable billiard is an ellipse. \footnote{See remark \ref{local-Birkhoff}  for a more precise claim}
	\end{oldtheorem}
	
	\medskip 
	Another kind of inverse problems related to integrability is as follows. One can define 
	the length spectrum of a domain, by looking at perimeters of all periodic billiard orbits. 
	The closure of the union is called the length spectrum. How much of information is encoded 
	into this spectrum? This question is studied for example in \cite{Zelditch2004TheIS}.
	
	It turns out the length spectrum is connected with other spectra of the domain such as 
	the Laplace spectrum, the latter being the quantum version of the former. The famous inverse 
	problem of hearing the shape of a drum \cite{kac} in mathematical terms is to determine 
	a domain from its Laplace spectrum. The relation between spectra is explored in \cite{mm} and in \cite{zel}, among other papers. In \cite{iantchenko}, authors developed a new approach for studying this connection.
	
	Specifically, the problem is as follows. Given a bounded smooth planar domain and Laplace 
	equation inside of it, along with some standard specified boundary conditions, can the domain 
	be uniquely determined by the eigenvalues up to isometries? The relation to billiard dynamics 
	comes from the fact that the Laplace operator is structurally similar to the euclidean metric, 
	with billiard balls moving along the broken geodesics of the latter. 
	
	Several results were obtained by studying the various trace asymptotic, related to the Laplacian. 
	For example, discs were determined to be spectrally rigid since both perimeter and area are 
	heat trace asymptotic invariants and discs minimise the ratio between them, see \cite{brown}. 
	In \cite{mm}, authors considered wave trace asymptotic and obtained that some parametrized 
	family of domains, determined by an ODE on the curvature, are spectrally determined.  
	This method generally results in studying various Euler-Lagrange equations. However, 
	there are currently a limited number of feasible equations to study and it's doubtful whether 
	any studied domains satisfy them. As such, this method has problems studying general or specific domains.
	
	In a series of papers by Popov and Topalov deal with this problem using more dynamical approach. 
	Their project consists of five papers already, with \cite{popovtopalov} being the last one 
	at the moment. They study the connection between Laplace spectrum and KAM-theory. Specifically, 
	they obtained spectral rigidity of elliptical tables in the class of analytic symmetric domains under 
	weak conditions. Their results also apply to more general class of systems, for example to 
	multidimensional manifolds with broken geodesic flow.
	
	Another method was introduced in \cite{hezzel}. Their main idea was to connect the wave trace 
	singularities to the length spectrum and the dynamical side of the picture. They manage 
	to determine that the domain is integrable just by looking at the wave trace. This allowed them 
	to obtain spectral uniqueness for ellipses close to the disc. Combining this method with 
	our result about local Birkhoff conjecture	we prove local spectral rigidity for almost all ellipses.
	
\subsection{Strong Birkhoff Conjecture and rigidity of integrable nearly elliptic billiards }
	
	Of course, one should rigorously define what integrability means. Many definitions were introduced. 
	For example, one can say that the map $M$ is integrable if there exists a smooth integral of 	
	motion near the boundary.
	
	Here, we study one of the most common definitions of integrability, i.e. preservation of a smooth 
	foliation by caustics near the boundary. Specifically, we study the preservation of rational caustics.
	
	\begin{definition}
		A smooth convex curve $\Gamma \subset \Omega$ is  called a  caustic,  if  whenever  a trajectory is tangent to it, then it remains tangent after each reflection\footnote{There are other 
		types of caustics, e.g. those formed by two branches of hyperbolas in an ellipse. We do not study them in this paper}. 
	\end{definition}

	If $\Omega$ is a disk, then its caustics are concentric circles by a classical Lemma of Poncelet. For an ellipse, its caustics are co-focal ellipses. Note, that if one considers tangent directions, a caustic defines a natural map on $\partial \Omega$ onto itself, as such it has a rotation number. We define
	
	\begin{definition}
		We say that $\Gamma$ is an integrable rational caustic for the billiard map in $\Omega$, if the corresponding (non-contractible) invariant curve $\hat{\Gamma} \subset M$ consists of periodic points; in particular, the corresponding rotation number is rational.
	\end{definition}
	
	Particularly, the rotation number $\omega \in (0, 1)$, however we would only consider $\omega \in (0, 1/2]$ since others correspond to reverse dynamics on the same caustic. Caustics near the boundary correspond to small rotation numbers, so we would study those. All rational caustics are present in a disc, while other ellipses lack a caustic with $\omega = 1/2$.
	
	In the recent years, there have been several articles on this topic, concerning a local case, namely, when 
	$\Omega$ is a small deformation of an ellipse. For example, in \cite{adsk}, authors prove that if locally caustics 
	with rotation numbers $\frac{1}{q}$ for $q \ge 3$ are preserved near an ellipse with small eccentricity, then $\Omega$ 
	is also an ellipse. Later \cite{ks}  generalized this, studying ellipses with other eccentricities. However, these results 
	rely, for example, on preservation of caustics with rotation number $1/3$ and $1/4$, and those are not near the boundary.
	
	Our goal is to study domains with caustics only near the boundary 
	$\partial \Omega$.
	
	\begin{definition}
		Let $q_0 > 2$. If the billiard map, associated to $\Omega$ admits integrable rational 
		caustics with rotation numbers $\frac{p}{q}$ for all $0 < \frac{p}{q} \le \frac{1}{q_0}$ we say that $\Omega$ is $q_0$-rationally integrable.
	\end{definition}
	Domains that are $q_0$-rationally integrable and are near ellipses of small eccentricities 
	studied in \cite{hks}. However, they only succeeded in proving rigidity for $q_0 \le 5$ unconditionally. Our next result proves their Conjecture 1.9 that such ellipses are rigid and 
	generalises it to ellipses that are not nearly-circular. 
	
	\begin{theorem}
		Let any $q_0>0$ and $\mathcal E_0$ be an ellipse of eccentricity $0 < e < 1$ and semi-focal distance 
		$c$. Let $k \ge 39$ and $K>0$. Then there exist a locally finite set $\mathcal Z(q_0)\subset (0,1)$ 
		and $\varepsilon = \varepsilon(e, c, K, q_0)>0$ for any $e\notin \mathcal Z(q_0)$ such that 
		the following holds: if $\Omega$ is a $q_0$-rationally integrable $C^k$-smooth domain so 
		that $\partial \Omega$ is $C^k$-$K$ close and $C^1$-$\varepsilon$ close to $\mathcal E_0$, then 
		$\Omega$ is itself an ellipse. \label{maintheorem}
	\end{theorem}
	\begin{remark} \label{local-Birkhoff} 
	This result proves a local version of a strong Birkhoff conjecture for most ellipses. 
	Namely, for almost every eccentricities $e$ being integrable near the boundary 
	and being close to an ellipse of eccentricity $e$ implies it is an ellipse. 
	
	We can also state the result of Bialy-Mironov: any centrally symmetric domain with 
	$4$-rationally integrable billiard is an ellipse \cite{bialymironov}. 
	\end{remark}
	
\subsection{Spectral Rigidity of Ellipses}	
To state the next result we need auxiliary definitions.
	\begin{definition}
		A set $\mathcal Z \subset [0, 1)$ is called locally finite if it has no accumulation points in $[0, 1)$.
	\end{definition}

	\begin{definition}
		A set $\mathcal Z \subset [0, 1)$ is called small, if its accumulation points in $[0, 1)$ form a locally finite set.
	\end{definition}
	
	\begin{remark}
		Note that we do not need all the caustics with $0<p/q \le 1/{q_0}$. In fact, we only need to preserve 
		caustics with bounded $p \le 7$, with only a finite number of them having $p>1$. It is useful, since usually 
		it may be easier to prove their existence. In fact, we just need $2$ libration numbers: $(p=1, p=3)$,  or $(p = 1, p=5)$, or $(p=1, p=7)$. At least one of these pairs gives us rigidity, though we don't know which one exactly, see \cite{artin}.
	\end{remark}
	
	Now we describe 
	known spectral results and state our spectral rigidity results for ellipses. 
	Hezari and Zelditch \cite{hezzel12} proved local rigidity for ellipses, assuming the deformation to be $\mathbb{Z}_2\times\mathbb{Z}_2$ symmetric. They only assume that the deformation is $C^\infty$ smooth instead 	of analytical. In \cite{hezzel12} Dirichlet and Neumann boundary conditions are studied, while \cite{vi21} is devoted to Robin boundary conditions. 
	
	
	However, we think that the strongest result and the one 
	heavily used in this paper - is the one from Hezari and Zelditch \cite{hezzel}. In that paper, they prove 
	the global spectral rigidity for ellipses with small eccentricity. Let us present some ingredients of the proof. 
	
	For nearly-circular domains Hezari and Zelditch transform a global 
	problem into a local one. That techniques are similar to the one uses for proving the rigidity of 
	discs. Then, they prove the existence of a smooth generating function in a neighbourhood 
	of certain periodic points, namely, those whose orbits form $q$-gons inside $\Omega$. This, 
	together with studying the length spectrum of such domains, allows them to prove that the deformation 
	preserves caustics with rotation numbers $1/q$ for all $q \ge 3$. Finally, they can use 
	the aforementioned dynamical result of  \cite{adsk} to prove rigidity.
	
	As we can see, there is a method of bringing dynamical results into the spectral rigidity problem. 
	One could ask whether it's possible to get some additional results, using developments from \cite{adsk}. 
	For example, \cite{ks} deals with ellipses with arbitrary eccentricities, can spectral rigidity be proven for 
	those as well? The answer is that it's rather challenging to do, since the existence of the smooth generating 
	functions for orbits with rotation number like $1/5$ is unclear.

	However, in our dynamical result we do not need caustics with large rotation numbers, so we always should be near the boundary. Billiard dynamics near the boundary has a few of good properties. For example, there are Lazutkin coordinates that nearly 
	straighten the dynamics. This allows us to guarantee that the smooth generating functions exist. 
	Our main spectral rigidity result is the following theorem:
	
	\begin{theorem}
		Let $\mathcal E_0$ be an ellipse of eccentricity $0<e<1$ and semi focal distance $c$. 
		Let $k \ge 39$ and $K>0$. Then there exist a small set $\mathcal Z\subset (0,1)$ and 
		$\varepsilon = \varepsilon(e, c, K)>0$ for any $e\notin \mathcal Z$ such that $\mathcal E_0$ is uniquely 
		determined (up to isometries) by its Laplace spectrum among domains $\Omega$ with 
		$\partial \Omega$ being $C^\infty$ smooth, $C^k$-$\,K$ and 
		$C^{10}$-$\,\varepsilon$ close to $\mathcal E_0$.
		\label{theoremspectr}
	\end{theorem}

	\begin{remark}
		The result says that most ellipses are locally spectrally rigid. Note that our spectral result is local, compared to \cite{hezzel}. They obtain a global result, since disks have the minima of a spectrally determined function. So, domains close to the disk cannot be isospectral to the domains away from the disk by the continuity of the aforementioned function. For general ellipses this argument doesn't work, so the result is local. However, in an appendix we prove global length spectral rigidity, assuming strong global Birkhoff conjecture. 
	\end{remark}

	\begin{remark}
		The small set $\mathcal Z$ consists of several components. First of all, there is locally finite set $\mathcal Z_e$ for which the dynamical result doesn't work. Secondly, there are some challenges for spectral rigidity when certain periodic billiard orbits of different types have the same length in an ellipse. The set of those $e$ is called $\mathcal I_e$ and is studied in the last section, its accumulation set $\mathcal A_e$ is also studied and a first few points of the latter are computed there. See Figure \ref{incigrpah} for a plot of these sets.
	\end{remark}

	Finally, in order to study Laplace spectral rigidity, we use its connection to the length spectrum and essentially study the rigidity of the latter object. So, we get a similar result for the length spectrum rigidity automatically from our Laplace spectrum discussion. 
	
	\begin{theorem}
		Let $\mathcal E_0$ be an ellipse of eccentricity $0<e<1$ and semi focal distance $c$. Let $k \ge 39$ and $K>0$. Then there exist a locally finite set $\mathcal Z \subset (0,1)$ and $\varepsilon = \varepsilon(e, c, K)>0$ such that $\mathcal E_0$ is 
uniquely determined (up to isometries) by its length spectrum among domains $\Omega$ with 	$\partial \Omega$ being $C^\infty$ smooth, $C^k$-$\,K$ and $C^{10}$-$\,\varepsilon$ close to $\mathcal E_0$.
		\label{theoremlen} 
		\end{theorem}

	\begin{remark}
	We do not need all the spectral information, and we only use singularities of the wave trace near the multiples of the perimeter for Laplace case and part of the length spectrum near the multiples of the perimeter for its case.

	Smallness of the exceptional set of eccentricities  $\mathcal Z$ 
	implies it is of measure $0$, nowhere dense, countable.
	
	We denote that in this paper $e$ is always being an eccentricity of an ellipse, exponentials are denoted by $\exp$.
	\end{remark}

	\subsection{Outline of the proof}
	
	The proof breaks up into $3$ parts, each part was influenced by different papers. 
	
	The fist part deals with the proof of Theorem \ref{maintheorem} for ellipses that are close to the circle. 
	This part is an improvement to \cite{hks}. That paper also was dealing with the same problem. They 
	have obtained rigidity for ellipses with small eccentricity for $q_0 = 3, 4, 5$. For larger $q_0$, they 
	weren't able to get an unconditional result. Specifically, they have proven that ellipses are rigid, 
	provided some constant matrix (independent of deformation) is non-degenerate. The dimensions of 
	the matrix were of order $q_0$. The general formula for the coefficients of the matrix was not obtained, 
	so proving full rank condition was challenging. 
	
	The main idea behind the proof is the following. Each deformation can be described by a function on 
	a circle. In order to preserve $p/q$ caustic, the deformation should satisfy several conditions, each 
	of these can be thought of as some function on Fourier harmonics of deformation being zero. 
	These functions can be of course complicated and non-linear, but we may consider an expansion of 
	them over the deformation. The zeroth order term should of course be $0$, since ellipses are integrable. 
	So, we consider the linear term. If the dependency between the set of Fourier coefficients and 
	the collection of linearized functions (it can be thought of as a linear operator) is full rank or injective, 
	then no matter how we deform, there will always be some function in the family with non-zero linear 
	term, so its caustic will be destroyed by said deformation. 
	
	We consider the following expansion of a deformation in elliptic coordinates \eqref{elliptic-coord}:
	\begin{equation}
		\partial \Omega = \{(\mu_0 + \mu(\varphi), \varphi), \ \varphi \in [0, 2\pi]\}.
	\end{equation}
	
	Here, if $\mu_0$ is a constant value, it describes an ellipse, while $\mu$ is a perturbation. Let $\mu(\varphi) = a_0 + \sum_{k} a^+_k \cos(k\varphi) + a^-_k \sin(k\varphi)$ be the Fourier expansion of $\mu$ in elliptic coordinates.
		
	We derive explicit formulas for the linearised conditions in this paper. Specifically, 
	if we want to preserve $\omega = p/q$ caustic, we have a set of conditions that a deformation $\mu$ should satisfy. 
	
	These conditions are written in \eqref{apqdef} in the original integral form.  However, it is easier to consider them in the Fourier form, written below. Here, $A_{p, q, j}$ are some well-defined coefficients, independent of deformation.
	
	\begin{equation}
		\sum_{j=0}^{\infty} A^\pm_{p, q, j} a_j^\pm= O_e(q^8\left\| \mu \right\|_{C^1}^2).
		\label{startsys}
	\end{equation}

	The LHS of this formula is a linear functional, evaluated at a deformation. We will call those functionals $A_{p, q}^{\pm}$. Note the comma between $p$ and $q$, since we have 
	many conditions for the same caustic, for $p$ and $q$ may share a common divisor. 
	For example, functionals $A_{1,4}^\pm$ and $A_{2, 8}^\pm$ are both involved in preserving $1/4$ caustic. We also note that the conditions on odd and even parts of deformation are identical ($A_{p, q, j}^- = A_{p, q, j}^+$), so we drop $\pm$ from the notation as redundant.
	
	Our main goal for the paper would be to prove a basis property for these functionals. Then, any non-trivial deformation would break some of the conditions. Hence, we postpone working with a deformation until Section $7$, instead focusing on the functionals.
	
	The condition \eqref{startsys} arises from the following. If a $p/q$ caustic exists, then all the periodic orbits with $q$ reflections and $p$ rotations should share the same length. This is true, since the periodic orbits are always the critical points of the length functional. A length functional on arbitrary $q$ points on the boundary is defined in a following way:
	
	\begin{equation}
		\mathcal L(P_1, P_2, \cdots, P_q) = |P_1P_2|+|P_2P_3| + \cdots + |P_qP_1|.
	\end{equation} 
	
	Since we have a one-parameter family of critical points, a functional is constant along the family, as we have stated. 
	
	\begin{figure}
		\includegraphics[width=6cm]{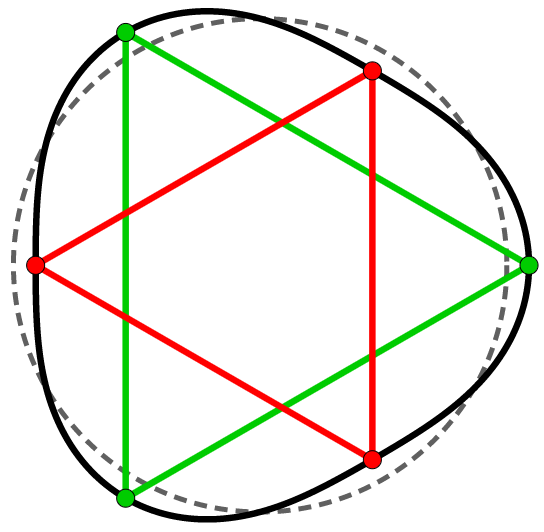}
		\centering
		\caption{$q$-harmonic destroys $p/q$ caustic. Normal deformation of a unit circle by $0.09\cos 3 \varphi$ produces $2$ orbits with $\omega = 1/3$ of different lengths.}
	\end{figure}
	
	It was proven in \cite{adsk}, that a periodic orbit in $\Omega$ is always a deformation of some periodic orbit in $\mathcal{E}$. So, a change of lengths under this deformation should stay constant along a one-parameter family of periodic orbits. This change of lengths is essentially the sum of values of $\mu$ at reflection points in a periodic orbit in an ellipse. These reflection points are $2\pi/q$ - equispaced in an action-angle parametrization of an ellipse, defined in Section 2. So, in these coordinates $\mu$ cannot have a $q$-periodic component, otherwise an orbit that reflects at minima of this component would get much shorter than the one falling on the maxima. So, harmonics of deformation in these action-angle coordinates that have frequencies, divisible by $q$, should be negligible, and exactly this is written in \eqref{startsys}, but converted to elliptic coordinates.  
	
	We establish several formulas 	for $A_{p, q, j}$ that allow for easier study of the conditions, see Lemma \ref{qlemma} and Lemma \ref{klemma}. They involve $\mathbf{q}$ and $k$, the nome 
	and the eccentricity of the caustic, defined in the next section.
	
	An exact formula is given by \eqref{nomesum}. This formula works for every ellipse, not only close to 
	the disc, as well as for every caustic, not only close to the boundary. 
	
	However, these formulas get a nicer representation when turned into an expansion for small eccentricities. The most important result is 
	an expansion in terms of eccentricity of an ellipse:
	
	\begin{theorem} 
		For every integer $j,q > 1$ of the same parity and $p$, $p/q < 1/2$ the function 
		$A_{p, q, j}(e)$ has the following expansion.
		\begin{equation}
			A_{p,q,j}(e) = \binom{j+y-1}{y}  
			\frac{e^{2y}}{2^{4y} \cos^{2y} \frac{\pi p}{q}  }  + O(e^{2y+2}), \; \; e \rightarrow 0, \; j \le q
		\end{equation}
		\begin{equation}
			A_{p,q,j}(e) = (-1)^y\binom{j}{y}\frac{e^{2y}}{2^{4y} \cos^{2y} 
			\frac{\pi p}{q}  }  + O(e^{2y+2}), \; \; e \rightarrow 0, \; j > q
		\end{equation}
		Here, $y = \frac{|q-j|}{2}$. Moreover, the following bound holds for small $e$ with some constant $C$.
		\begin{equation}
			|A_{p,q,j}(e)| \le C^{3y+j+1}e^{2y}
			\label{abound}
		\end{equation}
		If they are of different parity, $A_{p, q, j} = 0$.
		\label{th3}
	\end{theorem} 
	
	Hence, if $\Omega$ is a $q_0$-rationally integrable domain, then \eqref{startsys} should hold for all $0<\frac{p}{q} \le \frac{1}{q_0}$, with $\gcd(p, q) \le 2$, so the functional $A_{p, q}$ is available to us. In order to prove Theorem \ref{maintheorem}, we need to find a system of linear functionals (a linear operator) on harmonics that is complete, namely, find $(p_i, q_i),\ i\ge 1$ such that satisfying all conditions forces $\mu$ to be an elliptical deformation, that means to lie in the $5$-parameter family.
	
	One can see that the right part in \eqref{startsys} is not exactly zero, but an error term of the second order. However, if the operator is invertible, then we would get that $\mu$ is nearly elliptical deformation: the distance from $\Omega$ to the closest ellipse is of lesser order than $\mu$. So, if we originally consider $\Omega$ as a deformation of the closest ellipse, the size of perturbation will decrease by an order of magnitude. Repeating the same discussion for the new ellipse one would invert a new operator and find an even closer ellipse, and that would lead to contradiction. This idea was already used in several other papers, see e.g. \cite{adsk}, \cite{hks} and \cite{ks}. So, we primarily focus on inverting the operator and related issues.
	
	The formulas in Theorem \ref{th3} come from the following ideas. According to \cite{adsk}, to preserve a $p/q$ caustic one should make $q$-harmomic small. However, it is not a harmonic in elliptic coordinates, but a harmonic in the unique coordinates for each caustic, called {\it action--angle}. They are good since they simplify the associated dynamics to caustic quite well. However, they are different for each caustic, and we need a uniform parametrization to study all these conditions in. Hence we are using elliptic coordinates. So, to get these formulas and $A_{p, q, j}$, we convert the elliptic harmonics into action angle and demand the $q$-th one to be small. This conversion leads to elliptic integral calculations, that give rise to Theorem \ref{th3}.
	
	In the proof we greatly use elliptic nomes, Jacobi elliptic functions, Lambert series as well as several combinatorial identities connected to Stirling numbers.
	
	\medskip
	
	\begin{center}
		\large{A finite dimensional reduction}
	\end{center}
	Since we are currently dealing with small eccentricities,	it can be shown that the system of functionals can be reduced to essentially a finite-dimensional system. More precisely, for some large $q_1$ and $q>q_1$ it turns out that $A_{p, q}^\pm$ is close to $E_q^\pm$, a functional that gives the $q$-th Fourier coeffiient. Thus, 
	$A_{p, q}$ for all $q>q_1$ up to an error annihilates all Fourier 
	coefficients of $\mu$ with indices $q$. 
	
	 One can easily see it from Theorem \ref{th3}. From \eqref{abound} we can conclude that in \eqref{startsys} $A_{p, q, q}$ is by far the biggest coefficient by an order of $e^2$. So, $A_{p, q}(\mu) \approx 0$ states that essentially $a_{q}^\pm$ is very small. For larger eccentricities that is generally not the case and one should study harmonics in other parametrisations of an ellipse, not an elliptic one. We will come to it later.
	
	Hence, one can essentially not focus on large harmonics and caustics with large $q>q_1$, since these functionals are very close to the basis elements, and study small harmonics where we have the main struggle. This way, we essentially reduce our infinite dimensional operator to a finite dimensional one. 
	
	We should note that in the first part of the proof we already deal with a finite dimensional case. One could reduce it in the first part, but it is easier to just do it in the second part, which uses the first part as its core. In the second part we deal with ellipses of arbitrary eccentricity, and there is an intermediate step for them. So we cannot reduce until we have made this step in the second part.

	\begin{center}
		\large{A finite dimensional nondegeneracy}
	\end{center}
	
	The main difficulty we will face is with Fourier coefficients 
	whose indices $< q_0$, since vanishing of the other 
	$a_q^\pm$'s is closely related to satisfying 
	the respective conditions \eqref{startsys} with $p=1$. 
	This connection is used in \cite{adsk}. For harmonics with small indices, however, we lack $1/q$ caustic along 
	with $A_{1, q}$ functional, so we are proposing the following method. We study the dependency 
	of other functionals $A_{p',q'}$ on the $q$-th harmonic for $0<p'/q' \le 1/q_0$. The main idea 
	of this paper is to find a finite collection of functionals $A_{p_i, q_i},\ i=1,\dots,N(q_0)$ 
	having full rank or being non-degenerate, i.e. being in the kernel of all the $A_{p_i, q_i}$ implies $a_q^\pm=0$ for 
	$q< q_0$.  
	
	\medskip

	One can link this collection of conditions	
	to nondegeneracy of some finite square matrix. The coefficients 
	of this matrix will be related to $A_{p, q, j}$. They will be their main term coefficients in Theorem \ref{th3}. The matrix 
	will be a constant one, independent of $e$ and $\mu$. This constant matrix arises from the original one, when we take 
	a meaningful limit as $e \rightarrow 0$. As stated, we would use the irrationality of matrix coefficients and the algebraic 
	field theory to prove invertibility. We will discuss this now.
	
	\medskip
	
	\subsection{Algebraic structure and Vandermonde reduction}
	
	Let's give two examples of motivated by algebraic nature of our matrix. The first example would be 
	simple, while the second is more involved and is closely related to our problem.\\
	
	\textbf{Example 1.}
	Prove that the matrix
	\begin{equation}
		\begin{pmatrix}
			\sqrt[3]{2} & 5 & 2 \\
			4 & 3\sqrt[3]{2} & 7 \\
			2 & 8 & 1\\ 
		\end{pmatrix}
	\end{equation}
	is non-degenerate. \medskip 
	
	Of course, one could just compute the determinant of the matrix approximately and prove it. However, we can do it in a more 
	conceptual way. We substitute $z$ instead of $\sqrt[3]{2}$. 
	Then, we can find the determinant to be the polynomial from $z$ 
	of degree $2$ over rationals. If our original matrix had been 
	degenerate, $\sqrt[3]{2}$ would have been a root of this 
	polynomial. This, however, would mean that our polynomial 
	divides the minimal polynomial of $\sqrt[3]{2}$ over rationals. It is 
	impossible, of course, since the said minimal polynomial, 
	$z^3 - 2$, is of degree $3$, so it cannot divide polynomial 
	of degree $2$. So, the matrix is non-degenerate. 
	\medskip

	So, in this example we used irrationality of $\sqrt[3]{2}$ and algebraic field theory to prove non-degeneracy of a matrix with rational numbers. \\
	
	\textbf{Example 2.} Let $\alpha_j = e^{\frac{2\pi j}{5}i}$, $j=1,2,3,4$. Prove that the following matrix is non-degenerate:
	
	\begin{equation}
		\begin{pmatrix}
			3 & 7 & 0 & 0\\
			4 & 1 & 2 & 0\\ 
			1 & \alpha_1 & \alpha_1^2 & \alpha_1^3\\ 
			1 & \alpha_2 & \alpha_2^2 & \alpha_2^3 
			\label{112}
		\end{pmatrix}.
	\end{equation}
	
	This problem is fairly similar to the problems we will soon encounter. In particular, one could interpret the latter two lines in a matrix as representing preservation of caustics with $\omega = 1/5$ and $2/5$, respectively. Moreover, the method of handling this problem is very similar to the method in the main proof. 
	
	Here, one could also compute the determinant approximately. Alternatively, one could substitute $\alpha_2 = \alpha_1^2$ into 
	the matrix and use a method from previous example for 
	$\alpha_1$ instead of $\sqrt[3]{2}$:
	
	\begin{equation}
		\begin{pmatrix}
			3 & 7 & 0 & 0\\
			4 & 1 & 2 & 0\\ 
			1 & z & z^2 & z^3\\ 
			1 & z^2 & z^4 & z^6 \\ 
		\end{pmatrix}.
		\label{zmatrix}
	\end{equation}
	
	This will run into some problems though, because the resulting 
	polynomial from $z$ will be of degree $8$, while the minimal 
	polynomial of $\alpha_1$ over rationals, that is 
	$z^4 + z^3 + z^2 + z + 1$, only has degree $4$. Since $8  > 4$, 
	the determinant can divide the minimal polynomial. Let's propose 
	a viable option. 
	
	Recall that a number $m$ is a primitive root module $n$ if $m^j$ travels through all the residues, except $0$ modulo $n$. For example, $2$ is 
	a primitive root modulo $5$.
	\begin{lemma}
		The matrix \eqref{zmatrix} is non-degenerate, since $2$ is a primitive root modulo $5$.
	\end{lemma}
	
	Our method extends to the following 
	\begin{proposition}
		Let $z= e^{\frac{2\pi }{p}i}$ with $p > 3$ being prime and $2$ is a primitive root modulo $p$. Then the matrix \eqref{zmatrix} is non-degenerate.
	\end{proposition}
	
	\begin{remark}
		Notice that $2$ is a not a primitive root modulo $7$ and the method of 
		proof of this proposition does not apply.  
	\end{remark}
	
	To prove both statements we propose a method, which we call a \textit{Vandermonde reduction}. We will reduce the matrix \eqref{112} to the Vandermonde matrix \eqref{117}.

	\begin{proof} The proof is by contradiction.
		Suppose the determinant is zero. Then, we know that
		\begin{equation}
			\left(1, \alpha_2, \alpha_2^2, \alpha_2^3 \right) 
			\in Lin\left((3, 7, 0, 0), (4, 1, 2, 0), 	
			\left(1, \alpha_1, \alpha_1^2, \alpha_1^3 \right) \right).
		\end{equation}
		We already know that the determinant of \eqref{zmatrix} divides 
		$z^4 + z^3 + z^2 + z + 1$. Since $5$ is prime, it has roots at all the unity roots, except $z = 1$, for example, 
		at $z = \alpha_2$. Substitute it into \eqref{zmatrix}:
		\begin{equation}
			\begin{vmatrix}
				3 & 7 & 0 & 0\\
				4 & 1 & 2 & 0\\ 
				1 & \alpha_2 & \alpha_2^2 & \alpha_2^3\\ 
				1 & \alpha_2^2 & \alpha_2^4 & \alpha_2^6 \\ 
			\end{vmatrix} = \begin{vmatrix}
				3 & 7 & 0 & 0\\
				4 & 1 & 2 & 0\\ 
				1 & \alpha_2 & \alpha_2^2 & \alpha_2^3\\ 
				1 & \alpha_4 & \alpha_4^2 & \alpha_4^3 \\ 
			\end{vmatrix} = 0.
		\end{equation}
		\begin{equation}
			\left(1, \alpha_4, \alpha_4^2, \alpha_4^3 \right) \in Lin\left((3, 7, 0, 0), (4, 1, 2, 0), \left(1, \alpha_2, \alpha_2^2, \alpha_2^3 \right) \right) 
			\subset 
		\end{equation}
		\begin{equation} \nonumber
			\qquad \qquad \qquad 
			Lin\left((3, 7, 0, 0), (4, 1, 2, 0), 	\left(1, \alpha_1, \alpha_1^2, \alpha_1^3 \right) \right).
		\end{equation}
		Further substituting $z = \alpha_4$ and so on leads us to 
		\begin{equation}
			\left(1, \alpha_{2^k}, \alpha_{2^k}^2, \alpha_{2^k}^3 \right) \in   Lin\left((3, 7, 0, 0), (4, 1, 2, 0), 	\left(1, \alpha_1, \alpha_1^2, \alpha_1^3 \right) \right)
		\end{equation}
		Now, since $2^k$ goes through all the residues modulo $5$ (here we use that $2$ is a primitive root), we get:
		\begin{equation}
			\left(1, \alpha_j, \alpha_j^2, \alpha_j^3 \right) \in Lin\left((3, 7, 0, 0), (4, 1, 2, 0), 	\left(1, \alpha_1, \alpha_1^2, \alpha_1^3 \right) \right), \; j = 1, 2, 3, 4.
			\label{5ex}
		\end{equation}
		
		This would mean that all these four vectors are linearly dependent on each other. Consequently, 
		
		\begin{equation}
			\begin{vmatrix}
				1 & \alpha_1 & \alpha_1^2 & \alpha_1^3\\ 
				1 & \alpha_2 & \alpha_2^2 & \alpha_2^3\\ 
				1 & \alpha_3 & \alpha_3^2 & \alpha_3^3\\ 
				1 & \alpha_4 & \alpha_4^2 & \alpha_4^3 \\ 
			\end{vmatrix} = 0.
			\label{117}
		\end{equation}
		
		This is of course impossible, since we have a Vandermonde of $\alpha_1, \alpha_2, \alpha_3, \alpha_4$, and it is nonzero, since all of them 
		are distinct from each other. This means, that the original determinant 
		couldn't have been zero, so the system is complete.\\
	\end{proof}
	
	Similar algorithm is described in Section 4 of this paper to prove the main result. The main differences is that instead of $5$ we take arbitrary prime number $q$, instead of roots of unity we have their real parts (cosines) and instead of determinant we study the rank of the matrix.
	
	\subsection{Selection of $(p_i, q_i)$}
	
	It can be noted that we used several properties of number $5$ in the second example. First, it was important that $5$ is a prime number, since otherwise the minimal polynomial would have been different. Moreover, we needed to get all the roots in the Vandermonde matrix, so effectively we have used that $2$ is a primitive root modulo $5$. For example, if we had chosen $7$ instead of $5$, we would have only connected $3$ roots: $\alpha_1, \alpha_2$ and $\alpha_4$, since $\alpha_4^2 = \alpha_1$. We wouldn't have a way of proving \eqref{5ex} for $j = 3, 5$ and $6$. So, in this case the method wouldn't work.
	
	Note that $2$ is a primitive root modulo $q$ if the minimal subgroup of $\mathbb{F}_q^*$, containing $\left\lbrace 1, 2 \right\rbrace $ is $\mathbb{F}_q^*$ itself. Since this example is similar to our problem, we give the following definition:
	
	\begin{definition}
		A prime number $q$ is said to be $q_0$-good, if $q > 7q_0$ and at least one of $3$, $5$ or $7$ is 
		the primitive root modulo $q$.
		\label{q0agood}
	\end{definition}
	
	The existence of such numbers is related to the following conjecture:
	
	\begin{conjecture}(Artin's conjecture)
		Every given integer that is nor a perfect square, nor $-1$ is a primitive root modulo infinitely many primes. 
	\end{conjecture} 
	 
	 The question is still open, although in \cite{artin} it was proven that it can only fail for some $2$ primes, hence 
	 the choice of $3, 5, 7$ (we cannot use $2$ as a prime). So, there is an infinitely many $q_0$-good numbers 
	 for every $q_0$.
	 
	 \subsection{Analytic continuation and rigidity for non-perturbative ellipses}
	 
	 The second part of the proof extends the results of the first part to ellipses with arbitrary eccentricity. It uses analytic continuation in terms of eccentricity $e$ to obtain them. Specifically, we can prove that ellipses with degenerate and not full rank operators have eccentricities that behave like the zeros of holomorphic function, meaning the their set is either the whole domain or is locally finite. Since in the first part we have proven the system is not degenerate when $e$ is close to zero (it is degenerate at $0$, though), we can say that the set is locally finite.
	 
	 In this part, we use some facts from \cite{ks}. They also study rigidity of ellipses with arbitrary eccentricity and use complex and functional analysis in their work. We should note however, that there are strong fundamental differences in our part. The main one is that we study analytic dependency on $e$, while \cite{ks} studies it with respect to boundary parametrization in a fixed ellipse. They also deal with the width of the strip of analycity, while we don't care about the width. 
	 
	 Specifically, we turn the previously mentioned set of functionals $A_{p_i, q_i}$ into a linear operator depending on $e$ and acting on the $L^2$ space of deformations. To prove rigidity, the operator should have $1$ in its resolvent set. Otherwise the system may have  degeneracy and a non-trivial solution. 
	 
	 This operator turns out to be compact and analytical over $e$ in terms of \cite{kato}. This analyticity particularly means that each element of the matrix is analytical over $e$ and the operators are uniformly bound. So, we would be able to use the result in \cite{kato} that states that $1$ is an eigenvalue for every $e$ in the domain or only for a locally finite set.
	 
	 First, we prove analycity of the coefficients of the operator, similar to $A_{p, q, j}$. They are just some functions, related to elliptic integrals and caustic parameters. For example, we need to prove that the dependency between caustic eccentricity and its rotation number and eccentricity is holomorphic. Of course, everything is not defined when $e$ or the rotation number are complex, but we claim we can extend the definition holomorphically. We extend everything into an extremely narrow neighborhood of the real line. This strip is uniform for all the caustics and the approach also works for caustics with other rotation numbers.
	 
	 Then, we construct the mentioned operator. However, if we just construct an operator, consisting of coefficients in Theorem \ref{th3} minus identity (since we prove that $1$ isn't an eigenvalue), it wouldn't look compact. Being compact requires the coefficients to decay at infinity, while the assymptotics in Theorem \ref{th3} for $y = 1$ and $j\rightarrow \infty$ hint otherwise. 
	 
	 The reason are the poorly chosen coordinates on the boundary. While elliptic coordinates functioned great in the first part, the Lazutkin coordinates $\vartheta$ make things easier when $q\rightarrow \infty$. They will also be defined in the next section. The reason is that they are still uniform and do not depend on a caustic, but they approximate action angle for small rotation numbers. They lack nice conversion formulas though, so we can't use them to efficiently study matrix coefficients and to do field theory. 
	 	 
	 After we have constructed the operator and proven its qualities, we can apply the result from \cite{kato}. However, to get what we want we also should prove that $1$ is not identically an eigenvalue. So, we consider the case of small eccentricities. Since we have proven a full rank property for those, $1$ will not be eigenvalue for them, so we make the first case impossible. 
	 
	 One may use similar analytic continuation over $e$ in other similar problems. For coefficient analycity we do not require the rotation number to be close to $0$. However, we suggest that for large $q$ we use closer to the boundary caustics. Otherwise Lazutkin coordinates will lose their main feature and the operators will fail to be compact. This can possibly be solved by letting the rotation numbers to approach a KAM-curve and studying its action-angle coordinates instead of Lazutkin, but we would still need to prove some other version of Proposition \ref{smallharm}.
	 
	 We also note that for other problems one may just consider $e = 0$ to prove that $1$ is not identically an eigenvalue, instead of expanding everything as $e\rightarrow0$. It will work in the context of \cite{ks} where we have all the $1/q$ caustics, but in our case we do not have full rank at $e = 0$, since $A_{1, q}$ and $A_{2, q}$ functionals coincide for the disc. That is why we considered the case of small eccentricities separately in the first part.
	 
	 At the end of the second part there is a technical section that derives rigidity of ellipses from the basis property (or from the fact that $1$ is not an eigenvalue). Similar proofs were given in \cite{adsk}, \cite{hks} and \cite{ks}. Our proof is extremely close to one in \cite{ks}, since the background is similar (caustic preservation functions forming a basis with non-small $e$), so we go over the proof relatively briefly. We note that \cite{ks} uses the words "Fourier coefficients" when expanding the deformation over a basis, since their elements of the functional basis are similar to trigonometric Fourier basis in properties. In our case there is less similarity (several functionals may share the same frequency for example), so we won't talk about the coefficients as Fourier. But still it doesn't affect the proof.
	 
	 \subsection{Laplace spectral rigidity}
	 
	 The third and the last part of the paper is devolved to study the Laplace spectral rigidity of ellipses of arbitrary eccentricity. We use the method of extending the dynamical result to the spectral case, already performed in several papers. We will leave the technical results for later, but our proof is based on Poisson relation, that states that each singularity in the wave trace (some distribution on the real line, that can be derived from Laplace spectrum) can be attributed to the billiard orbit(s) and is located at its length. 
	 
	 So, if the Laplace spectrum is preserved, then so is the wave trace as well as the length spectrum in some form. Then, we can derive the existence of caustics from it. If we are able to derive the existence of all the caustics that we have used for the dynamical result (we call this set $\mathcal F$), then we prove the deformation to be an ellipse.
	 
	 The main problem here is so called cancellation: if two orbits share the same length, along with some other characteristics, their contributions to the wave trace may cancel each other out, making it smooth at the point of their lengths. This is very bad for us, because it means that there may be points in the length spectrum that we have no way of obtaining from the Laplace spectrum. And these points may give information about the caustic.
	 
	 This part is closer to \cite{hezzel}, since their paper also dealt with similar problems, but for a nearly circular ellipse. Some part of their paper is spent to construct the smooth $q$-loop function $L_q(s)$ for all the orbits with $p=1$, for nearly circular domains. This phase function is very important, since without it we cannot find if the orbits with the same $p$ and $q$ cancel each other out or not (since they may share the same length). If there is such a function, then there can't be such a cancellation, as mentioned in \cite{mm} and \cite{hezzel}. Luckily for us, if an orbit is close to the boundary and its $p$ is bounded, then the $q$-loop function exists due to the Lazutkin coordinates $\vartheta$. So, this allows us not to focus on the study of distributions and just study the length spectrum of the domain. 
	 
	 The existence of the phase function only guarantees that there is no cancellation with orbits with the same $(p, q)$, but another orbits in $\Omega$ may still cancel with them. That is why we fear the incidence of orbits with different rotation numbers inside an ellipse. Because then under a small deformation the caustics may break up, there can be a lot of cancellations with no way of studying them (since the deformation is arbitrary), so we can lose dynamical information. 
	 
	 Hence we separately study the lengths of periodic orbits inside an ellipse. We also prove that they are holomorphic in $e$ and prove that this incidence may only happen on the small set of eccentricities.

	\subsection{Plan of the paper}
	
	In Section $2$ we will remind the reader various notions about ellipses. This includes properties of billiards inside of them, as various identities and definitions, related to elliptic functions. We will use these objects throughout the paper. 
	
	In Section $3$ we will apply those identities and develop formulas for coefficients $A_{p, q, j}$. Using them, we are going to prove Theorem \ref{th3}. 
	
	Section $4$ is essentially devoted to proving local strong Birkhoff conjecture for nearly circular ellipses. Using Theorem \ref{th3} we will reduce the rigidity problem to non-degeneracy of a finite matrix. Then, using algebraic field theory, we are going to prove the matrix to be full rank. This section breaks up into two parts: studying odd and even frequency harmonics of a perturbation. The mechanisms are slightly different and easier in the odd case.
	
	In Section $5$ we start to work on analytic continuation. Specifically, we take some notions, introduced in Section $2$, and see how they can be extended to complex eccentricities. We also study periodic orbit length in ellipses there to use in the spectral part.
	
	In Section $6$ we continue working with complex eccentricities. We introduce a rigidity operator there and study its properties for complex eccentricities using functional analysis and results from Section $5$. Our goal is to prove operator to be invertible. We show that it is either invertible for almost every $e$, or for no $e$ at all. Then, we link it with the results of Section $4$. Particularly, we reduce this operator to a finite matrix studied in Section $4$ for small $e$. This means that the operator is invertible for some small eccentricities.
	
	In Section $7$ we use the fact that the operator is invertible to complete the proof of Theorem \ref{maintheorem}. This section uses the same methods as \cite{ks}, so we do not go into the details of the proof. 
	
	In Section $8$ we deal with Laplace and Length spectral rigidity of ellipses. We use Theorem \ref{maintheorem} and prove Theorem \ref{theoremspectr} and \ref{theoremlen}. After it there is an Appendix, proving global length spectral rigidity of ellipses assuming global strong Birkhoff conjecture.
	
	\subsection{Acknowledgments}
	
	The author acknowledges the partial support of the ERC Grant \#885707. He also thanks Vadim Kaloshin for proposing the idea of the project and greatly aiding the implementation.  The author is also grateful to Hamid Hezari, Amir Vig, Steve Zelditch, Comlan E. Koudjinan, Corentin Fierobe, Ngo Nhok Tkhai Shon and Roman Sarapin for useful discussions. The author also acknowledges partial support of ISTern summer program. The project started in the summer of 2021, when the author was an intern at IST Austria.  
	
	\section{Elliptic functions and rational caustic preservation condition}

	Let us introduce some of the important notions, related to ellipses, that we will use in this paper. For simplicity we will assume that semi-major axis of an ellipse is $1$.
	
	First of all, every ellipse has 
	semi-major and semi-minor axis $1$ and $b$, as well as 
	eccentricity $e=\sqrt{1-b^2}$ and linear eccentricity 
	$c=e$.  Elliptic coordinates on a plane take the following form:
	\begin{equation}
		\begin{cases}
			x = c \cosh \mu \cos \varphi\\
			y = c \sinh \mu \sin \varphi
		\end{cases}
		\label{elliptic-coord}
	\end{equation}
	
	When $\mu = \mu_0 = \cosh^{-1}(1/e)$, $\varphi \in \left[0, 2\pi \right]$ gives a so called elliptic parametrization of a boundary of an ellipse. We will also study a perturbation of a domain using these coordinates and a periodic function $\mu(\varphi)$. From now on, 
	we have that
	\begin{equation}
		\partial \Omega = \mathcal{E}_{e, c} +
		\mu(\varphi).
	\end{equation}
	
	
	We also consider a family of caustics -- co-focal ellipses $C_\lambda$ parametrized by a parameter $\lambda$: 
	\begin{equation}
		C_\lambda = \left\lbrace (x, y)\in \mathbb{R}^2 : 
		\frac{x^2}{1 - \lambda^2}  + \frac{y^2}{b^2 - \lambda^2} = 1
		\right\rbrace, \; 0 < \lambda < b.
		\label{lambdadef}
	\end{equation}
	
	We shall also use another parameterization of caustics
	$k_\lambda = \frac{e}{\sqrt{1 - \lambda^2}}$, with $k_\lambda > e$ being the eccentricity of the caustic and a rotation number $\omega$. We also use incomplete and complete elliptic integrals of the first  and second kind, namely
	\begin{equation}
		F(\varphi, k) = 
		\int_{0}^{\varphi}\frac{d\tau}{\sqrt{1 - k^2\sin^2\tau}}; \; \; K(k) = F\left( \frac{\pi}{2}, k\right),
	\end{equation}
	
	and
	
	\begin{equation}
		E(\varphi, k) = \int_0^\varphi\sqrt{1 - k^2\sin^2\tau}d\tau; \; \; \; E(k) = E\left(\frac{\pi}{2}, k \right).
	\end{equation}
	
	Then, the following formula holds:
	
	\begin{equation}
		\omega(\lambda, e) = \frac{F(\arcsin(\lambda/b), k_\lambda)}{2K(k_\lambda)}.
	\end{equation}
	
	The rotation number $\omega$ is strictly increasing in $\lambda$  and goes to $0$ as $\lambda \rightarrow 0$. To simplify formulas we denote $\phi_\lambda = \arcsin(\lambda / b)$.
	
	We also write the boundary parametrization induced by caustic $C_\lambda$, denoted by $\theta$, such that the orbit 
	starting at $\theta_0$ and tangent to $C_\lambda$ hits the boundary 
	at $\theta_0 + 2\pi \omega_\lambda$. It is called an action-angle parametrization. We note that this parametrization is different for every caustic. We have the following relation:
	
	\begin{equation}
		\theta(\varphi, e, \lambda) = \frac{\pi}{2} \frac{F(\varphi, k_\lambda)}{F(\frac{\pi}{2}, k_\lambda)}
		\label{thetaphi}
	\end{equation}

	\begin{figure}
		\includegraphics[width=10cm]{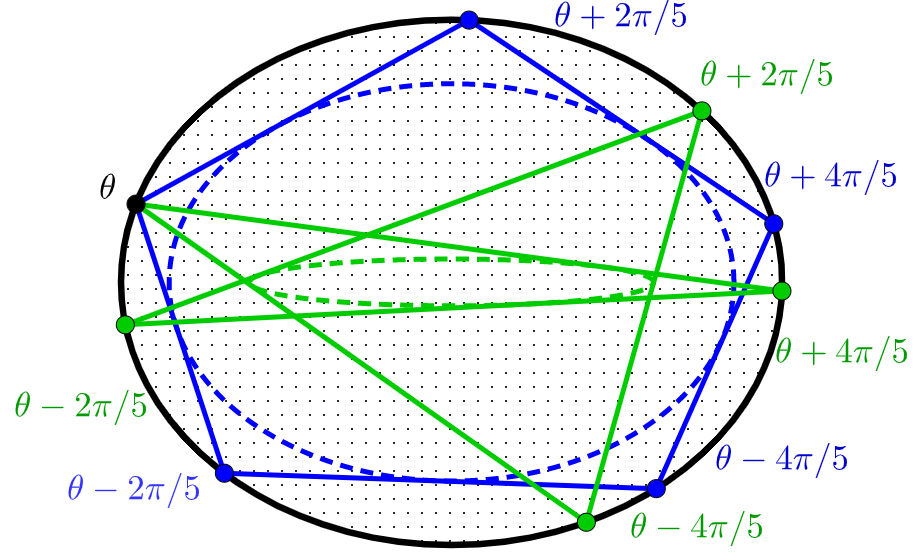}
		\centering
		\caption{Action-angle coordinates in an ellipse for $\omega = 1/5$ and $\omega = 2/5$. Parametrizations of different caustics do not agree. }
	\end{figure}
	
	There is also a Lazutkin parametrization of an ellipse, that we will denote $\vartheta$. They can be defined in terms of curvature, the following way:
	
	\begin{equation}
		\vartheta = C\int_{0}^{s}\rho^{-2/3}(s')ds',
	\end{equation}
	
	where $s$ is a parameterization of the boundary in terms of its length, while $C$ is a normalizing constant, so that $\vartheta \in [0, 2\pi]$.
	
	They are the limit of action-angle coordinates $\theta$ as the rotation number $\omega$ goes to zero. The formulas for it are the same, as for action-angle, one should just use $\omega = 0$, or $\lambda = 0$ and $k = e$. Because they are the limit, they nearly linearize billiard dynamics near the boundary, assuming $p$ is bounded.
	
	One can find more information about these objects and their relation to the billiards in \cite{hks} and \cite{ks}.
	
	Now we introduce some objects, related to the elliptic integrals. $k$ is called a modulus and $\varphi$ - an amplitude. One can define a complementary modulus $k' = \sqrt{1 - k^2}$. After that an elliptic nome can be introduced:
	
	\begin{equation}
		\mathbf{q} = \exp \left( - \frac{\pi K(k')}{K(k)}\right), \; \; \; 0 < \mathbf{q} < 1.
	\end{equation}
	
	The nome has the following expansion for small $k$.
	
	\begin{equation}
		\mathbf{q} = \frac{k^2}{16} + \frac{k^4}{32} + O(k^6), \; \; \; k \rightarrow 0
		\label{nomek}
	\end{equation}
	
	There is also a Jacobi amplitude function, inverse to the elliptic integral:
	
	\begin{equation}
		F\left( am(\theta, k), k\right)  = \theta, \; \; \; \varphi(\theta, e, \lambda) = am\left( \frac{4K(k_\lambda)\theta}{2\pi}, k_\lambda \right). \label{phitheta}
	\end{equation}
	
	Then, there is an important relation for us:
	
	\begin{equation}
		am\left( \frac{4K(k)\theta}{2\pi}, k\right)  = \theta + 2\sum_{n=1}^\infty \frac{\mathbf{q}^n}{n\left( 1 + \mathbf{q}^{2n} \right) } \sin(2n\theta)
		\label{amformula}
	\end{equation}
	
	for real $\theta$ and $0<k<1$.
	
	For the spectral result it will be important to study the lengths of periodic orbits, corresponding to a caustic. They all share the same length, that according to \cite{sieber} is
	
	\begin{equation}
		l_{p, q}^1 = 2q\sin\phi-\frac{2eq}{k}E(\phi, k) + \frac{4ep}{k}E(k).
		\label{lengthell}
	\end{equation}
	
	This length travels from $2q\sin(\pi p/q)$ to $4p$ as $e$ goes from $0$ to $1$.   
	
	We shall use the following results from Section 3, \cite{hks}. 
	
	\begin{lemma} (Lemma 3.2 \cite{hks})
		There exists $C > 0$ such that for each $e \in \left[0, \frac{1}{2} \right] $ and $\omega \in \left(0, \frac{1}{2} \right) $, we have 
		\begin{equation}
			\left|\lambda(e, \omega) - b \sin \omega \pi \right| \le Ce^2
		\end{equation}
		\label{lemma32}
	\end{lemma}
	
	There are also other periodic orbits in an ellipse. We will not focus on studying them, but we use them in the spectral result. First of all, there are bouncing ball orbits, that just travel along axes of an ellipse. For the given $p$, their lengths are given by (major and minor resp.) by $4p$ and $4\sqrt{1-e^2}p$, the last one traveling between $4p$ and $0$ for $e \in [0, 1)$.
	
	The last class of periodic orbits are orbits, ones staying tangent to the hyperbolae. These orbits are located inside the "eyes" of the phase cylinder and all have a rotation number $1/2$. We won't use these orbits for dynamical results, but we need to know their lengths to prove they won't get in the way for the spectral result, making an incidence. 
	
	We will once again use \cite{sieber} to study their lengths. We also note that since the usual notion of rotation number $p/q$ or even number $p$ at all fails here, one can introduce a notion of short axis libration number $1 \le \tilde{p} < q/2$. This number indicates how many times did the orbit rotate around the center of the eye. We denote $\tilde{\omega} = \tilde{p}/q$. We also demand of course that $q$ is even, since reflection points alternate between eyes. 
	
	Hyperbolae correspond to the same equation \eqref{lambdadef}. But now, instead of $0 < \lambda^2 < 1 - e^2$, we have the case $1 - e^2 < \lambda^2 < 1$. Many of the same definitions above can be reintroduced for them, including their eccentricity:
	
	\begin{equation}
		k^{-1} = \frac{\sqrt{1 - \lambda^2}}{e}, \; \; \phi = \arcsin\left(\frac{\sqrt{1 - e^2}}{\lambda} \right) , \; \; \tilde{\omega} = \frac{F(\phi, k^{-1})}{2K(k^{-1})},\; \; \theta = \frac{\pi}{2}\frac{F(\varphi, k)}{F(\arcsin k^{-1}, k)}.
		\label{hyperb}
	\end{equation}
	
	Their lengths are also given in \cite{sieber}:
	
	\begin{equation}
		l_{\tilde{p}, q}^2 = 2q\sin \phi - 2eqE(\phi, k^{-1}) + 4e\tilde{p}E(k^{-1}).
		\label{lenghyper}
	\end{equation}
	
	Unlike the earlier caustic orbits, these do not exist for every eccentricity. Specifically, they only exist when
	
	\begin{equation}
		e \in \left(\cos \tilde{\omega} \pi, 1\right).
	\end{equation}  
	
	Particularly, their lengths travel from $2q\sin(\pi\tilde{p}/q)$ to $4\tilde{p}$ in this range of $e$.
	
	Although it is not the subject of the paper, we can see that orbits, tangent to hyperbolae and ellipses share many similarities. In fact, instead of integrability near the boundary, one can study integrability near a bouncing ball orbit. This was the subject of several recent papers, like \cite{treschev} and \cite{wangzhang}. Surprisingly, they state that there are billiard tables, where dynamics near a bouncing ball orbit is conjugate to a rotation by any irrational angle. First, in a series of papers, starting with \cite{treschev}, Treschev developed a formal series method for studying those domains, while \cite{wangzhang} proved that these domains are in fact Gevrey regular.

	\section{An alternative formula for caustic preservation}
	
	The formula \eqref{startsys} for caustic preservation is linear over $\mu$, meaning we can treat is as a linear condition on Fourier coefficients of the deformation. We can in fact write down this functional in a rather nice form. To get the coefficient in front of the harmonic, one should of course just substitute this harmonic as $\mu$. The following ideas will be using elliptic harmonics $\varphi$, due to the formula \eqref{amformula}. The point is to study this condition in the coordinate system independent on the caustic, so we cannot do it in action-angle $\theta_{p/q}$. The Lazutkin coordinates $\vartheta$ are also good, because they are close to $\theta_{p/q}$ for small rotation numbers. However, we cannot use them here, because we don't have the respective formula. So, we will be using elliptic for now. They are generally not close to $\theta_{p/q}$ and this will cause us significant problems later. 
	
	So, the main idea is to compute the following integral:
	
	\begin{equation}
		A_{p, q, j} = \frac{1}{\pi}\int_{0}^{2\pi}\cos\left( j\varphi(\theta_{p/q})\right) \cos(q\theta_{p/q})d\theta_{p/q}.
		\label{idef}
	\end{equation}
	
	This integral tells us how much does the $j$-cosine harmonic destroy $p/q$ caustic. The intuition behind the sum is as follows. Equation \eqref{startsys} just tells us that $q$-harmonic in action angle should be small. We want to express this condition in terms of elliptic harmonics. To do it, we just express $\mu$ into elliptic Fourier series:
	
	\begin{equation}
		\int_0^{2\pi}\mu(\theta_{p/q}) \cos(q\theta_{p/q})d\theta_{p/q} = \int_0^{2\pi} \sum_{j} \left( a_j \cos\left( j\varphi(\theta_{p/q})\right) \right) \cos(q\theta_{p/q})d\theta_{p/q}
	\end{equation}
	
	Now to get the left part of \eqref{startsys} out of this and \eqref{idef} one just needs to change the order of integration and summation.
	
	One doesn't have to integrate cosines, but sine on cosine will give $0$, and sine on sine will be identical. We use \eqref{phitheta} and then \eqref{amformula}.
	
	\begin{equation}
		A_{p, q, j}=\frac{1}{\pi}\int_{0}^{2\pi} \cos j \left( am \left( \frac{4K(k_{p/q})}{2\pi}\theta_{p/q}; k_{p/q}\right) \right) \cos q \theta_{p/q} d\theta_{p/q}=
	\end{equation}

	\begin{equation}
		\frac{1}{\pi}\int_{0}^{2\pi} \cos j \left( \theta_{p/q} + 2\sum_{ n=1}^{+\infty}\frac{\mathbf{q}^n}{n\left( 1+\mathbf{q}^{2n}\right) }\sin(2n\theta_{p/q}) \right) \cos q \theta_{p/q} d\theta_{p/q}
	\end{equation}
	
	Now we want to replace cosines with exponents, so that the series for $A_{p,q,j}$ would be simpler.
	\begin{equation}
	\reale \frac{1}{2\pi}\int_{0}^{2\pi} \exp \left( i j \left( \theta_{p/q} + 2\sum_{n=1}^{+\infty}\frac{\mathbf{q}^n}{n\left( 1+\mathbf{q}^{2n}\right) }\right) \sin(2n\theta_{p/q}) \right) \left( \exp (-i q \theta_{p/q}) + \exp (i q \theta_{p/q}) \right)  d\theta_{p/q}
	\end{equation}
	We transfer $\theta_{p/q}$ term to the right exponents and expand the left exponent
	of $A_{p,q,j}$.
	\begin{equation}
		\reale\frac{1}{2\pi} \sum_{l = 0}^{\infty} \frac{2^li^lj^l}{l!} \int_{0}^{2\pi}  \left( \sum_{n=1}^{+\infty}\frac{\mathbf{q}^n}{n\left( 1+\mathbf{q}^{2n}\right) } \sin(2n\theta_{p/q}) \right)^l\left(  \exp (-i (q - j) \theta_{p/q}) +   \exp (i (q + j) \theta_{p/q}) \right) d\theta_{p/q}
	\end{equation}
	Now let's for simplicity denote $\theta = \theta_{p/q}$. We also see that the exponents on the right are similar, so we will compute the following for $y = \frac{q-j}{2}$ and $y = \frac{-q-j}{2}$.
	
	\begin{equation}
		\reale\frac{1}{2\pi} \sum_{l = 0}^{\infty} \frac{2^li^lj^l}{l!} \int_{0}^{2\pi}  \left( \sum_{n=1}^{+\infty}\frac{\mathbf{q}^n}{n\left( 1+\mathbf{q}^{2n}\right) } \sin(2n\theta) \right)^l\exp (-2iy\theta) d\theta
	\end{equation}
	
	Now we use exponential formula for the sines. Then, we get a sum over all non-zero integers $n$. This sum converges exponentially, since $\mathbf{q}<1$.
	
	\begin{equation}
		\reale\frac{1}{2\pi} \sum_{l = 0}^{\infty} \frac{j^l}{l!} \int_{0}^{2\pi}  \left( \sum_{n\ne0}\frac{\mathbf{q}^n}{n\left( 1+\mathbf{q}^{2n}\right) } \exp(2in\theta) \right)^l\exp (-2iy\theta) d\theta
	\end{equation}

	Now we want to evaluate this integral. If we expand the $l$-power, there would be the sum of exponents in the integral. Their frequencies would be integer, since $2y$ is an integer, hence the integral wouldn't be zero only if the frequency would be zero. That means that we are only interested in a term with $\exp(2iy\theta)$ in the $l$-power. This of course means that $y$ is an integer. We can represent this integral in the following way.
	
	\begin{equation}
		\reale \sum_{l = 0}^{\infty} \frac{j^l}{l!} \sum_{x_1+x_2+\ldots + x_l = y}^{x_1\ne0, \ldots, x_l \ne 0} \frac{\mathbf{q}^{x_1+\ldots+x_l}}{x_1 \ldots x_l \left(1 + \mathbf{q}^{2x_1} \right) \ldots  \left(1 + \mathbf{q}^{2x_l} \right) }
	\end{equation}
	The sum here is something like a composition in combinatorics, meaning we do care about the order of the elements. So, the total result will be the following. 
	
	\begin{equation}
		 A_{p, q, j} = \sum_{l = 0}^{\infty} \frac{j^l}{l!} \sum_{x_1+x_2+\ldots + x_l = \frac{q-j}{2} or \frac{-q-j}{2}}^{x_1\ne0, \ldots, x_l \ne 0} \frac{\mathbf{q}^{x_1+\ldots+x_l}}{x_1 \ldots x_l \left(1 + \mathbf{q}^{2x_1} \right) \ldots  \left(1 + \mathbf{q}^{2x_l} \right) }
		 \label{nomesum}
	\end{equation}
	
	There is a bit nicer way to write this. Consider the following element of $\ell_{\mathbb{Z}}^1$, that we will call $w$.
	
	\begin{equation}
		w_n = \frac{\mathbf{q}^n}{n\left( 1 + \mathbf{q}^{2n}\right) }, \; n \in \mathbb{Z} \setminus {0}, \; \; \; \; \; w_0 = 0
	\end{equation}

	Then, we can introduce multiplication on $\ell_{\mathbb{Z}}^1$, using convolution $\ast$. Then, we get that our result is just
	
	\begin{equation}
		A_{p, q, j} = \exp_\ast(jw)_\frac{q-j}{2} + \exp_\ast(jw)_\frac{-q-j}{2}.
	\end{equation}

	Here, $q$ and $j$ should share the same parity, otherwise the result is just zero. This formula works for every eccentricity and rational caustic.
	
	\subsection{Bounding the results}
	
	Now, we want to achieve some uniform bounds on these coefficients to study their expansions as $e$ goes to zero or to study them as elements of an infinite matrix. The main idea is to use the exponential decay in the sum in \eqref{nomesum}. Here, we will demand that $\mathbf{q}$ is sufficiently small, since in goes to zero as $e\rightarrow 0$. First, we take the absolute value of every term and bound $\frac{\mathbf{q}^x}{1+\mathbf{q}^{2x}}\ge \mathbf{q}^{|x|}$. 
	\begin{equation}
	|A| \le  \sum_{l = 0}^{\infty} \frac{j^l}{l!} \sum_{x_1+x_2+\ldots + x_l = \frac{q-j}{2} or \frac{-q-j}{2}}^{x_1\ne0, \ldots, x_l \ne 0} \frac{\mathbf{q}^{|x_1|+\ldots+|x_l|}}{|x_1| \ldots |x_l|} \le \sum_{l = 0}^{\infty} \frac{j^l}{l!} \sum_{|x_1|+|x_2|+\ldots + |x_l| \ge y}^{x_1\ne0, \ldots, x_l \ne 0} \frac{\mathbf{q}^{|x_1|+\ldots+|x_l|}}{|x_1| \ldots |x_l|}  
	\label{boundi}
	\end{equation}

	Here, $y = \frac{|q-j|}{2}$. We also have relaxed a sum a bit. Now we can just proceed to the sum over positive integers:
	
	\begin{equation}
		|A| \le \sum_{l = 0}^{\infty} \frac{j^l}{l!} \sum_{x_1+x_2+\ldots + x_l \ge y}^{x_1, \ldots, x_l \in \mathbb{N}} \frac{2^l\mathbf{q}^{x_1+\ldots+x_l}}{x_1 \ldots x_l} \le \sum_{s=y}^{\infty} \mathbf{q}^s\sum_{l = 0}^{\infty} \frac{(2j)^l}{l!}  \sum_{x_1+x_2+\ldots + x_l = s}^{x_1, \ldots, x_l \in \mathbb{N}} \frac{1}{x_1 \ldots x_l} .
	\end{equation}
	
	Now we obviously have $l \le s$ for a non-zero result, hence we can modify the sum a little.
	
	\begin{equation}
		|A|  \le \sum_{s=y}^{\infty} (2\mathbf{q})^s\sum_{l = 0}^{s} \frac{j^l}{l!}  \sum_{x_1+x_2+\ldots + x_l = s}^{x_1, \ldots, x_l \in \mathbb{N}} \frac{1}{x_1 \ldots x_l} =  \sum_{s=y}^{\infty} (2\mathbf{q})^s \binom{j+s-1}{s},
		\label{binomform}
	\end{equation}
	
	since the sum over $l$ is a known formula for a binomial coefficient, discussed in \cite{stirling}. Now we do some rough estimates, like the following.
	\begin{equation}
	|A_{p,q,j}|  \le  \sum_{s=y}^{\infty} (2\mathbf{q})^s 2^{j+s-1} = 2^{j-1} \sum_{s=y}^{\infty} (4\mathbf{q})^s = \frac{2^{2y+j-1}\mathbf{q}^y}{1-4\mathbf{q}} \le 2^{2y+j+1}\mathbf{q}^y
	\end{equation}

	\subsection{Asymptotic for small nomes}
	
	The previous formula allows us to produce an asymptotic for $A$ as $\mathbf{q} \rightarrow 0$. In particular, we propose the following lemma:
	
	\begin{lemma}
		For every natural $j,q > 1$ of the same parity and $p$, $p/q < 1/2$ the function $A_{p, q, j}(\mathbf{q}(e))$ has the following expansion.
		\begin{equation}
			A_{p, q, j}(\mathbf{q}) =  \binom{j+y-1}{y}\mathbf{q}^y + O(\mathbf{q}^{y+2}), \; \; \mathbf{q} \rightarrow 0, \; j \le q
		\end{equation}
		\begin{equation}
			A_{p,q,j}(\mathbf{q}) = (-1)^y  \binom{j}{y}\mathbf{q}^y + O(\mathbf{q}^{y+2}), \; \; \mathbf{q} \rightarrow 0, \; j > q
		\end{equation}
		Here, $y = \frac{|q-j|}{2}$.
		\label{qlemma}
	\end{lemma} 
	
	\begin{proof}
		We start by analyzing \eqref{nomesum}. Firstly, if the sum of $x$-s is equal to $\frac{-q-j}{2}$ we already get an order of $\mathbf{q}^{\frac{q+j}{2}}$ just by following the same bounds \eqref{boundi}, so it will go to the error term. We are only interested in the case $\frac{q-j}{2}$. Then, let $\sigma = -1$ if $j>q$ and $1$ otherwise. Then, by inverting all the $x$-s if $\sigma = -1$, we get:
		
		\begin{equation}
			A = \sum_{l = 0}^{\infty} \frac{\sigma^lj^l}{l!} \sum_{x_1+x_2+\ldots + x_l = y}^{x_1\ne0, \ldots, x_l \ne 0} \frac{\mathbf{q}^{x_1+\ldots+x_l}}{x_1 \ldots x_l \left(1 + \mathbf{q}^{2x_1} \right) \ldots  \left(1 + \mathbf{q}^{2x_l} \right) } + O(\mathbf{q}^{y+2}).
		\end{equation}
		Now, either all the $x$-s are positive, or the sum of their absolute values is at least $y+2$. When the latter is true, we just use the same bounds \eqref{boundi} and get an order of $O(\mathbf{q}^{y+2})$. So, we get:
		
		\begin{equation}
			A = \sum_{l = 0}^{y} \frac{\sigma^lj^l}{l!} \sum_{x_1+x_2+\ldots + x_l = y}^{x_1, \ldots, x_l \in \mathbb{N}} \frac{\mathbf{q}^y}{x_1 \ldots x_l \left(1 + \mathbf{q}^{2x_1} \right) \ldots  \left(1 + \mathbf{q}^{2x_l} \right) } + O(\mathbf{q}^{y+2}).
		\end{equation}
	
		Now, we have a finite sum and we can collect the common term $\mathbf{q}^y$. We can also get rid of $1+\mathbf{q}^{2x}$, since it goes to the error term.
		
		\begin{equation}
			A =  \mathbf{q}^y \sum_{l = 0}^{y} \frac{\sigma^lj^l}{l!} \sum_{x_1+x_2+\ldots + x_l = y}^{x_1, \ldots, x_l \in \mathbb{N}} \frac{1}{x_1 \ldots x_l} + O(\mathbf{q}^{y+2})
		\end{equation}
	
	The inside sum is related to Stirling numbers of the first kind $S_n^m$. According to \cite{stirling}, the sum reduces to
	
	\begin{equation}
		A = \mathbf{q}^y \sum_{l = 0}^{y} \frac{\sigma^lj^l(-1)^{y-l}}{y!}S_y^l + O(\mathbf{q}^{y+2}).
	\end{equation} 
	
	For $j > q$ the result follows again from \cite{stirling}:
	
	\begin{equation}
			A = (-1)^y\mathbf{q}^y \frac{1}{y!}\sum_{l = 0}^{y} j^lS_y^l + O(\mathbf{q}^{y+2}) = (-1)^y\binom{j}{y}\mathbf{q}^y +  O(\mathbf{q}^{y+2}).
	\end{equation}
	
	For $j \le q$ the result also follows from \cite{stirling}:
	
	\begin{equation}
		A = \mathbf{q}^y \frac{1}{y!}\sum_{l = 0}^{y} (-1)^{y+l}j^lS_y^l + O(\mathbf{q}^{y+2}) = \binom{j+y-1}{y}\mathbf{q}^y + O(\mathbf{q}^{y+2})
	\end{equation} 

	\end{proof}
	
	\subsection{From nome to eccentricity}
	
	Now we know that $\mathbf{q}$ depends on $k_{p/q}$ and it -- on $e$. We want to express both the bound and the asymptotic first through $k_{p/q}$ and then through $e$. We easily achieve the first step by using \eqref{nomek}:
	
	\begin{lemma}
		For every natural $j,q > 1$ of the same parity and $p$, $p/q < 1/2$ the function $A_{p, q, j}(k_{p/q}(e))$ has the following expansion.
		\begin{equation}
			A_{p,q,j}(k) = \binom{j+y-1}{y} \left( \frac{k^{2y}}{2^{4y}} + \frac{yk^{2y+2}}{2^{4y+1}}\right)  + O(k^{2y+4}), \; \; k \rightarrow 0, \; j \le q
		\end{equation}
		\begin{equation}
			A_{p,q,j}(k) = (-1)^y \binom{j}{y} \left( \frac{k^{2y}}{2^{4y}} + \frac{yk^{2y+2}}{2^{4y+1}}\right)  + O(k^{2y+4}), \; \; k \rightarrow 0, \; j > q
		\end{equation}
		Here, $y = \frac{|q-j|}{2}$. Moreover, the following bound holds for small $k_{p/q}$.
		\begin{equation}
			|A_{p,q,j}(k)| \le 2^{3y+j+1}k^{2y}
		\end{equation}
	\label{klemma}
	\end{lemma} 
	
	Now we will use Lemma \ref{lemma32} to express $k_{p/q}$ in terms of the eccentricity and get the bound and the expansion. 
	\begin{equation}
		k^2 = \frac{e^2}{1 - \lambda^2}
	\end{equation}

	To achieve bounds we also can use that $k_{p/q} < k_{1/3}$, so we can just bound $k_{1/3}$. In particular, we get Theorem \ref{th3}.

	\section{Finite-dimensional matrices for near-circular ellipses}
	
	We have some knowledge about coefficients $A_{p, q, j}$. The original purpose was to use them to study functionals $A_{p, q}$ and the linear operator that arises from combining them. As mentioned earlier, we want to cutoff this operator to a finite dimensional square matrix that connects small frequency harmonics and $A_{p, q}$ with small $q$. We cannot just take all the $A_{p, q, j}$ that we have, because there are more of those, then harmonics, so we have to choose between them, since we want a square matrix. The coefficients of this matrix will just be $A_{p, q, j}$, as one can see from \eqref{startsys}. Important thing to note is that the first five harmonics are the one close (tangent to) elliptic perturbations, so we will study them separately, but now we will have $j \ge 3$. Also, $q_1$ that will be defined is unrelated to the first element of $q_i$. The following lemma is the main goal of this section.
	
	\begin{lemma}
		For any $q_0$, there exists a cutoff $q_1 > q_0$, and a family $\{p_i, q_i\}_{i=1}^{q_1 - 2}$, such that the following is satisfied: $\forall i: p_i/q_i < 1/q_0, \; \; q_i \le q_1$. Moreover, the dependency of functionals $A_{p_i, q_i}$ on the Fourier harmonics of the deformation starting from frequency $j = 3$ to $j = q_1$ is non-degenerate for small $e$. This means that if we create a square matrix $A$ of size $q_1 - 2$ with its $(i, j)$ element equal to $A_{p_i, q_i, j}$ where $1 \le i \le q_1-2$, $3 \le j \le q_1$, it will have nonzero determinant for small $e > 0$.
		\label{mainfieldlemma}
	\end{lemma}
	
	A few important notes should be said here. First, in the lemma we describe a square matrix, but to find $\{p_i, q_i\}$ it is easier to add all the possible functionals $A_{p, q}$ with $p/q < 1/q_0$ and $q \le q_1$ as new rows, making the matrix rectangular. This is not a big deal, since we would only need to proof that this matrix has a rank of $q_1 - 2$. Then we can remove excess rows not reducing the rank. We will get a square matrix and the $A_{p, q}$ it rows correspond to will become $p_i$ and $q_i$. Another note is that $A_{p, q, j}$ is zero when $q$ and $j$ are of different parity, so $A$ is just a direct sum of two matrices, one corresponding to odd $j$ and $q$, another for even. Inverting $A$ is equivalent to inverting both small matrices, so the problem splits in two. We will start by inverting odd matrix, since it is simpler, later we will invert the even one. We also say that if Lemma \ref{mainfieldlemma} holds for some $q_1$, then it also holds for larger ones, as will become evident. So, we can take different $q_1$ in even and odd case, and later just take the maximum.
	
	\subsection{Odd nodes}
	
	\subsubsection{Changing the matrix}
	Our first step is to modify the matrix a little by multiplying the columns and rows by some values. This won't change the rank of the matrix. Right now we will study odd nodes, so $j = 2k+1$ and $q = 2r+1$. We consider $k \ge 1$ and $r > r_{0, odd} = \frac{q_0-1}{2}$.
	
	Specifically, we will make the following transformation:
	
	\begin{equation}
		\tilde{A}_{p, r, k} =  \frac{2^{4r - 3k-1} \cos^{2r-2}\frac{\pi p}{q}}{e^{2(r - k)}} A_{p, r, k}
		\label{scaleodd}
	\end{equation}

	Then, we get the following:
	
	\begin{equation}
		\tilde{A}_{p, r, k} = \binom{r+k}{r-k}   \left( \cos \frac{2\pi p}{2r+1} + 1 \right)^{k-1}     + O(e^{2}), \; \; e \rightarrow 0, \; k \le r
	\end{equation}
	\begin{equation}
		\tilde{A}_{p, r, k} = O(e^{2}), \; \; e \rightarrow 0, \; k > r
	\end{equation}
	
	We study the finite matrix, so now we can take the limit as $e$ goes to zero. If it is full rank, then our matrix is full rank for small eccentricities. We get
	
	\begin{equation}
		\hat{A}_{p, r, k} = \binom{r+k}{r-k}   \left( \cos \frac{2\pi p}{2r+1} + 1 \right)^{k-1},\; k \le r;  \; \; \; \hat{A}_{p, r, k} = 0,\; k > r.
		\label{matrixel}
	\end{equation}
	
	The new matrix is independent of $e$ and the deformation. It is some constant matrix. We can also define limit functionals $\hat{A}_{p, q}$. Now, let us prove this matrix can be made full rank, if one chooses correct pairs $(p, q)$.
	
	\subsubsection{Choosing caustics}
	
	We need to find $q_{1, odd} = 2r_1+1$, so that the matrix is full rank and contains only conditions for caustics with $r \le r_1$. We will denote a matrix with coefficients $\hat{A}_{p, r, k}$ as $\hat{A}^{r_1}$. The rows of this matrix will correspond to every functional $\hat{A}_{p, 2r+1}$ with $p/(2r+1) < 1/q_0$ and with $r \le r_1$. The columns will correspond to every odd cosine elliptic harmonic of deformation from $k = 1$ to $k = r_1$ inclusive. Denote $K_{r_1} = \ker \hat{A}^{r_1}$, and $\kappa_{r_1} = \dim K_{r_1}$. We need to prove that for some $r_1$ the matrix $\hat{A}^{r_1}$ is full rank, or that $K_{r_1} = 0$ or that $\kappa_{r_1} = 0$. If the rank is full, this would mean that we can choose the family of $r_1$ caustics, whose rows form a full rank square matrix. That would mean that this family of caustics nearly kills first $r_1$ harmonics. We will use it later, but right now let's prove that this $r_1$ exists.
	
	\begin{lemma}
		There exists $r_1$, such that $\hat{A}^{r_1}$ is full rank.
	\end{lemma}
	
	 We propose an algorithm of the construction:

	\begin{enumerate}
		\item Start with $r_1 = r_{0, odd} + 1$. Then, $\hat{A}^{r_1}$ has $1$ row for caustic $1/(2r_1+1)$ and $r_1$ columns in it, so $\kappa_{r_1} = r_1 - 1$. 
		\item Set $r_1 = r_1 + 1$, and consider the difference between $\hat{A}^{r_1}$ and $\hat{A}^{r_1-1}$. We have added a column for harmonic $2r_1+1$, and some rows (at least one: $\hat{A}_{1, 2r_1+1}$) for caustics $p/(2r_1+1)$. Since only these rows have non-zero elements in the new column due to \eqref{matrixel}, we have $\kappa_{r_1} \le \kappa_{r_1-1}$.
		\item If $\kappa_{r_1} = 0$, then $\hat{A}^{r_1}$ is complete, and we have finished the proof. 
		\item If $q = 2r_1+1$ is a prime number with some properties ($q_0$-good) and $\kappa_{r_1} = \kappa_{r_1-1}$, we prove that $\kappa_{r_1} = 0$. So, otherwise the rank should fall at least by one. So we should hit zero at some point.
	\end{enumerate}
	
	\subsubsection{Field introduction}
	Let us prove, that $\kappa_{r_1}$ would actually decrease for some $r_1$. We will prove it using algebraic field theory. Let's say $q_1 = 2r_1 + 1$, is a prime number. Presume $\kappa_{r_1}$ did not fall. Let $p_1, p_2$ be some numbers, such that $\frac{p_1}{q_1} < \frac{1}{q_0}, \frac{p_2}{q_1} < \frac{1}{q_0}$, unrelated to yet to be constructed sequence $p_i$ in Lemma \ref{mainfieldlemma}. Then, note that the angle $\frac{2\pi p_2}{q_1}$ is some multiple of the angle $\frac{2\pi p_1}{q_1}$ modulo $2\pi$. As such, we can express $\cos \frac{2\pi p_2}{q}$ through $\cos \frac{2\pi p_1}{q_1} $ via the formula for cosine of the natural multiple of an angle.
	Let 
	\begin{equation}
		\frac{2\pi p_2}{q_1} + 2\pi s = 
		\rho(p_1, p_2) \frac{2\pi p_1}{q_1}, \; 
		\rho(p_1, p_2) \in \mathbb{Z}, 1 \le \rho(p_1, p_2) \le q_1-1.
	\end{equation}
	One can also note that if we will consider $p_1, p_2$ as elements of $\mathbb{F}_{q_1}$, the following would be true: 
	\begin{equation}
		\rho(p_1, p_2) = p_2p_1^{-1}.
	\end{equation}
	The formula for $
	\rho(p_1, p_2)$-multiple cosine will always be a polynomial with rational coefficients. Precisely, let
	\begin{equation}
		\cos \frac{2\pi p_2}{q_1}  = P\left( \cos \frac{2\pi p_1}{q_1} \right) .
		\label{p2thrup1}
	\end{equation}
	Now, consider the matrix $\hat{A}^{r_1}(p_1, p_2)$, that is obtained from $\hat{A}^{r_1}$ by removing all rows for $\hat{A}_{p, q}$ with $q = q_1$ and $p \ne p_1, p_2$.
	Since it is a submatrix of $\hat{A}^{r_1}$, we have
	\begin{equation}
		K_{r_1} \subset \ker(\hat{A}^{r_1}(p_1, p_2)) .
	\end{equation}
	However, 
	\begin{equation}
		\dim \ker(\hat{A}^{r_1}(p_1, p_2)) \le \kappa_{r_1-1} = \kappa_{r_1} = \dim(K_{r_1}).
	\end{equation}
	So, 
	\begin{equation}
		K_{r_1} = \ker(\hat{A}^{r_1}(p_1, p_2)),
	\end{equation}
	\begin{equation}
		\dim \ker(\hat{A}^{r_1}(p_1, p_2)) = \kappa_{r_1-1}.
	\end{equation}
	That means that by adding 2 new rows and only 1 new column to $\hat{A}^{r_1-1}$, the rank of kernel did not fall. If we write down a condition on that, we will receive that some minors of the matrix of $\hat{A}^{r_1}(p_1, p_2)$ vanish.
	
	Let's describe our following steps. First, we will construct a field, containing some elements of the matrix of $\hat{A}^{r_1}(p_1, p_2)$. Out of all coefficients, only $\cos \frac{2\pi p_1}{q_1} $ will not be present in said field. Then, we will consider a ring of polynomials over the field, depending on some variable $z$. After that, we will substitute $z$ instead of $\cos \frac{2\pi p_1}{q_1}$ in the matrix and write down the described minors of the matrix. These minors will be polynomials of $z$, and will have a root at $z = \cos \frac{2\pi p_1}{q_1} $. Then, they will be divisible by the minimal polynomial of $\cos \frac{2\pi p_1}{q_1} $. We will use it to substitute other roots of the minimal polynomial instead of $z$.
	
	Let's us construct a field $F$. First of all, we will consider the field of rational numbers $\mathbb{Q}$. Next, let $W$ be the lowest common multiple of all the numbers, less than $q_1$. Then, let $w$ be the primitive root of unity of order $W$. Our 
	field F would be $\mathbb{Q}$ with added element $w$. Now, 
	\begin{equation}
		\left[ F : \mathbb{Q} \right] = \varphi(W),
	\end{equation}
	where $\varphi(W)$ is Euler's totient function. 
	
	Now let's discuss the element of such field. First of all, rational numbers are obviously present in this field. So, all the binomial coefficients are present. Also, all roots of unity of degree $W$ are present. Then, for every $q < q_1$, its roots of unity are present. This means that 
	\begin{equation}
		\cos\left( \frac{2\pi p}{q}\right)  + i\sin\left( \frac{2\pi p}{q}\right)  \in F, \; q<q_1.
	\end{equation}
	Since conjugate root is also present, we have that
	\begin{equation}
		\cos\left( \frac{2\pi p}{q}\right)  \in F, \; q<q_1.
	\end{equation}
	So, note that every row in $\hat{A}^{r_1}(p_1, p_2)$, except rows of $\hat{A}_{p_1, q_1}$ and $\hat{A}_{p_2, q_1}$, has all their elements present in $F$. Of those two, the elements of $\hat{A}_{p_1, q_1}$ row will have polynomial dependency on $\cos \frac{2\pi p_1}{q_1} $. For $\hat{A}_{p_2, q_1}$ this will also be true, after considering \eqref{p2thrup1}. Now let's write down the matrix of $\hat{A}^{r_1}(p_1, p_2)$:
	\begin{equation}
		\begin{pmatrix}
			f_{11} & f_{21} & \ldots & f_{\alpha1}\\
			f_{12} & f_{22} & \ldots & f_{\alpha2}\\
			\vdots & \vdots & \ddots & \vdots\\
			f_{1\beta} & f_{2\beta} & \ldots & f_{\alpha\beta}\\
			R_1\left( \cos\frac{2\pi p_1}{q_1} \right)  & R_2\left( \cos\frac{2\pi p_1}{q_1} \right)  & \ldots & R_{\alpha}\left( \cos\frac{2\pi p_1}{q_1} \right) \\
			R_1\left( P \left( \cos\frac{2\pi p_1}{q_1} \right) \right)  & R_2\left( P \left( \cos\frac{2\pi p_1}{q_1} \right) \right) & \ldots & R_{\alpha}\left( P \left( \cos\frac{2\pi p_1}{q_1} \right) \right)
		\end{pmatrix}.
		\label{bigmatrix}
	\end{equation}
	Here, $f_{ij}$ denote some elements of $F$, and $R_i$ -- some polynomials over $F$. The first $\beta$ rows represent functionals with $q<q_1$, the second-to-last represents $\hat{A}_{p_1, q_1}$, and the last one -- $\hat{A}_{p_2, q_1}$. Also note that $P$ is also a polynomial over $F$. Now, introduce new variable $z \in \mathbb{C}$, and put it into this matrix instead of $\cos \frac{2\pi p_1}{q_1} $:
	\begin{equation}
		\begin{pmatrix}
			f_{11} & f_{21} & \ldots & f_{\alpha1}\\
			f_{12} & f_{22} & \ldots & f_{\alpha2}\\
			\vdots & \vdots & \ddots & \vdots\\
			f_{1\beta} & f_{2\beta} & \ldots & f_{\alpha\beta}\\
			R_1\left(z \right)  & R_2\left( z \right)  & \ldots & R_{\alpha}\left( z \right) \\
			R_1\left( P \left( z \right) \right)  & R_2\left( P \left( z \right) \right) & \ldots & R_{\alpha}\left( P \left( z \right) \right)
		\end{pmatrix}.
	\end{equation}
	Now we know, that at $z = \cos \frac{2\pi p_1}{q_1} $ some minors of this matrix are zero. Since all the minors of this matrix are polynomials over $F$ from $z$, that means that these polynomials are divisible by the minimal polynomial $\tilde{\Psi}$ of $\cos \frac{2\pi p_1}{q_1}$ over $F$. We know, that the minimal polynomial of $\cos \frac{2\pi p_1}{q_1} $ over $\mathbb{Q}$ is $\Psi$. We also know, that 
	\begin{equation}
		\deg(\Psi) = \frac{\varphi(q_1)}{2} = \frac{q_1 - 1}{2}.
	\end{equation}
	The roots of $\Psi$ take the form 
	\begin{equation}
		\cos \frac{2\pi p}{q_1} , \; p = 1, \ldots ,\frac{q_1 - 1}{2}.
	\end{equation}
	Now we will prove, that $\tilde{\Psi} = \Psi$, meaning that by adding new elements to the field, we did not reduce the degree of the minimal polynomial.
	\begin{lemma}
		\begin{equation}
			\tilde{\Psi} = \Psi
		\end{equation}
		\label{lemapsi}
	\end{lemma}
	\begin{proof}
		Let's assume it is not true. Then 
		\begin{equation}
			d = \deg(\tilde{\Psi}) < \deg(\Psi) = \frac{q_1 - 1}{2}.
		\end{equation}
		Consider $F_1$ by adding $\cos\frac{2\pi p_1}{q_1}$ to $F$. 
		\begin{equation}
			\left[ F_1 : F \right] = d .
		\end{equation}
		Now, consider $F_2$ by adding $i\sin \frac{2\pi p_1}{q_1} $ to $F_1$ as a solution to $z^2 = \cos^2\frac{2\pi p_1}{q_1}- 1$, if it is not already there. Then, 
		\begin{equation}
			\left[ F_2 : F_1 \right] \le 2. 
		\end{equation}
		So, we have that 
		\begin{equation}
			\left[ F_2 : \mathbb{Q} \right] = 	\left[ F_2 : F_1 \right]	\left[ F_1 : F \right]	\left[ F : \mathbb{Q} \right] \le 2d\varphi(W).
		\end{equation}
		But since $\cos \frac{2\pi p_1}{q_1}  + i\sin\frac{2\pi p_1}{q_1}$ is present in $F_2$, all the other roots of unity of degree $q_1$ are present there. Since roots of unity of degree $q_1$ and $W$ are present in $F_2$, the roots of unity of degree $q_1W$ should be present there, since $q_1$ and $W$ are co-prime. Since the primitive roots of unity of degree $q_1W$ are present there, the expansion of $F_2$ over $\mathbb{Q}$ should at least have the degree of their minimal polynomial. Then,
		\begin{equation}
			\left[ F_2 : \mathbb{Q} \right] \ge \varphi(q_1W).
		\end{equation}
		So, 
		\begin{equation}
			(q_1-1)\varphi(W) = \varphi(qW) \le 2d\varphi(W).
		\end{equation}
		that immediately leads to contradiction. So, $\tilde{\Psi} = \Psi$.
	\end{proof}
	\subsubsection{Changing roots}
	So, all the described minors are divisible by $\Psi(z)$. Then, they have all the roots of $\Psi(z)$ as their roots. In particular, we can substitute $z = \cos \frac{2\pi}{q_1} $. This means, that the following matrix has the same dependencies, as the matrix of $\hat{A}^{r_1}(p_1, p_2)$:
	\begin{equation}
		\begin{pmatrix}
			f_{11} & f_{21} & \ldots & f_{\alpha1}\\
			f_{12} & f_{22} & \ldots & f_{\alpha2}\\
			\vdots & \vdots & \ddots & \vdots\\
			f_{1\beta} & f_{2\beta} & \ldots & f_{\alpha\beta}\\
			R_1\left( \cos \frac{2\pi}{q_1} \right)  & R_2\left( \cos \frac{2\pi}{q_1} \right)  & \ldots & R_{\alpha}\left( \cos \frac{2\pi}{q_1}\right)  \\
			R_1\left( P \left( \cos \frac{2\pi}{q_1}  \right) \right)  & R_2\left( P \left( \cos \frac{2\pi}{q_1}  \right) \right) & \ldots & R_{\alpha}\left( P \left( \cos \frac{2\pi }{q_1} \right) \right)
		\end{pmatrix}.
		\label{bigmatrix2}
	\end{equation}
	Now, we know that 
	\begin{equation}
		P \left( \cos\frac{2\pi }{q_1} \right) = \cos\left( \frac{2r(p_1, p_2)\pi }{q_1}\right) .
	\end{equation}
	So, \eqref{bigmatrix2} actually has the similar structure to $\hat{A}^{r_1}(p_1, p_2)$, bu instead of $\hat{A}_{p_1, q_1}$ and $\hat{A}_{p_2, q_1}$ functional, we get a functional $\hat{A}_{1, q_1}$ and a row for "functional" $\hat{A}_{\rho(p_1, p_2), q_1}$. Note that we didn't require to preserve $\rho(p_1, p_2)/q_1$ caustic. 
	
	It is natural to denote this matrix as $\hat{A}^{r_1}(1, \rho(p_1, p_2))$. Then,
	\begin{equation}
		\dim \ker (\hat{A}^{r_1}(p_1, p_2)) = \kappa_{r_1-1} \Rightarrow \dim \ker (\hat{A}^{r_1}(1, \rho(p_1,p_2))) = \kappa_{r_1-1}.
	\end{equation}
	This means that
	\begin{equation}
		\ker (\hat{A}^{r_1}(1, \rho(p_1,p_2))) = \ker (\hat{A}^{r_1}(1)) = \ker(\hat{A}^{r_1}) = K_{r_1}
	\end{equation}
	for a logical definition of $\hat{A}^{r_1}(1)$.
	
	Let's understand what we did here. We had two rows, that had $K_{r_1}$ in a kernel, namely $\hat{A}_{p_1, q_1}$ and $\hat{A}_{p_2, q_1}$.
	
	We did some operations and deduced, that a row for "functional" $\hat{A}_{\rho(p_1, p_2), q_1}$ also has $K_{r_1}$ in a kernel.
	
	Let us now generalize this process. Let $G$ be the set of all $p \in \mathbb{F}_{q_1}^*$, such that a row for "functional" $\hat{A}_{p,q_1}$ has $K_{r_1}$ in a kernel. We have proven that if $p_1 \in G, p_2 \in G$, then $\rho(p_1, p_2) = p_2p_1^{-1} \in G$. Note that since a functional $\hat{A}_{1,q_1}$ is available to us, that means that $1 \in G$. This immediately proves that $G$ is a subgroup of $\mathbb{F}_{q_1}^*$.
	
	Now note that if $\frac{p}{q_1} < \frac{1}{q_0}$, then $p \in G$(of course here $p$ is a natural number). Also note that $G$ is symmetrical by multiplying by $-1$, since the cosine is an even function.
	
	Now suppose, that for given $q_1$, these demands force $G$ to be equal to the whole group. We will discuss, for which primes this is true, later, but now notice that this condition depends only on prime number itself and on $q_0$.
	
	If $G$ is the whole group, we will show that $K_{r_1} = \left\lbrace 0 \right\rbrace $.
	\subsubsection{Finishing steps}
	Let's count the columns. We have $k = 1, \ldots, k = r_1$. We have $r_1$ of them. Now consider rows $\hat{A}_{p, q_1}$ for $p = 1, \ldots, r_1$, and the matrix consisting only of them. Since the system has $K_{r_1}$ inside its kernel, and if $K_{r_1}$ is not zero, then the determinant of its matrix is zero. We will show that it cannot be this way. Write down the equations more precisely (let's say first column corresponds to $k = r_1$, last -- to $k = 1$)
	\begin{equation}
		\begin{pmatrix}
			\left( \frac{\cos\frac{2\pi p}{q_1} + 1}{2} \right)^{r_1-1} \binom{2r_1}{0},  & \left( \frac{\cos\frac{2\pi p}{q_1} + 1}{2} \right)^{r_1 - 2} \binom{2r_1 - 1}{1},   & \ldots & \left( \frac{\cos\frac{2\pi p}{q_1} + 1}{2} \right)^{0} \binom{r_1}{r_1} 
		\end{pmatrix}.
	\end{equation}
	Now the binomial coefficients do not depend on $p$, so elements in the same column have the same coefficients. Since they are non-zero, we can multiply whole columns by their inverses and cancel them. Determinant will remain zero. Consider the rows of the new matrix:
	\begin{equation}
		\begin{pmatrix}
			\left( \frac{\cos\frac{2\pi p}{q_1} + 1}{2} \right)^{r_1-1},  & \left( \frac{\cos\frac{2\pi p}{q_1} + 1}{2} \right)^{r_1-2},   & \ldots & \left( \frac{\cos\frac{2\pi p}{q_1} + 1}{2} \right)^{0}
		\end{pmatrix}.
	\end{equation}
	Notice that this matrix is just the rotated Vandermonde matrix, and its determinant is nonzero, since 
	\begin{equation}
		\cos\frac{2\pi p_1}{q_1} \ne \cos\frac{2\pi p_2}{q_1}; \; p_1, p_2 = 1, 2, \ldots ,r_1; \; p_1 \ne p_2 .
	\end{equation}
	So, if $q_1$ is $q_0$-good, then the dimension of the kernel should fall at least by one. Since there are infinite number of those primes, the kernel of the system will become zero at some $r_1$. Note that $r_1$ depends only on $q_0$, and does not depend on $e$ or the deformation $\mu$.
	
	We note that $a^+$ and $a^{-}$ are studied the same way.
	
	\subsection{Selection of Primes}
	Previously, we described a following problem. Let $q$ be a prime, and let $G$ be the minimal subgroup of $\mathbb{F}_{q_1}^*$, that contains $1, \ldots, \left[ \frac{q_1}{q_0} \right] $ and is symmetrical under negation. For what $q_1$ $G = \mathbb{F}_{q_1}^*$? If some number from the starting ones is a primitive root modulo $q$, then the whole group is generated by its powers. Then, all the $q_0$-good numbers introduced in Definition \ref{q0agood} satisfy this relation. Since there are infinite amount of them, the rank should fall to zero.
	
	\subsection{Even nodes}
	
	\subsubsection{Introduction}
	The algorithm concerning even indices would be similar to the odd one. We will study similar matrices to the odd case. Then, it is possible to prove that the dimension of the kernel does not increase. Then one would consider caustics of rotation number $\frac{p}{2q_1}$ for some prime $q_1$ and odd $p$. Then, we will prove that the dimension of the kernel decreases, when $q_1$ is $q_0$-good.  
	
	\subsubsection{New rows}
	In this section it is important to understand a difference with an odd case. We said earlier that the rows of the matrix will correspond to the preserved caustic, like $A_{p, q, j}$ corresponds to the caustic $p/q$. Here, it will be important to consider rows with $A_{2p, 2q, j}$. It is a bit unnatural, but it is just another condition of preservation of $p/q$ caustic, because to preserve a caustic one needs to kill not only the $q$ harmonic, but also $2q$ harmonic and so on in the action angle coordinates. $A_{2p, 2q, j}$ corresponds to $2q$ harmonic for $p/q$ caustic, as seen in \eqref{idef}. The same formulas apply for them as well and we can use them when $(2p)/(2q)<1/q_0$, so these rows are extremely similar to normal ones.
	\subsubsection{Changing the matrix}
	We will make similar adjustments before using the field theory. We once again introduce new indexes, since we are studying an even case. We get $j = 2k$ and $2q = 2r$. We consider $k\ge 2$ and $r > r_{0, even} = \frac{q_0}{2}$.   
	
	We make the following change.
	
	\begin{equation}
		\tilde{A}_{p, r, k} =  \frac{2^{4r - 3k-1} \cos^{2r-2}\frac{\pi p}{2q}}{e^{2(r - k)}} A_{p, r, k}
		\label{scaleeven}
	\end{equation}
	
	Then, the following estimates hold:
	
	\begin{equation}
		\tilde{A}_{p, r, k} = \binom{r+k - 1}{r-k}   \left( \cos \frac{\pi p}{r} + 1 \right)^{k-1}     + O(e^{2}), \; \; e \rightarrow 0, \; k \le r
	\end{equation}
	\begin{equation}
		\tilde{A}_{p, r, k} = O(e^{2}), \; \; e \rightarrow 0, \; k > r
	\end{equation}
	
	Then, we get the introduce the limit as $e \rightarrow 0$.
	
	\begin{equation}
		\hat{A}_{p, r, k} = \binom{r+k - 1}{r-k}   \left( \cos \frac{\pi p}{r} + 1 \right)^{k-1},\; k \le r;  \; \; \; \hat{A}_{p, r, k} = 0,\; k > r.
		\label{matrixel2}
	\end{equation}
	
	Now we will consider the procedure, similar to the odd case. We start with $r_1 = r_{0, even}+1$ and we will increase it by $1$ step-by-step. We know that $\kappa_{r_1}$ would not increase (since we have $(1, 2r_1)$ condition). So we want to prove the matrix $\hat{A}^{r_1}$ to be full rank for some $r_1$. 
	
	Similarly to the odd case, let's prove, that if $r_1 = q$ is a prime number with some properties, then the rank falls at least by one.
	\subsubsection{Even case field theory}
	In the odd case we were only considering $2$ rows with $q = q_1$ each time and doing some field theory with it. In the even case, the situation is very similar, but a bit more complex. We will once again consider only $2$ rows for $q = 2q_1$ and writing down the minimal polynomials and changing roots. The major difference with the case of odd $q$ is that even for prime $q_1$ all the cosines do not share the same minimal polynomial. Specifically, cosines with even $p$ (coming from $(p/2)/{q_1}$ caustic) have the same polynomial, while cosines with odd $p$ (coming from $p/{2q_1}$), have another one. Because of that, our task breaks up into $2$ parts. The first is to get all the residues for odd $p$ using field theory, second - the same for even $p$. To succeed in both tasks, we would need $2$ conditions on $q_1$, so we would need to join them together. To accomplish them, we will be taking $p_1$ and $p_2$ of the same parity.
	
	Let us discuss the algebraic structure. For even $p$ we have the following cosines:
	
	\begin{equation}
		\cos\frac{2\pi s}{q_1},
	\end{equation}

	when $p = 2s$. These cosines are the same as the ones studied for the odd nodes and their minimal polynomial is $\Psi$. When $p_1$ and $p_2$ are both even, we can introduce $s_1$ and $s_2$ and consider them as elements of $\mathbb{F}_{q_1}^*$. Since other $q$ in the matrix are not divisible by $q_1$, we get that $\tilde{\Psi} = \Psi$ still, and we can do the same things as for odd nodes, specifically go from $s_1, s_2$ to $1, \rho(s_1, s_2)$. 
	
	We can then construct a subgroup $G$ here. It would also be symmetric around $0$. It will also have $s = 1, \ldots , s = \left[ \frac{q_1}{q_0}\right] $. So, to guarantee that this $G$ is the whole $\mathbb{F}_{q_1}^*$, we demand for $q_1$ to be $q_0$-good.
	
	Now we consider the case that $p$ is odd. First of all, we need to find a minimal polynomial for $\cos \frac{\pi p}{q_1}$ in this case. We can do a trick:
	
	\begin{equation}
		\cos \frac{\pi p}{q_1} = - \cos \frac{\pi (q_1+p)}{q_1}.
	\end{equation}
	
	Moreover, in this case $q_1+p$ is even, we naturally denote it as $2s$. From this we can see that if one removes the minus sign, the same $\Psi$ is the minimal polynomial again. Now we also consider $s \in \mathbb{F}_{q_1}^*$. We can try to do the same thing by going from $(s_1, s_2)$ to $(1, \rho(s_1, s_2))$, but the main problem is that caustic with $s = 1$ is not necessarily preserved (in fact it is $\frac{q_1-2}{2q_1}$ caustic), so the respective functional may not be available. However, when $p = 1$, $s$ is equal to $\frac{q_1+1}{2}$, and this should be preserved. Hence, we can change $1$ to $\frac{q_1+1}{2}$ in our proof and go
	
	\begin{equation}
		(s_1, s_2) \rightarrow \left(\frac{q_1+1}{2}, \rho(s_1, s_2)\frac{q_1+1}{2}  \right)
	\end{equation} 
	
	$G$ would still be symmetrical by negation. However, instead of $1, \ldots , \left[ \frac{q_1}{q_0} \right] $ inside of $G$ by default, we would get $\frac{q_1 - 1}{2} + 1, \frac{q_1 - 1}{2} + 2, \ldots , \frac{q_1 - 1}{2} + \left[ \frac{q_1}{q_0} \right] $, since we rotated everything by $\pi$.
	
	So, $G$ has similar structure to the subgroup, and actually becomes a subgroup, if one multiplies it by $2$. Then $G$ would include $1, 3, \ldots, 2\left[ \frac{q_1}{q_0} \right] - 1$. Since $q$ is already a $q_0$-good number, $3$, $5$ and $7$ are present in the starting set and one of them is a primitive root, so the subgroup is the whole group. This is also the reason we don't use $2$ instead of $7$ in a definition of $q_0$-good numbers.
	
	So, if $q$ is $q_0$-good, we will be able to add all the "functionals" $\hat{A}_{p,2q_1}$ for all $p$ from $1$ to $q_1-1$, both even and odd. We will now once again construct a Vandermonde matrix.
	
	Assume that the rank did not fall. This would mean that the matrix will all those $\hat{A}_{p,2q_1}$ still is not full rank. Let us only consider  $\hat{A}_{p,2q_1}$. Then, we have a matrix with $q_1-1$ rows and $q_1-1$ columns for $k = 2, \ldots , k = q_1$. We get a contradiction, since we once again have a Vandermonde matrix. So, when $q_1$ is $q_0$-good, we get a rank decrease. 
	
	\subsubsection{Some functionals are dependent}
	
	One may assume that the condition of being full rank is rather expected, since our matrix has a lot of rows. There are, however, some surprising connections between them. For example, one can find a deformation $\mu$ such that all the conditions $A_{p, q}$ for odd $p$ and $q \equiv 2\mod 4$ are zero in the main order in $e$ (meaning $\hat{A}_{p, q}$ are linearly dependent). So, some rows of $\hat{A}$ are linearly dependent on each other. The cosine harmonics of associated deformation are given by
	
	\begin{equation}
		a_{2k+4}^+ = (-1)^kC\frac{k+1}{k+2}\left(\frac{e}{2}\right)^{2k}.
	\end{equation} 
	This relation can be obtained using alternative formula for Chebyshev polynomials of first kind $T_k(z)$. However, this doesn't mean that $p/q$ caustics can be all preserved. There are other functionals for these caustics, like $A_{2p, 2q}$ that are not nullified by this deformation, and even the latter are only up to the linear order of $\mu$ and only in the main term over $e$.
	\section{Analytic dependency on the eccentricity}
	
	In the previous sections we studied the caustic preservation for small eccentricities. Now we study the case of non-small $e$. The main idea is to prove that the dependence of our objects on $e$ is holomorphic in some domain. That would help us to use the case of small eccentricities to obtain some information (rigidity) for other $e$.
	
	\subsection{Elliptic functions and related objects}
	
	In this section those objects will be the caustic parameters $\omega, \lambda, k, \phi$. Specifically, since we study caustic with a fixed rational rotation number, we want $\omega$ to be fixed and study $\lambda, k$ and $\phi$ as the functions from $e$ and $\omega$. Unfortunately, there seem to be no formula that derives $\lambda$ from $\omega$, but we can find $\omega$ as a function of $\lambda$:
	
	\begin{equation}
		\omega(\lambda, e) = \frac{F(\phi(\lambda, e),k(\lambda, e))}{2K(k(\lambda, e))}
		\label{omegale}
	\end{equation}

	So, the only way is to study $\lambda$ as an implicit function. This forces us to first study $\omega(\lambda, e)$ when they are both complex (in some thin neighborhood of the real line). We also want to bound the considered domain for real $e$ and $\lambda$ away from $2$ degenerate cases. The first is when $e=1$, while the second has $e^2+\lambda^2 = 1$, that corresponds to the family of orbits going through foci and to a segment "caustic" with $\omega = 1/2$.
	
	So, we fix a constant $e_{\max} < 1$ and $\omega_{\max} < 1/2$. We will not be considering ellipses with larger eccentricities (though we can always increase $e_{\max}$) and caustics with larger rotation numbers. We also should mention that for the dynamical result we only need $\omega_{\max} = 1/{q_0}$, but we may require larger $\omega$ in the spectral case to study non-incidence.
	
	To define studied complex domains, we introduce $2$ small parameters: $d$ and the width of a complex strip $\delta$. We demand the following:
	
	\begin{equation}
		\delta \ll d \ll 1 - e_{\max} \ll 1.
	\end{equation}
	
	\begin{figure}
		\includegraphics[width=10cm]{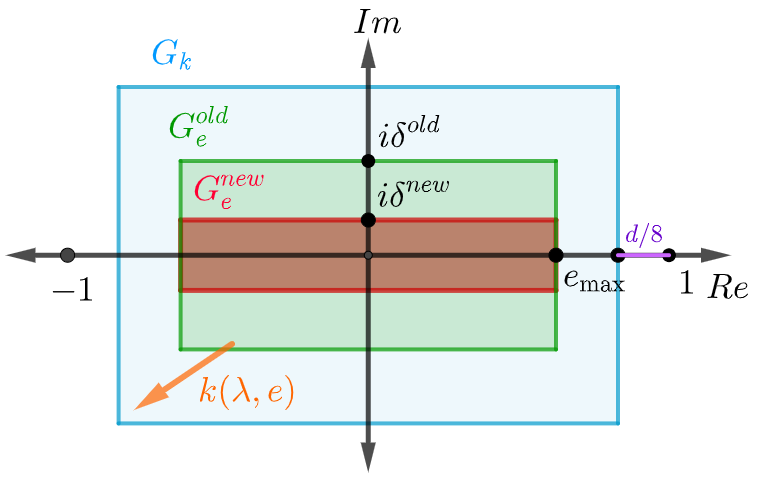}
		\centering
		\caption{Very narrow complex strips for eccentricities. Strip for $k$ is much closer to $1$ than that of $e$. When we use implicit function theorem, we decrease $\delta$.}
	\end{figure}
	
	The main strip we introduce is $G_e$ -- the domain of eccentricity:
	
	\begin{equation}
		G_e = \left\{ e \in \mathbb{C}: |\reale e| < 1 - e_{\max}, |\imag e| < \delta \right\}.
	\end{equation} 
	
	We introduce several a couple of auxiliary thin complex strips: $U$ -- the domain of $\lambda$ and $e$, where \eqref{omegale} is defined and $G_k$ -- the strip, containing the image of $U$ under $k(\lambda, e)$:
	
	\begin{equation}
		U = \left\{ (\lambda, e) \in \mathbb{C}^2 : e \in G_e, |\reale \lambda| < \sqrt{1 - (\reale e)^2 - d}, |\imag \lambda| < \delta\right\}
	\end{equation}

	and 
	
	\begin{equation}
		G_k = \left\{  k \in \mathbb{C}: |\reale k| < 1 - d/8, |\imag k| < C\delta \right\}
	\end{equation}

	for some constant $C$.
	
	\begin{figure}
		\includegraphics[width=9cm]{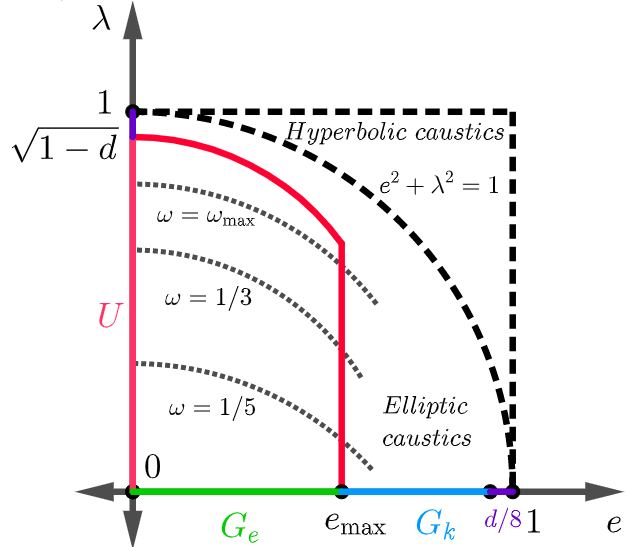}
		\centering
		\caption{Relation between $\lambda$ and $e$ in the first quadrant. $U$ is away from $e^2 + \lambda^2 = 1$ and $e = 1$. $\lambda(e, \omega)$ are drawn for various values of $\omega$.}
	\end{figure}
	
	Before we venture into the complex analysis, we should deal with real values $\lambda$ and $e$. On one hand, $\lambda$ should be bounded away from $\lambda^2 + e^2 = 1$, on the other it should realize all the needed values of $\omega$. The function $\omega(\lambda, e)$ is continuous and strictly increasing in $\lambda$, so we only need the following lemma:

	\begin{lemma}
		There exists a constant $d$ depending only on $e_{\max}$ and $\omega_{\max}$, such that
		
		\begin{equation}
			 \exists \lambda \in \left(0, \sqrt{1 - e^2 - d}\right): \omega(\lambda, e) > \omega_{\max}.
			\label{lambdaex}
		\end{equation}
		
		\label{lexi}
	\end{lemma}
	
	\begin{proof}
		
		We have the following relations:
		
		\begin{equation}
			\omega(\lambda, e) = \frac{F(\phi, k)}{2K(k)}, \; \;\; k = \frac{e}{\sqrt{1 - \lambda^2}}, \; \; \; \phi(\lambda, e) = \arcsin \left( \frac{\lambda}{\sqrt{1 - e^2}} \right)
			\label{omph}
		\end{equation}
		
		From \cite{abse} we know that
		
		\begin{equation}
			\sqrt{1-k^2}\tan \phi \tan \psi = 1 \Rightarrow F(\phi, k) + F(\psi, k) = K(k).
		\end{equation}
		
		In our case, we get that
		
		\begin{equation}
			\tan \psi = \frac{\sqrt{1 - \frac{\lambda^2}{1-e^2}}}{\sqrt{1 - \frac{e^2}{1 - \lambda^2}}\frac{\lambda}{\sqrt{1-e^2}}} \Rightarrow \cos \psi = \lambda
		\end{equation}
		
		Then, 
		\begin{align}
			\omega(\lambda, e) = \frac{1}{2} - \frac{F(\psi, k)}{2K(k)} \ge \frac{1}{2} - \frac{\psi}{2\sqrt{1-e^2} K(k)} \ge \frac{1}{2} - \frac{C\sqrt{1 - \lambda}}{K(k)}.
		\end{align}
		
		Now, we can say that $K(k)$ has a logarithmic singularity as $k \rightarrow 1$ (meaning $e^2+\lambda^2 \rightarrow 1$):
		
		\begin{equation}
			K(k) \ge C \left( \log\frac{1}{1 - k^2} + 1\right)  = C \left( \log\frac{1 - \lambda^2}{1 - \lambda^2 - e^2}+ 1\right) 
		\end{equation}
		
		Hence,
		
		\begin{equation}
			\omega(\lambda, e) \ge \frac{1}{2} - C\frac{\sqrt{1 - \lambda}}{\log(1 - \lambda^2) - \log(1 - \lambda^2 - e^2) + 1}
		\end{equation}
		
		The denominator is at least one, so when $\lambda > c(e_{\max}, \omega_{\max})$ we will get the bound $\omega(e, \lambda) > \omega_{\max}$ just by looking in the numerator, provided the second logarithm exists. If not, the first logarithm is bounded and the needed is true for $1 - \lambda^2 - e^2 < d(e_{\max}, \omega_{\max})$. The first case can be integrated inside the second one by decreasing $d$.
		
	\end{proof}
	
	Now we can say that the dependency of $\omega$ on $e$ and $\lambda$ is holomorphic in $U$, if $\delta$ is small enough.
	
	\begin{lemma}
		$\exists \delta > 0$, such that $\omega(\lambda, e)$, defined by \eqref{omegale} is a holomorphic function of $(\lambda, e) \in U$.
		
		\label{omegalemma}
		
	\end{lemma}
	
	To prove this lemma, we prove the analycity of all the simpler functions in \eqref{omegale} step by step. We start with the function $k(\lambda, e)$ and we choose $\sqrt{x}:\mathbb{C} \backslash [-\infty, 0] \rightarrow \mathbb{C}$ to be holomorphic.
	
	\begin{lemma}
		
		For small enough $\delta$ the function $k(\lambda, e)$ is holomorphic in $U$, mapping into $G_k$.
		
	\end{lemma}
	
	\begin{proof}
		
		The only noteworthy part is bounding the real part of $k$:
		
		\begin{equation}
			\begin{aligned}
				|\reale k| \le |k| = \frac{|e|}{|\sqrt{1 - \lambda^2}|} \le \frac{|e|}{\sqrt{1 - |\lambda|^2}} \le  \frac{|e|}{\sqrt{1 - (1 - (\reale e)^2 - d) - \delta^2}} \le \\
				\le  \frac{|e|}{\sqrt{|e|^2 + d/2}}
				 \le \frac{1}{\sqrt{1+d/2}} \le \sqrt{1 - d/4} < 1 - d/8
			\end{aligned}
		\end{equation}
	\end{proof}
	
	Now, we move on to $\phi(\lambda, e)$. It involves the inverse sine, do we specify that we study it on the following set: 
	
	\begin{equation}
		\arcsin(z):(-1, 1)\times \mathbb{R} \rightarrow\mathbb{C}.
	\end{equation}

	Then, the function $\phi(\lambda, e)$ is well defined on $U$, holomorphic and it maps into $|\reale \phi| < \frac{\pi}{2} - C\delta$, $|\imag \phi| <C\delta$. To prove the Lemma \ref{omegalemma}, we are only left with elliptic integrals of the first kind $F$ and $K$ with modulus $k$. We will also use these integrals for other $\varphi \ne \phi$, so we propose a general lemma:
	
	\begin{lemma}
		For any $d$, $\exists \delta> 0$, such that 
		
		\begin{equation}
			F(\varphi, k) : (\mathbb{R}\times (-iC\delta, iC\delta)) \times G_k \rightarrow \mathbb{C}
		\end{equation}
		
		is holomorphic in $(\varphi, k)$ and produces needed values for positives.
		Moreover, $\delta$ can be chosen to be small enough, so that
		
		\begin{equation}
			K(k) = F\left( \frac{\pi}{2}, k\right) 
		\end{equation}
		
		is a holomorphic function in $G_k$ and its real part is greater than $0$.
		\label{fklemma}
	\end{lemma}
	
	The proof of this lemma is pretty straightforward.

	We also note that the same happens with elliptic integrals of the second kind. We have proven Lemma \ref{omegalemma}.

	\subsection{Implicit function $\lambda$ of $\omega$}
	
	 Now we want to find the inverse function, since generally we know the needed rotation number and eccentricity and we have to express the parameter $\lambda$ to study the caustic and its related objects. We need to show that the function $\lambda(e, \omega)$ exists and is holomorphic.
	 
	 We will be decreasing $\delta$ to invert a function. But we want for $U$ to remain the same set.
	
	\begin{lemma}
		For every $e_{\max} < 1$ and $\omega_{\max} < 1/2$, there exists $\delta> 0$ so that the function $\lambda(e, \omega)$ is holomorphic on $e \in G_e$, $|\reale \omega| < \omega_{\max}$, $|\imag \omega| < \delta$ and
		
		\begin{equation}
		 (\lambda(e, \omega), e) \in U, \; \; \omega(\lambda(e, \omega), e) = \omega.
		\end{equation}
		
		Moreover, this function produces needed values for positives.
		\label{lambomeg}
	\end{lemma}
	
	We'll do this, using implicit function theorem. Precisely, consider
	
	\begin{equation}
		f(\lambda, e, \omega) = \omega(\lambda, e) - \omega,
	\end{equation}
	
	where $(\lambda, e) \in U$, $\omega \in \mathbb{C}$. Assume
	
	\begin{equation}
		f(\lambda_0, e_0, \omega_0) = 0
	\end{equation}
	
	for some $\lambda_0, e_0, \omega_0 \in \mathbb{R}$. Lets prove that $f_\lambda$ is not zero at this point: 
	
	\begin{lemma}
		Assume $\lambda_0, e_0, \omega_0 \in \mathbb{R}$, $(\lambda_0, e_0) \in U$ and $f(\lambda_0, e_0, \omega_0) = 0$. Then,
		
		\begin{equation}
			f_\lambda(\lambda_0, e_0, \omega_0) = \omega_\lambda(\lambda_0, e_0) > 0.
		\end{equation}
		
		\label{impllam}
		
	\end{lemma}
	
	\begin{proof}

		If $\lambda_0 = 0$, the only non-quadratic dependency on $\lambda$ in $f$ will be in $\phi(\lambda, e)$. Its not hard to show that this will generate a non-zero derivative, since $F$ and $\arcsin$ have non-zero derivatives at zero under our conditions. If $e_0 = 0$, then $\omega(\lambda, e) = \frac{\arcsin(\lambda)}{\pi}$ and it has non-zero derivative. Since $f$ is odd over $\lambda$ and even over $e$, we are only left with the case $\lambda_0>0$, $e_0>0$, that represents a standard ellipse. 
		
		Since in a standard ellipse this function is strictly increasing over $\lambda$, the derivative is non-negative. If we assume that the derivative is zero, it will lead to contradictions in action-angle coordinates for an ellipse. Specifically, since the rotation function $\alpha(I)$ has non-zero derivative for $I>0$, our degeneracy will mean that we map the strip of width $\varepsilon$ in arc length coordinates into the strip of width $\varepsilon^3$ in action angle coordinates, while preserving the area. It is impossible, so the derivative should be positive.    
		
	\end{proof}
	
	Then, in the neighborhood of $(\lambda_0, e_0, \omega_0)$ the implicit function exists. We want this neighborhood to lie in $U \times \mathbb{C}$. But first, when $e_0 \in \left( -e_{\max}, e_{\max} \right)$ and $\omega_0 \in \left( -\omega_{\max}, \omega_{\max} \right) $, we need to find a real $\lambda_0$ in our domain to apply the implicit function theorem. Note that since the derivative is positive, we can find at most one real $\lambda$ in the domain. Of course, we also need to prove such $\lambda$ even exists. That is why we have proven Lemma \ref{lexi}. It shows that for a fixed $e_0$ some $\lambda$ in our domain satisfies $\omega(\lambda, e_0) > \omega_0$. If we take $-\lambda$ it would be less than $\omega_0$. So, the needed $\lambda_0$ exists somewhere in $(-\lambda, \lambda)$ and it lies in our domain.

	So, we have proven that for $e_0 \in \left( -e_{\max}, e_{\max} \right)$ and $\omega_0 \in \left( -\omega_{\max}, \omega_{\max} \right) $ that there exists only one $\lambda_0 \in \mathbb{R}$, $  (\lambda_0, e_0)\in U$, $ \omega(\lambda_0, e_0) = \omega_0$. Hence, we can denote it by $\lambda(e_0, \omega_0)$. Using compactness arguments we unite these local implicit functions and prove Lemma \ref{lambomeg}.
	
	So, the function
	
	\begin{equation}
		k_\omega = k(e, \omega) = k(\lambda(e, \omega), e)
	\end{equation}
	
	is defined and analytical. Particularly, 
	
	\begin{equation}
		k_{p/q}(e) = k_{p/q}
	\end{equation}
	
	exist for $p/q \in (0, \omega_{\max})$ and are holomorphic in $G_e$.
	
	We can also make sure so that $G_e \subset G_k$ and when $e \in G_e$ all the elliptic integrals are defined and the properties in Lemma \ref{fklemma} apply for $e$.
	
	\subsection{Jacobi amplitude}
	
	Now we are only left with analysis of the Jacobi amplitude function for $e$ in $G_e$ and $\theta$ in $\mathbb{R}$.
	
	\begin{lemma}
		We can decrease $\delta$ in such a way, that
		
		\begin{equation}
			\varphi(\theta, e) = am\left( \frac{4K(e)}{2\pi}\theta, e\right)
		\end{equation}
		
		is a holomorphic function of $(\theta, e)$ when $e \in G_e$ and $|\imag \theta| < \delta$. Moreover, $\varphi(\theta, e)$ can be used as amplitude in Lemma \ref{fklemma}.
		\label{jacobi}
	\end{lemma}
	
	The proof is straightforward. 
	
	\subsection{Preservation conditions}
	Now we study the analycity of preservation conditions for various caustics. These conditions are functions on a boundary of an ellipse, to which the deformation must be orthogonal to. In the action-angle coordinates, these functions are just harmonics, but we study them in Lazutkin parametrization $\vartheta$. So, we also need to include the Jacobian for changing coordinates from $\vartheta$ to $\theta_{p/q}$ inside of them. We add them because we are studiying these functions as functionals on the space of deformations, see Proposition \ref{smallharm}.
	
	Now we combine our functions together to form a family of functions of $\vartheta$ that depend on $e$ as a parameter. They are essentially $A_{p, q}^+$ and $A_{p, q}^-$, but written in Lazutkin coordinates, instead of elliptic. Specifically, one can introduce
	
	\begin{equation}
		c_{p, q}(\vartheta) = \frac{K(e)}{K(k_{p/q})}\frac{\sqrt{1 - e^2\sin^2\varphi(\vartheta)}}{\sqrt{1 - k_{p/q}^2\sin^2\varphi(\vartheta)}}\cos\left( q \theta_{p/q}(\vartheta)\right) 
		\label{cdef}
	\end{equation}
	
	and
	
	\begin{equation}
		s_{p, q}(\vartheta) = \frac{K(e)}{K(k_{p/q})}\frac{\sqrt{1 - e^2\sin^2\varphi(\vartheta)}}{\sqrt{1 - k_{p/q}^2\sin^2\varphi(\vartheta)}}\sin\left( q \theta_{p/q}(\vartheta)\right) 
	\end{equation}
	
	for $|\imag \vartheta| < \delta$, $e \in G_e$, $0 < p/q < \omega_{\max}$. Here,
	
	\begin{equation}
		\varphi(\vartheta) = am\left( \frac{4K(e)}{2\pi}\vartheta, e\right), \;\; \theta_{p/q}(\vartheta) = \frac{2\pi}{4K(k_{p/q})}F(\varphi(\vartheta), k_{p/q}).
	\end{equation}
	
	For a fixed $p, q$ and $\vartheta$ these functions are defined and holomorphic due to the previous lemmas. 
	
	Note that in $c_{p,q}$ or $s_{p,q}$ the values $p$ and $q$ are not necessarily co-prime. In our further discussions (and in the even nodes section) it will be important to study functions of type $c_{2p, 2q}$ and $s_{2p,2q}$, as we do with $A_{p, q}^\pm$.  
	
	We also introduce $5$ elliptic functions:
	
	\begin{equation}
		h_i(\vartheta) = \frac{4K(e)}{2\pi} \sqrt{1 - e^2\sin^2\varphi(\vartheta)}\frac{e_i(\varphi(\vartheta))}{1-e^2\cos^2\varphi(\vartheta)}; \;\; i = 1, 2, 3, 4, 5,
		\label{hdef}
	\end{equation}
	where
	\begin{equation}
		e_1(\varphi) = 1,\; e_2(\varphi) = \cos\varphi,\; e_3(\varphi) = \sin\varphi,\; e_4(\varphi) = \cos2\varphi,\; e_5(\varphi) = \sin2\varphi.
	\end{equation}
	
	According to \cite{ks}, these correspond to elliptic motions (rotations, translations and homothety). There, they are defined in terms of elliptic coordinates $\varphi$, but we need to consider them in Lazutkin $\vartheta$ (since we are considering all the other functions in them). That means we have also added a Jacobian factor in front of it. 
	
	All of these functions are also holomorphic, when $e\in G_e$, $|\imag \vartheta| < \delta$. Lets summarize our main results of this part:
	
	\begin{lemma}
		For every $e_{\max} > 0$, there exists $\delta > 0$ and previously defined strip $G_e$, so that the functions 
		
		\begin{equation}
			h_j(\vartheta), \; c_{p,q}(\vartheta), \; s_{p,q}(\vartheta)
		\end{equation}
		
		are holomorphic for $j = 1, 2, 3, 4, 5$ and $0 < \frac{p}{q}<\omega_{\max}$. Moreover, as a direct consequence functions like
		
		\begin{equation}
			\int_{0}^{2\pi}c_{p,q}(\vartheta)\cos(j\vartheta)d\vartheta, \; j \in \mathbb{Z}
		\end{equation}
		
		are holomorphic for $e \in G_e$ (we can change $c$ for $s$ or $h$ and $\cos$ for $\sin$).
		\label{mainanlycity}
	\end{lemma}

	We research the caustic rigidity of an ellipse with eccentricity $e$, not necessarily close to $0$. We will use the ideas from \cite{ks}, where the main objective is to construct a system of functions, each corresponding to the preservation of a caustic or elliptic motions. The goal is then to prove these functions span the whole deformation space, so then the caustic rigidity would follow.

	\subsection{Lengths of periodic orbits in an ellipse}
	
	Now we want to study the lengths of periodic orbits inside some ellipse. This will be used to prove the spectral rigidity at the end, not the dynamical one. Specifically, there is Definition \ref{noninc} of non-incidence condition for an ellipse. If an ellipse doesn't satisfy this condition, then our proof wouldn't work for this ellipse. 
	
	So, we just want to proof that incidence is a rather rare phenomenon. We prove that incidence cannot happen for an open interval or a dense set of $e$. In order to do that, we study types of periodic billiard orbits in the ellipse, prove their lengths to be holomorphic in $e$, using \eqref{lengthell} and \eqref{lenghyper}.
	
	The lengths of bouncing ball orbits are clearly holomorphic. For those tangent to the ellipse, \eqref{lengthell} gives an analytic function. The only problem is division by $k$, but $k$ just has a simple root at $e = 0$. 
	
	For orbits, tangent to hyperbolae we can similarly develop analytic theory and apply the same methods to these periodic points as we did to the regular orbits. One can obtain the following lemma:
	
	\begin{lemma}
		For each $e_{\max} < 1$, $\varepsilon > 0$ and $0 < \tilde{\omega} < 1/2$ there exists $\delta > 0$, so that for
		\begin{equation}
			\tilde{G}_e = \left\lbrace e \Big| \arccos \tilde{\omega} \pi + \varepsilon < |\reale e| < e_{\max}, |\imag e| < \delta \right\rbrace,
		\end{equation}
		the function $\lambda(e, \tilde{\omega})$ is holomorphic on $\tilde{G}_e$ and produces needed values for positive $e$.
	\end{lemma}

	So, their length \eqref{lenghyper} is also analytic in $\tilde{G}_e$
		
	\section{Holomorphic preservation operator study}
	
	\subsection{Estimates for small rotation numbers}
	
	Now we want to achieve some bounds for the objects we introduced in the previous section, when the rotation number $\omega$ is small. Primarily, we are interested in the case $p = 1$ and $q \ge 3$. For these numbers $\omega <\omega_{\max}$, so all the objects introduced previously are defined. We start by bounding $\lambda(e, 1/q)$.
	
	\begin{lemma}
		There exists $C(e_{\max}) > 0$, so that $\forall e \in G_e$, $|\imag \vartheta| < \delta$ and $q \ge 3$:
		\begin{equation}
			|\lambda_{1/q}| < \frac{C}{q}; \; \; |k_{1/q} - e| < \frac{C}{q^2}; \; \; |\theta_{1/q}(\vartheta) - \vartheta| < \frac{C}{q^2}.
		\end{equation}
	\end{lemma}
	\begin{proof}
		
		We begin with  the formula \eqref{omph}:
		\begin{equation}
			\frac{2}{q}\int_{0}^{\pi/2}\frac{d\tau}{\sqrt{1 - k_{1/q}^2\sin^2\tau}} = \int_{0}^{\arcsin \lambda_{1/q}/\sqrt{1-e^2}}\frac{d\tau}{\sqrt{1 - k_{1/q}^2\sin^2\tau}}
		\end{equation}

		$k_{1/q} \in G_k$ for $e\in G_e$, so integral on the left is bounded from above by some constant $C$. In the integral on the right, we can assume that we integrate along a complex interval, so we can perform a change of variables $\tau \rightarrow \gamma \tau$ to intagrate on reals:
		
		\begin{equation}
			\int_{0}^{|\arcsin \lambda_{1/q}/\sqrt{1-e^2}|}  \frac{d\tau}{\reale\sqrt{1 - k_{1/q}^2\sin^2(\gamma\tau)}} < \frac{C}{q}.
		\end{equation} 
		
		Since the function under integral is positive and bounded from zero, we get
		
		\begin{equation}
			|\arcsin \lambda_{1/q}/\sqrt{1-e^2}| < \frac{C}{q}.
		\end{equation} 
		
		Since $\arcsin$ has no other roots in our domain, the argument should be in the neighborhood of $0$. Applying Taylor approximation there and bounding the denominator away from $0$, we get
		
		\begin{equation}
			|\lambda_{1/q}| < \frac{C}{q}.
		\end{equation}
		
		The second assertion follows from the definition of $k$. The third assertion follows from the same arguments as Lemma 48 of \cite{ks}.
		
	\end{proof}
	
	Now we also introduce a lemma, that was inspired by Lemma 50 of \cite{ks} and has essentially the same proof, however one should account for complex eccentricity:
	
	\begin{lemma}
		For $q\in \mathbb{Z}_+$ and $j \ge 3$ the following are true:
		
		\begin{equation}
			\left| \left\langle \cos(q\vartheta),  c_{1,j}(\vartheta)\right\rangle_{L^2} - \pi \delta_{q,j} \right| \le C_\varepsilon j^{-1}\exp(-\delta|q-j|),
		\end{equation}
		
		where $\delta_{q,j}$ is a Dirac's delta. One can also change $\cos$ for $\sin$ and $c$ for $s$, removing $\delta_{q,j}$ in two cases.
		
		\label{scalarsmall}
	\end{lemma}
	
	The proof is given in \cite{ks}, however one should keep in mind that our functions are analytical over $\vartheta$ on a strip with width $\delta$, so $\rho = \delta$. The proof depends heavily on a result in previous lemma. Our estimates do not depend on the eccentricity since neither it did in the previous lemma.

	\subsection{Operator definition}
	
	Now we proceed to prove the main lemma. Let's say we have fixed some $q_0$ and $e_{\max}$. We know that in order to prove ellipse preservation for small eccentricities, we the condition that some of the caustics with rotation numbers smaller than $1/q_0$ are preserved. Note that out of these caustics, there were only a finite amount of ones with $p > 1$, since will be using only $p = 1$ for $q > q_1$.
	
	We also know, similarly to \cite{ks}, that if a deformation $\mu(\vartheta)$ preserves $p/q$ caustic, then
	
	\begin{equation}
		\left\langle \mu(\vartheta), c_{p,q}(\vartheta) \right\rangle_{L^2} = \left\langle \mu(\vartheta), s_{p,q}(\vartheta) \right\rangle_{L^2} = 0.  
	\end{equation}
	
	We will discuss this relation later, but now the following question arises: When do $s$ and $c$ for all of our caustics form a basis in $L^2$? If they form a basis, then the deformation cannot be orthogonal to all of them, so it would be trivial. 
	
	Of course, there always will be elliptic transformations, like translations and rotations, and these would always be valid. So, in order to adjust for this situation, we need to add $5$ functions $h_j$ into consideration. We propose a following result:
	
	\begin{lemma}
		For every $e_{\max}<1$ and $\omega_{\max} < 1/2$, there exists $\delta > 0$, so that the following holds. Let $\left\lbrace f_j(\vartheta) \right\rbrace $ be a sequence of functions depending on $e \in G_e$ as a parameter and defined as following:
		
		\begin{equation}
			f_j(\vartheta) = h_j(\vartheta), \; j=1,2,3,4,5,
		\end{equation}
		
		for every $j \ge 3, \; \exists p_j/q_j < \omega_{\max}$ so that
		
		\begin{equation}
			f_{2j} = c_{p_j, q_j}; \; \; f_{2j+1} = s_{p_j, q_j},
		\end{equation}
		
		and $p_j = 1$, $q_j = j$ for large enough $j$. Then, either $\left\lbrace f_j(\theta) \right\rbrace $ do not form a basis for only a finite amount of $e \in G_e$, or they are not a basis for all $e \in G_e$. 
		\label{allornoth}
	\end{lemma}
	
	We note that a second option could often be proven impossible for $e$ close to zero, so it can be easy to prove false. For example, when $e = 0$, the domain is a disc and all the functions $c$, $s$ and $h$ trivialize.
	
	We also note that previously we used a bit different definition of $G_e$. However, we can always re-scale it by dividing $\varepsilon$ by some constant.
	
	The main idea will be to use the result in \cite{kato} about the analycity of compact operators. First, we will construct a family of operators, depending on $e$, prove that they are analytical and compact. Then, we will study the behavior of eigenvalues of these operators for different $e$. 
	
	Denote $\left\lbrace e_j(\vartheta) \right\rbrace $ a regular orthonormal basis in $L^2[0, 2\pi]$, consisting of sines and cosines. Now, we will introduce the system of bounded linear operators $L_e$ by giving their action on the basis vectors $\left\lbrace e_j \right\rbrace $
	
	\begin{equation}
	L_e: L^2[0, 2\pi] \rightarrow L^2[0, 2\pi]: \;\;\;	L_e(x) =\sum_j (e_j - \frac{1}{\sqrt{\pi}}f_j)\left\langle x, e_j \right\rangle 
	\end{equation}
	
	\subsection{Proving operator to be compact and holomorphic}
	
	\begin{lemma}
		$L_e$ are all Hilbert - Schmidt operators, and
		\begin{equation}
			||L_e||_{HS} < C(e_{\max}), \; e \in G_e,
		\end{equation}
		
		where $C(e_{\max})$ depends only on $e_{\max}$ and the choice of $\left\lbrace f_j \right\rbrace $. Particularly, they all are uniformly bounded and compact operators.
		\label{hilbschm}
	\end{lemma}

	\begin{proof}
		First of all, we can choose $\delta$, such that all the functions $h$, $c$ and $s$ are holomorphic and bounded (not necessarily uniformly) on $e \in G_e$ and $|\imag \vartheta| < \delta$ for $p/q < \omega_{\max}$, due to Lemma \ref{mainanlycity}. Then, they are elements of $L^2$ and the operators are correctly defined on the basis functions. Now, we only need to bound the Hilbert-Schmidt norm, that being
		
		\begin{equation}
			||L_e||_{HS} = \sum_j ||L_e e_j||^2 = \sum_j||e_j - \frac{1}{\sqrt{\pi}}f_j||^2.
		\end{equation}
		
		Since all of the elements in this series are bounded independently over $e$, we will not consider a finite amount of small $j$. For large $j$, all the $f_j$ would be of the type $1/q$ for $q \ge 3$. So, the rest of our series would have the following structure:
		
		\begin{equation}
			\frac{1}{\pi}\sum_{j \ge j_0} \left( ||c_{1, j}(\vartheta) - \cos j\vartheta||^2 + ||s_{1, j}(\vartheta) - \sin j\vartheta||^2 \right) 
			\label{feclose}
		\end{equation} 
		
		Using Parseval's identity, we get:
		
		\begin{equation}
			\frac{1}{\pi}\sum_{j \ge j_0}\sum_{q = 1}^{\infty} \left(\left|\left\langle c_{1, j}(\vartheta) - \cos j\vartheta, e_q \right\rangle  \right|^2 + \left|\left\langle s_{1, j}(\vartheta) - \sin j\vartheta, e_q \right\rangle  \right|^2  \right).
		\end{equation}
		
		Since $e_q$ can either represent a cosine or a sine, we get a sum of $4$ terms in each element of the series. For simplicity lets consider only the terms, where cosines are multiplied by cosines. After removing constants, arising from norms (we should be careful with $e_1$, since it has a different norm, but we are achieving an upper bound), we obtain:
		\begin{equation}
			\sum_{j \ge j_0}\sum_{q = 0}^{\infty} \left|\left\langle c_{1, j}(\vartheta) - \cos j\vartheta, \cos q\vartheta \right\rangle  \right|^2 .
		\end{equation}
		
		Now we use Lemma \ref{scalarsmall} to receive an upper bound:
		
		\begin{equation}
			\sum_{j \ge j_0}\sum_{q = 0}^{\infty} C_\varepsilon^2j^{-2}\exp(-2\delta|q - j|) \le C\sum_{j \ge j_0}j^{-2}.
		\end{equation}
		
		The latter sum is bounded, so we have proven this lemma.
		
	\end{proof}
	
	Now lets introduce the concept of holomorphic family of operators. These operators share several properties with the holomorphic functions. There are several possible definitions of holomorphic family, but we include this one from \cite{kato} on the page $365$, dealing with bounded operators:
	
	\begin{lemma}
		A bounded system of operators $T(\varkappa):X\rightarrow Y$ is holomorphic for $\varkappa \in D$ if and only if it is bounded in some neighborhood of $\varkappa$ and $(Tu, g)$ is holomorphic in $D$ for every $u$ in a fundamental subset of $X$ and every $g$ in a fundamental subset of $Y^*$.
	\end{lemma}
	
	\begin{lemma}
		The family $\left\lbrace L_e\right\rbrace $ is holomorphic when $e \in G_e$ for $X = Y = L^2[0, 2\pi]$.
	\end{lemma}
	
	\begin{proof}
		In our case, $Y^* = L^2$, since we have a Hilbert space. We can pick any fundamental subset (with dense linear combinations), so we set both of these sets to be ${e_j}$. We have already proven the operators to be uniformly bound in Lemma \ref{hilbschm}, so we only need to check the second condition.
		
		However, we have that
		
		\begin{equation}
			\left\langle L_e e_i, e_j\right\rangle = \left\langle e_i - \frac{1}{\sqrt{\pi}}f_i, e_j\right\rangle = \delta_{i,j} - \frac{1}{\sqrt{\pi}}\int_{0}^{2\pi}f_i(\vartheta)e_j(\vartheta)d\vartheta,
		\end{equation}
		
		and this is analytical due to Lemma \ref{mainanlycity}.
		
	\end{proof}
	
	Now we have proven the system to be holomorphic, however we also know that all the operators are compact. For compact operators, \cite{kato} has the following theorem (Theorem VII $1.9$)

	\begin{oldtheorem}
		Let $T(\varkappa)$ be a family of compact operators in $X$ holomorphic for $\varkappa \in D_0$. Call $\varkappa$ a singular point, if $1$ is an eigenvalue of $T(\varkappa)$. Then, either all $\varkappa \in D_0$ are singular points or there are only a finite number of singular points in each compact subset of $D_0$ .
	\end{oldtheorem}
	
	We have already proven all the prerequisites in the previous couple lemmata, so now we can use the results. Then, we know when $1$ is an eigenvalue of $L_e$. We note that the result holds in each compact subset, but we can just increase $\varepsilon$ a little to say that there are either a finite amount of eigenvalues in the whole $G_e$, or all the points are eigenvalues. 
	
	What does it mean to have one as an eigenvalue? Since the all the operators are compact, their spectrum consists of eigenvalues, so we have that $L_e - I$ can be inverted. However, this operator maps maps $e_j$ into $f_j$ times a constant -$\frac{1}{\sqrt{\pi}}$. So, if $1$ is not an eigenvalue, then $\left\lbrace f_j\right\rbrace $ form a basis of $L^2[0, 2\pi]$, and vise versa. We have proven the lemma. We got the result for a compact subset, but we can always increase $e_{\max}$, doing the same idea as earlier. 
	
	Next, we need to show that the first option cannot happen. The first option would say that the system $\left\lbrace f_j\right\rbrace $ is not a basis for small eccentricities. Then we will show that it contradicts our result in these ellipses. We will be considering positive real eccentricities from now on.

	This subsection discussed the notion of operator $L_e$ for a general selection of $(p_i, q_i)$, but in order to move forward we need to get back to our original set of conditions, since we will be using the results from the first part of the paper. Hence, from this point forward, the family $(p_i, q_i)$ in Lemma \ref{allornoth} consists of the family defined in Lemma \ref{mainfieldlemma}, with all the $(1, q)$ pairs added for $q > q_1$. For other families one can prove the second option to hold in other ways. We also introduce a family $\mathcal F$ as all the caustics that give us the functionals $A_{p_i, q_i}$. 
	
	\subsection{Deformed Fourier nodes}
	
	To study operators for small eccentricities we need to define the deformed Fourier nodes. The reason is very simple. We have established in previous sections that in order to study the difference between $2$ caustics with the same $q$, one need to use formulas like \eqref{idef}. And since those formulas are better in elliptic coordinates, we need to use elliptic harmonics for small $j$. However, we know that for large $j$ elliptic harmonics and action-angle harmonics are not that close to each other, unlike action-angle and Lazutkin. So, we should use Lazutkin harmonics for large $j$, if we want to maintain compactness and other qualities of our operator.
	
	So, we introduce the following family of harmonics. If the frequency of the harmonic $e_i$ is in $[3, q_1]$, then 
	
	\begin{equation}
		d_i(\vartheta) =  e_i(\varphi(\vartheta)).
	\end{equation}

	When the frequency is more than $q_1$, then
	
	\begin{equation}
		d_i(\vartheta) = e_i(\vartheta).
	\end{equation}
	
	If the frequency of the harmonic is $2$ or less, then we just have an elliptic motion and
	
	\begin{equation}
		d_i(\vartheta) = \frac{e_i(\varphi(\vartheta))}{1-e^2\cos^2\varphi(\vartheta)}
	\end{equation}
	
	\begin{lemma} 
		The system $\left\lbrace d_i(\vartheta) \right\rbrace_{i\ge 0} $ forms a not necessarily orthogonal basis in $L_\vartheta^2[0, 2\pi]$ for $0 < e < e_0$ for some small $e_0$.
		\label{lema77}
	\end{lemma}

	\begin{proof}
		Define an operator $D_e$, such that
		\begin{equation}
			D_e(e_i) = d_i.
		\end{equation}
	We want to prove that $D_e$ is a bounded invertible operator in $L^2$. Let's estimate the following norm:
	
	\begin{equation}
		||I - D_e|| = \sup_{||x||=1}||(I-D_e)(x)|| \le \sup_{||x||=1}\sum_{i=0}^{2q_1+1}||e_i-d_i||x_i\le \left( \sum_{i=0}^{2q_1+1}||e_i-d_i||^2\right)^{1/2} \rightarrow 0, e \rightarrow 0. 
	\end{equation}
	So, the operator is bounded and for small enough $e$ it will be also invertible, since the norm above would be less than $1$.
	
	This would mean that $d_i$ is a basis of this Hilbert space.
	\end{proof}
	
	\subsection{Operator non-degeneracy for small eccentricities}
	
	\begin{lemma}
		The number $1$ is not an eigenvalue of the operator $L_e$, when $0 < e < e_0$ for some small $e_0$.
	\end{lemma}
	
	\begin{proof}
		Now we can choose $\omega_{max} = 1/{q_0}$. We will be considering operators $L_e^*$, since they are easier to study. We can see that they take the following form:
		
		\begin{equation}
			L_e^*(x) = I - \sum_n \frac{1}{\sqrt{\pi}}e_n\left\langle x, f_n \right\rangle 
			\label{conjl}
		\end{equation} 
		
		Note that this is defined, meaning the series converges, due to Lemma \ref{scalarsmall}.
		
		Assume that $L_e$ have $1$ as an eigenvalue. Then $L_e^*$ are also compact and have $1$ as an eigenvalue. We will prove that cannot happen.
		
		Let's consider a pair of scaling operators, designed to resemble scaling in \eqref{scaleodd} and \eqref{scaleeven}. $S^l$ will be playing the role of coefficient of $r$, while $S^r$ will be playing a role of coefficient of $k$.
		
		Specifically, we define $S^l$ and $S^r$ on $e_j$.
		
		\begin{equation}
			S^le_j=S^re_j = e_j, \; \; \; j > 2q_1+1
		\end{equation}
	
		For $6 \le j \le 2q_1+1$, we define $p, q$ and $r$ for $S^l$ and $k$ for $S^r$ the same way as in \eqref{scaleodd} and \eqref{scaleeven}. Then,
		
		\begin{equation}
			S^le_j = 2^{4r-1}\cos^{2r-2}(\pi\omega) e^{ q_1-\left[\frac{j}{2} \right] }e_j; \; \; S^re_j = \frac{2^{-3k}}{e^{q_1-\left[\frac{j}{2} \right] }}e_j.
		\end{equation}
	
		For the case $j \le 5$ we have:
		
		\begin{equation}
			S^le_j = e^{q_1-3}e_j	\; \; S^re_j = \frac{1}{e^{q_1-3}}e_j; .
		\end{equation}
	
		Then, these operators are now defined for $e = 0$, but for positive $e$ they are well defined. Moreover, $S^l-I$ and $S^r - I$ are finite-dimensional and compact. 
		
		Now, let's construct an operator
		
		\begin{equation}
			M_e = S^l\left(L_e^* - I\right)D_eS^r + I
		\end{equation}
	
		Note that $M_e$ is a compact operator, since the product consists of operators that are identity matrix plus a compact operator. Next, we will study the coefficients 
		\begin{equation}
			\left\langle \left( M_e - I  \right)   e_i, e_n \right\rangle = S^r_i\left\langle S^l\left(L_e^* - I  \right) d_i, e_n \right\rangle =  S^r_i\left\langle  d_i, \left(L_e - I  \right)S^l e_n \right\rangle = -\frac{1}{\sqrt{\pi}} S^r_i S^l_n \left\langle d_i, f_n\right\rangle . 
		\end{equation}
		Note that there are $9$ different cases. The index $n$ can be from $1$ to $5$ (elliptic perturbations case), from $6$ to $2q_1+1$ (small harmonics case) and from $2q_1+2$ onward (large harmonics case). The same can be said about index $i$. Let us study them one-by-one.
		
		\begin{enumerate}
			\item Case $A$. Both $i$ and $n$ are small harmonics. Then $f_n$ corresponds to some $p/q$ caustic, while $d_i$ corresponds to the frequency $j = \left[ \frac{i}{2}\right] $. We also presume that $f_n$ and $d_i$, and $q$ and $j$ share parity, otherwise the result would be just zero.
			\begin{equation}
				\frac{1}{\sqrt{\pi}} \left\langle d_i, f_n \right\rangle = \frac{1}{\pi}\int_{0}^{2\pi}\cos(j\varphi(\vartheta)) \frac{K(e)}{K(k_{p/q})}\frac{\sqrt{1 - e^2\sin^2\varphi(\vartheta)}}{\sqrt{1 - k_{p/q}^2\sin^2\varphi(\vartheta)}}\cos\left( q \theta_{p/q}(\vartheta)\right) d\vartheta
			\end{equation}
			This fraction with the square roots and complete elliptic integrals is just a Jacobian for the change from $\vartheta$ to $\theta_{p/q}$ (through $\varphi$). Hence, we have
			\begin{equation}
				\frac{1}{\sqrt{\pi}} \left\langle d_i, f_n \right\rangle = \frac{1}{\pi}\int_{0}^{2\pi}\cos(j\varphi(\theta_{p/q}))\cos(q\theta_{p/q}) d\theta_{p/q} = A_{p, q, j}
			\end{equation}
			due to \eqref{idef}. We have already studied these coefficients, we know their behavior when $e\rightarrow 0$ as well as some bounds on them. So,
			
			\begin{equation}
				\frac{1}{\sqrt{\pi}} S^r_i S^l_n \left\langle d_i, f_n\right\rangle =  S^r_i S^l_n A_{p, q, j} = \tilde{A}_{p,q,j}
			\end{equation}
			
			due to \eqref{scaleodd} and \eqref{scaleeven}. It is natural to denote this finite square matrix (or an operator) as $\tilde{A}$. We know that $\tilde{A}$ is invertible for small enough $e$.
			
			\item Case $B$. Here, $i$ is still small, but $m$ is large. Then, the same ideas hold, as in previous case, but now $p = 1$, $q = \left[ \frac{m}{2}\right] $, $q > q_1$.
			\begin{equation}
				\frac{1}{\sqrt{\pi}}S^r_i S^l_n \left\langle d_i, f_n \right\rangle = \tilde{A}_{1, q, j}.
			\end{equation}
			We will denote this as $\tilde{B}$, same estimates still hold.
			\item Cases $C$ and $D$. $i$ is now large, $m$ is small ($C$) or large ($D$). The difference is that now $d_i$ is a harmonic in Lazutkin coordinates, meaning we lack these estimates now. However, we can still use formulas obtain in this section, like Lemma \ref{scalarsmall}, since:
			\begin{equation}
					\frac{1}{\sqrt{\pi}} \left\langle d_i, f_n \right\rangle = \frac{1}{\sqrt{\pi}} \left\langle e_i, f_n \right\rangle
			\end{equation} 
			\item Cases where $i$ is elliptic, and $m$ is either small or large. Then we have that elliptic perturbation preserves caustics, so
			\begin{equation}
				\frac{1}{\sqrt{\pi}} \left\langle d_i, f_n \right\rangle = 0
			\end{equation}
			\item Cases where $m$ is elliptic. We will denote operators $H^1, H^2, H^3$, depending on $i$. 
		\end{enumerate}

		Then, the operator $M_e - I$ can be expressed in a following form.
		
		\begin{equation}
			M_e - I = -\begin{pmatrix}
				H^1 & H^2 & H^3 \\ 0 & \tilde{A} & C \\ 0 & \tilde{B} & D
			\end{pmatrix}
		\end{equation} 
		
		It is of course defined for real $e > 0$. However, one can also define the "limit" of these operators as $e\rightarrow 0$. This is possible since the scaling operators $S^l$ and $S^r$ were introduced. Specifically, let
		
		\begin{equation}
			M_0 - I = -\begin{pmatrix}
				I & 0 & 0 \\ 0 & \hat{A} & 0 \\ 0 & 0 & I
			\end{pmatrix},
		\end{equation}
		
		where the operator $\hat{A}$ consists of elements, defined in \eqref{matrixel} and \eqref{matrixel2} as limits of $\tilde{A}$. Note that the operator $M_0 - I$ is invertible, since the matrix $\hat{A}$ is non-degenerate due to the choice of $q_1$ and the whole discussion in previous sections with algebraic field theory. This is another reason for scaling - otherwise we would just get a matrix with only zeros in several columns.
		
		Now assume $M_e$ has an eigenvalue $1$ with eigenvector $x$. Then:
		
		\begin{align}
			M_ex = x \Rightarrow (M_e - M_0)x = (I - M_0)x \Rightarrow ||M_e - M_0|| \ge ||(I-M_0)^{-1}||^{-1} = \min\left(1, ||\hat{A}^{-1}||^{-1} \right).
		\end{align}
		
		So, the norm of $M_e - M_0$ is bounded from below by some constant, independent of $e$. Then, we also have the following.
		
		\begin{equation}
			||M_e-M_0|| \le ||H^1-I|| + ||H^2|| + ||H^3|| + ||\tilde{A} - \hat{A}|| + ||B||+||C|| + ||D-I||
		\end{equation}
		
		Lets estimate each of those and say that it approaches $0$ as $e\rightarrow 0$. Then we will prove that $M_e$ cannot have $1$ as an eigenvalue. First of all, $H^1$, $H^2$ and $\tilde{A}$ are finite-to-finite-dimensional, and it is easy to see that the elements of their matrices have respective limits. The norm of $H^3$ also goes to zero, since $h_j - e_j$ goes to $0$ as $e\rightarrow 0$.
		
		The similar thing happens to $C$ as well. Its norm goes to zero, since $f_j$ approach some $e_i$ for $j \le 2q_1+1$. This $i$ is not necessarily equal to $j$, since we use not only $1/q$ caustics here, but $i$ is equal to either $2q$ or $2q+1$. Still, $i \le 2q_1+1$, so the dot products of $f_j$ and $e_i$ with $i > 2q_1+1$, represented in $C$, will go to zero along with the norm. Moreover, $S^l$ has multiplied some rows of $C$ on the positive powers of $e$, further decreasing the norm. 

		Now we deal with $D$. We introduce the following lemma.
		
		\begin{lemma} (8.1), \cite{hks}
			The following estimate on $f_j$ holds for $j \ge 2q_1+2$:
			\begin{equation}
				||\frac{1}{\sqrt{\pi}}f_j - e_j||_2 \le \frac{C(e)}{j},
			\end{equation}
			where $C(e) \rightarrow 0$ as $e \rightarrow 0$.
		\end{lemma}
	
	Then, we have the following:
	
	\begin{align}
		||D - I||^2 = \sup_{||x||=1}||Dx-x||^2 = \sup_{||x||=1} ||\sum_{n=2q_1+2}^{\infty} \left\langle x, \frac{1}{\sqrt{\pi}}f_n - e_n\right\rangle e_n||^2 \le \\ \le \sup_{||x||=1} \sum_{n=2q_1+2}^{\infty} \left\langle x, \frac{1}{\sqrt{\pi}}f_n - e_n\right\rangle^2 \le \sum_{n=2q_1+2}^{\infty} ||\frac{1}{\sqrt{\pi}}f_n - e_n||^2 \le \frac{\pi^2}{6}C(e)^2
	\end{align}
		So, the norm of $D - I$ approaches $0$. 
		
		We are only left with $B$. One may assume this bound to be trivial or similar to one we did for $C$. We  should note, however, that we have multiplied the columns of $B$ by negative powers of the eccentricity, coming from $S^r$. Hence, we would need much more accurate estimates, otherwise the negative powers will just make the norm tend to infinity. We will use \eqref{abound} to bound the value $B_{n, i}$. We also include the coefficient, coming from $S^r$.
		
		\begin{equation}
			|B_{n,i}| \le C^{3y+i+1}e^{2y}\frac{2^{-3k}}{e^{q_1-\left[ i/2 \right] }}\le C^{3n} e^{ \left[ n/2 \right] - \left[ i/2 \right] - q_1 + \left[ i/2 \right] } \le  C^{3n} e^{\left[ n/2 \right] - q_1}
		\end{equation}
		
		Note that the power of $e$ here is positive. Then, we can bound the norm of $B$ the following way:
		
		\begin{align}
			|B| \le \sum_{i = 6}^{2q_1+1}\sum_{n = 2q_1+2}^{\infty}|B_{n,i}| \le \sum_{i = 6}^{2q_1+1}\sum_{n = 2q_1+2}^{\infty} C^{3n} e^{\left[ n/2 \right] - q_1}
		\end{align}
	
		that goes to $0$ with the eccentricity. So, we can see that the norm of $||M_e - M_0||$ goes to zero, hence $M_e$ doesn't have $1$ as an eigenvalue, so 
		
		\begin{equation}
			S^l\left(L_e^* - I\right)D_eS^r
		\end{equation} 
	
		is a bijection, so $L_e^* - I$ is a bijection, so $L_e^*$ and $L_e$ do not have $1$ as an eigenvalue for small $e>0$.
	
	\end{proof}

	The main result of this section follows directly from this.

	\begin{proposition}
		The family $\left\lbrace f_j(\vartheta) \right\rbrace $ form a basis in $L^2[0, 2\pi]$ for all but a locally finite amount of $e \in G_e$.
	\end{proposition}
	
	\begin{definition}
		We denote this locally finite set as $\mathcal Z_e$.
	\end{definition}
	
	\subsection{From Lazutkin to elliptic coordinates}
	
	We have proven that some system of functions form a basis. In the next section, we want to finish the proof, applying similar ideas to \cite{ks}. The problem is that \cite{ks} uses functions, written in elliptic coordinates $\varphi$ for that, while our functions are written in Lazutkin coordinates $\vartheta$. Hence we are changing coordinates. This seems just to be a technical move, one could probable finish the proof in Lazutkin coordinates. We introduce new functions $f^\varphi$, that have the following form:
	
	\begin{equation}
		c_{p, q}^\varphi(\varphi) = \frac{\pi}{2K(k_{p/q})}\frac{1}{\sqrt{1 - k_{p/q}^2\sin^2\varphi}}\cos\left( q \theta_{p/q}(\varphi)\right),  \; \; h_i^\varphi(\varphi) = \frac{e_i(\varphi)}{1-e^2\cos^2\varphi}.	
	\end{equation} 
	
	There are also similarly defined $s_{p,q}^\varphi$. The family $f^\varphi(\varphi)$ is constructed from these functions the same way as the regular $f$ is constructed from $c_{p,q}$, $s_{p,q}$ and $h_i$. Notice that the new functions are similar to the old ones in \eqref{cdef} and \eqref{hdef}, we are just considering them as a function of $\varphi$. Since we are using them as a functionals on the space of deformations, we have also removed the jacobian for changing coordinates from $\vartheta$ to $\varphi$.
	
	We say that $f^\varphi$ also form a basis. This is true, since the operator $Y$ that changes the parametrization of deformations from $\varphi$ to $\vartheta$ in $L_2[0, 2\pi]$ is bounded and invertible, since the Lazutkin coordinates are just elliptic with some weight, this weight being smooth, bounded and positive. Then, the operator that transforms $f$ into $f^\varphi$ is just a conjugate of this operator.
	
	\begin{align}
		Y^*f\left(\mu(\varphi) \right) =  \int_0^{2\pi}f(\vartheta) \left(Y \mu\right) (\vartheta)  d\vartheta = \int_0^{2\pi}f(\vartheta) \mu(\varphi(\vartheta))d\vartheta = \\ = \int_0^{2\pi}f(\vartheta(\varphi))\mu(\varphi)J^{-1}d\varphi = \int_0^{\pi} f^\varphi(\varphi)\mu(\varphi)d\varphi = f^\varphi(\mu(\varphi))
	\end{align}

	Then, the basis property follows:
	
	\begin{proposition}
		The family $\left\lbrace f_j^\varphi(\varphi) \right\rbrace $ form a basis in $L^2[0, 2\pi]$ when $e \notin \mathcal Z_e$.
	\end{proposition}

	\begin{remark}
		When $e \notin \mathcal Z_e$, the operator $(I -L_e^*)^{-1}Y^{*-1}$ exists and has a uniformly bounded norm in the neighborhood of $e$.
	\end{remark}
	
	This remark follows from \cite{kato}. The main point is that if a compact analytical operator doesn't have $1$ as an eigenvalue, its eigenvalues are bounded away from $1$ in some parameter neighborhood. We need this to prove the main theorem, since we change ellipses in the proof, so we claim some uniformity. 
	\section{Proof of the main dynamical result}
	
	This part will be similar to the main result section in \cite{ks}. We start with the following fact from \cite{ks}. We will assume that the semi-major axis of the original ellipse is close to $1$.
	\begin{proposition}
		Assume that $q < c(e)||\mu||_{C^1}^{-1/8}$ and that $p$ is uniformly bounded. Also assume that the deformation $\mu$ preserves $p/q$ caustic. Then, we have that
		\begin{equation}
			\int_{0}^{2\pi} \mu(\varphi)c^\varphi_{p,q}(\varphi)d\varphi = O_e(q^8||\mu||_{C^1}^2)
			\label{apqdef}
		\end{equation}
		where  $O_e(q^8||\mu||_{C^1}^2)$ is a term bounded by $q^8||\mu||_{C^1}^2$ times a factor depending on $e$, bound on $p$ and $C^5$ norm of $\mu$. The same is true for $s_{p,q}$.
		\label{smallharm}
	\end{proposition}

	The proof was given in \cite{adsk} for $p = 1$. The proof also works for other $p$, like $p\le 7$.

	We will be using this fact for relatively small $q$. For larger $q$ we will be using the following lemma, that directly follows from the Lemma \ref{scalarsmall} and the bound $||\frac{1}{\sqrt{\pi}}f_j - e_j||_2 \le \frac{C}{j}$, coming from  \eqref{feclose}. In particular, we have
	
	\begin{lemma}
		Let $\mu(\vartheta) \in C^1[0, 2\pi]$. Then, there exists $C = C(e)$, such that for $j \ge 2q_1+2$,
		\begin{equation}
			\left|\int_0^{2\pi} \mu(\varphi) f_j^\varphi(\varphi) d\varphi \right| \le \frac{C\left\| \mu\right\|_{C^1} }{j}. 
		\end{equation}
	\label{largeharm}
	\end{lemma}

	Now, we introduce the following main lemma:
	
	\begin{lemma} (Approximation Lemma)
		Let us consider an ellipse $\mathcal E_{e, c}$ with $e \notin \mathcal Z_e$. Let there be a $q_0$-rationally integrable $C^{39}$ deformation of an ellipse, identified by $C^{39}$ function $\mu(\varphi)$. For every $L > 0$ there exists a constant $C = C(e, c, L)$, such that if $||\mu||_{C^{39}} \le L$, then the following holds. There exists an ellipse $\bar{\mathcal E}$, and a function $\bar{\mu}(\varphi)$, such that the same deformation of $\bar{\mathcal E}$ is identified by $\bar{\mu}$ and 
		\begin{equation}
			||\bar{\mu}||_{C^1} \le C||\mu||_{C^1}^{703/702}
		\end{equation} 
	\end{lemma}

	\begin{proof}
		Consider the basis $f_j^\varphi$ of $L_2[0, 2\pi]$. It forms a basis, since $e \notin \mathcal Z_e$. Also denote $H$ a span of first five elements of $f^\varphi$, the elliptic deformations. We decompose.
		\begin{equation}
			\mu = \mu_H + \mu^\perp
		\end{equation}
		
		Here $\mu_H$ is an orthogonal projection of $\mu$ on $H$. Similarly to \cite{ks} we also have that
		
		\begin{equation}
			||\mu_H||_{C^{39}} \le C(e, c, k)||\mu||_{C^1}.
 		\end{equation}
		
		We claim that 
		
		\begin{equation}
			\left\| \mu^\perp \right\|_{C^1} \le C(e, c, ||\mu||_{C^{39}})||\mu||_{C^1}^{1+\delta}
		\end{equation}
		
		with $\delta = 1/702$. According to \cite{ks} this will complete the proof. We also define Fourier coefficients
		
		\begin{equation}
			\hat{\mu}_j^\perp = \int_0^{2\pi} f_j^\varphi(\varphi)\mu^\perp(\varphi)d\varphi.
		\end{equation} 
		
		Those are zero for $j$ from $1$ to $5$ due to the definition of $\mu^\perp$, so
		
		\begin{equation}
			\left\| \mu^\perp \right\|_2^2 \le C(e, c)\sum_{j=6}^{\infty}|\hat{\mu}_j^\perp|^2 
		\end{equation}
		
		Then, we follow \cite{ks} and break up all $\hat{\mu}$ into $2$ groups, one for $j \le \left\| \mu \right\|_{C^1}^{-1/9} > 2q_1 + 2$, and another for larger $j$. For the first group we use Proposition \ref{smallharm}, and for the second - Lemma \ref{largeharm}.
		
		So, we get that
		
		\begin{equation}
			||\mu^\perp||_2 \le C(e, c)||\mu||_{C^1}^{19/18}
		\end{equation}
		
		Then we use the same Sobolev identities as in \cite{ks} to finish the proof.
		
	\end{proof}
	
	The proof of the main result is now identical to \cite{ks}. The only difficulty we may face is that the minimal ellipse that is constructed has eccentricity inside $\mathcal Z_e$ and that inverse operator will fail to be uniformly bounded in the bigger neighborhood of $e$, but we can just bound the size of deformation $\varepsilon$ to assure this won't happen.

	\section{From caustics to Laplace spectrum}
	Let $\mathcal E_e$ be an ellipse with eccentricity $e$ and large semi-axis of length $1$. We pick some number $e_{\max}$ close to $1$. We will only consider ellipses with $e < e_{\max}$. Since we can choose any $e_{\max}$ as we wish, this won't be a problem.
	
	We will be considering arc-length coordinates on an ellipse, we will denote them as $(s, \varphi)$. Using them, we define a billiard map on the phase cylinder
	
	\begin{equation}
		B_0 : (s, \varphi) \rightarrow (s', \varphi')
	\end{equation}
	
	\subsection{Deformations and spectra}
	
	We consider $\Omega = \mathcal E_e + \mu(s) n(s)$. We demand that $\mu$ is a $C^\infty$ smooth deformation, its $C^{39}$ norm is bounded and its $C^{10}$ norm is small. Then we can parameterize the cyllinder phase space of $\Omega$. We can define $s, \varphi$ to be arc-length coordinates on the deformation and define a billiard map $B_\mu$ on them.

	Now we introduce some new objects, that are related to the spectrum of the domain. First of all, we will say that the periodic billiard orbit has type $(p, q)$ if it hits the boundary $q$ times (not necessarily a minimal period) and winds around it $p$ times. These orbits have rotation number $p/q$, but there could be others since $p$ and $q$ may be share a common factor. We would assume that $2p \le q$, otherwise we just get reverse orbits. The closure of the union of lengths of every periodic point of the domain is called a length spectrum of the domain.
	
	For every domain, one could define the values $t_{p, q}$ and $T_{p, q}$ -- they are the lengths of minimal and maximal orbits of type $(p, q)$ respectively. These always exist as they are the minimax and the maximum of the length functional. So, the part of the spectrum, corresponding to $(p, q)$ orbits is restricted to the interval $\left[ t_{p, q}, T_{p, q} \right]$. If a domain has a caustic with rotation number $p/q$, then all the orbits of type $(p, q)$ share the same length, meaning $t_{p, q} = T_{p, q}$. The converse is also true. So, for an ellipse, the length spectrum consists of points $t_{p, q} = T_{p, q}$ for $2p < q$, and points with $2p = q$, since for them there are no caustics, as well as multiples of perimeter. For example, there are bouncing ball orbits on both axes, as well as orbits in the "eye" that stay tangent to the hyperbolae: all of them have type $(p, 2p)$.
	
	The length spectrum is closely related to the Laplace spectrum of the domain with Dirichlet or Neumann boundary conditions. So called wave trace of the domain is introduced:
	
	\begin{equation}
		w(t) = Tr \cos(t\sqrt{\Delta}).
	\end{equation}
 	
 	The singular support $\mathrm{sing supp}$ of $w(t)$, meaning places where $w(t)$ is not $C^\infty$ smooth, satisfies the following Poisson relation:
 	
 	\begin{equation}
 		\mathrm{sing supp} \; w(t) \subset \pm L \bigcup 0,
 	\end{equation}
 	
 	where $L$ is the length spectrum. The relation holds, since the singularities and waves travel along the billiard orbits under the wave equation inside a domain. The reverse relation is generally not true, see \cite{zel}. An orbit may not be visible in this set, because it's singularity may cancel with another orbit if they have the same length or there may be no smooth generating function in the neighborhood of the orbit. 
 	
 	If we are proving Laplace rigidity, we preserve wave trace singularities. We want to say something about caustics and periodic orbits. A plan then arises naturally: we take a singularity for an ellipse, corresponding to some needed caustic $p/q$ in our dynamical result. We want to prove that the deformation also has this caustic. We know that the singularity is also there for a deformation, it is generated by some periodic orbit due to the Poisson relation. We want to say that it is generated by the orbits $(p, q)$. We also notice that for an ellipse there are no singularities nearby it, so for a deformation orbits $(p, q)$ do not generate any other singularity, unless they cancel with some other orbit. Then, all the $(p, q)$ orbits have the same length, so we have a caustic. 
 	
 	We see that these cancellations pose a problem for us. Maybe we have destroyed a $p/q$ caustic for a  deformation, so we have $t_{p, q} < T_{p, q}$. However, the singularity at $T_{p, q}$ may cancel with an orbit of other type. Then, we will be just able to see $t_{p, q}$, looking at the singularities, the same as we get looking at the ellipse. Cancellations are extremely rare, but there are a lot of ways to perturb a domain, so the main idea is to guarantee that there won't be the incidence of lengths. That will further restrict the space of ellipses. 
 	
 	\subsection{Continuity of the spectrum at ellipses}
 	
 	First, we prove that the billiard map itself is continuous over $\mu$, proving the following lemma. We mention that near the boundary there are various singularities of the billiard map, so we will restrict away from it.
 	
 		\begin{lemma}
 		For every $\delta > 0$ when $\varepsilon$ is small enough, for $\varepsilon$-small deformation of $\mathcal E_e$ the following estimate takes place:
 		\begin{equation}
 			\left\|  B_\mu(s, \varphi) - B_0(s, \varphi) \right\|_{C^{9}}^{\delta \le \varphi \le \pi - \delta} = O_{e_{\max}, \delta}(\varepsilon), \varepsilon \rightarrow 0.
 		\end{equation}
 		
 		Moreover, the billiard distance traveled before the next hit $l$ has the following bound:
 		
 		\begin{equation}
 			\left| l_\mu(s, \varphi)  - l_0(s, \varphi)\right|^{\delta \le \varphi \le \pi - \delta} = O_{e_{\max}, \delta}(\varepsilon),  \varepsilon \rightarrow 0.
 			\label{lengthdeform}
 		\end{equation} 
 		Here, $(s, \varphi)$ can be any point of the phase cylinder.
 	\end{lemma}
 	
 	\begin{proof}
 		This lemma is pretty similar to Lemma 3.12 in \cite{hezzel}. We are going to use implicit function theorem several times and construct needed billiard map. We also want to make sure that all of our bounds are uniform. 
 		
 		First, under a deformation arc-length coordinates changed, meaning the point of parameter $s$ in $\Omega$ does not necessarily lie on the normal of point $s$ in $\mathcal E_e$. Just to have a common starting point, we assume that it is true for $s = 0$. In this proof we will also assume the perimeter of $\mathcal E_e$ to be equal to $2\pi$. Then, we can find the arc length coordinate in $\Omega$ of the point lying on the normal to an ellipse at $s$ as
 		
 		\begin{equation}
 			s_{\mu}(s_0) = \int_{0}^{s_0} \left| \frac{d}{ds}\left(E_e(s) + \mu(s) n(s)\right) \right|ds
 		\end{equation}  
 		
 		We claim that $||s_{\mu}(s_0) - s_0||_{C^{10}} = O(\varepsilon)$. Next we want to say that if the perimeter of $\Omega$ is not equal to $2\pi$, we will have to normalize this formula by multiplying it by some constant of order $\varepsilon$. Estimates will still hold. Next, we want to use inverse function theorem and find $s_0(s_\mu)$. Since the first derivative of $s_{\mu}(s_0)$ is bounded away from $0$, we see that $||s_0(s_{\mu})||_{C^{10}} = O(\varepsilon)$ uniformly.
 		
 		The current goal is to obtain that the generating function of the billiard map continuously depends on the deformation. We can study the vector-function $\Omega(s) - \mathcal E_e(s)$.
 		\begin{equation}
 			\Omega(s) - \mathcal E_e(s) = \Omega(s) - \mathcal E_e(s_0(s)) + \mathcal E_e(s_0(s)) - \mathcal E_e(s) = \mu(s_0(s))n(s_0(s)) + \mathcal E_e(s_0(s)) - \mathcal E_e(s)
 		\end{equation} 
 		
 		The first term has small $C^{10}$ norm because of the deformation, and the second - because $s_0$ and $s$ are close. Hence, the generating function of the billiard map
 		
 		\begin{equation}
 			h_\mu(s, s') = |\Omega(s') - \Omega(s)| = |\Omega(s') - \mathcal E_e(s') + \mathcal E_e(s') - \mathcal E_e(s) + \mathcal E_e(s) - \Omega(s)|
 		\end{equation}
 		
 		is smooth and $C^{11}$ close to the generating map of the ellipse
 		
 		\begin{equation}
 			h_0(s, s') = |\mathcal E_e(s') - \mathcal E_e(s)|.
 		\end{equation}
 		We note that this is only true when we restrict away from the boundary. If we allow $s$ and $s'$ to be close, then the function of absolute value has a singularity at $0$, so it will not respect the derivatives. However, we are away from the boundary, so the absolute value is bounded away from its singularity, so it preserves smallness of derivatives. We want to use these generating functions, because the billiard map can be described with them. If $y = -\cos(\varphi)$, then the following holds:
 		
 		\begin{equation}
 			y = -\frac{\partial}{\partial s} h(s, s'), \; \; y' = \frac{\partial}{\partial s'} h(s, s').
 		\end{equation}
 		We now know a function $\varphi(s, s')$, but we want to find $s'(s, \varphi)$ as an implicit function of the first equation. After that, we will just have to substitute it into the second relation and find $\varphi'$ as a function of $s$ and $\varphi$. We build a function
 		
 		\begin{equation}
 			F(s', s, \varphi) = \arccos h_s(s, s') - \varphi
 		\end{equation} 
 		and apply implicit function theorem when $s \ne s'$. We note that
 		\begin{equation}
 			F_{s'}(s', s, \varphi) = - \frac{h_{ss'}(s, s')}{\sqrt{1 - h_s^2(s, s')}} = \frac{\sin \varphi'}{h(s, s')}.
 		\end{equation}
 		
 		This value is bounded away from $0$ in terms of minimal curvature, so uniformly over deformations and considered ellipses. We already know global function $s_\mu'(s, \varphi)$ exists. We bound the difference $|s_\mu'(s, \varphi) - s_0'(s, \varphi)|$, since 
 		\begin{equation}
 			|s_\mu'(s, \varphi) - s_0'(s, \varphi)| \le \frac{1}{F_{\tilde{s}'}} |\varphi_{\mu}(s, s_\mu'(s, \varphi)) - \varphi_{\mu}(s,s_0'(s, \varphi))| \le  \frac{1}{F_{\tilde{s}'}} |\varphi_{0}(s, s_0'(s, \varphi)) - \varphi_{\mu}(s,s_0'(s, \varphi))|
 		\end{equation}
 		and that $\varphi_0$ and $\varphi_\mu$ are $\varepsilon$ - close. Now we can extend this to the derivatives of $s'(s, \varphi)$. For example, we can write down two identities
 		\begin{equation}
 			F_\mu(s_\mu'(s, \varphi), s, \varphi) = 0, \; \; F_0(s_0'(s, \varphi), s, \varphi) = 0
 		\end{equation}
 		and differentiate them several times over $s$ and $\varphi$. The highest derivative term of $s'(s, \varphi)$ will appear with a coefficient $F_{s'}$. Hence, we will be able to express this term as a fraction. The numerators will consist of derivatives of $F$, evaluated at $s_\mu'$ and $s_0'$ as well as of lower derivatives of $s'$. Hence, they will be by an order of $\varepsilon$ different between those identities. Since the denominators are also $\varepsilon$-close and bounded away from $0$, we get that derivatives of $s'$ are close for $\Omega$ and $\mathcal E_e$. We can do this while $F_{\mu}$ and $F_0$ are close, so $||s_\mu' - s_0'||_{C^{9}} = O(\varepsilon)$, since we lost one derivative to the derivation in $F$.
 		
 		If we are away from the boundary, then we can substitute $s' = s'(s, \varphi)$ into a formula for $\varphi'$ and obtain that $||\varphi_\mu' - \varphi_0'||_{C^{9}} = O(\varepsilon)$, thus proving the lemma.
 	\end{proof}
 	
 	We will denote $O_{e_{\max}}(\varepsilon)$ as $O(\varepsilon)$.
 	
 	We wanted to get that the lengths of orbits of needed types ($p \le 7$, $q$ is large) do not coincide with other types. This of course means studying $t_{p, q}$ and $T_{p, q}$. We can compute them for ellipses, using elliptic integrals and so on. However, we need some way of controlling them for a deformation. Specifically, we say that they cannot change greatly under the deformation. This is vital for us - otherwise there would be no way to prevent a cancellation since the lengths may be traveling as they please.
 	
 	$T_{p, q}$ are somewhat easier in this field, since they are the maximum of the length functional, depend continuously on the deformation and can be expressed using Mather's beta-function, that is continuous under deformations. Particularly, they increase over $q$. $t_{p, q}$ are harder, since they lack this good structure. For example, in some domains one can easily destroy an orbit of minimal length, increasing $t_{p, q}$ by a big amount under a small perturbation. However, for $p/q < 1/2$ ellipses have a caustic, so $t_{p, q}$, bounded from above by $T_{p, q}$, cannot increase rapidly under deformation. For $p/q = 1/2$, orbits can disappear, but we only care that the new orbits won't be created, because we only study them to assure incidences and cancellations won't happen. 
 	
 	\begin{lemma}
 		Assume $p < p_0$ and $p/q < 1/2$. Let $L_{p, q}(\Omega)$ be any obit of type $(p, q)$ for $\Omega$, where $\mu$ is an $\varepsilon$ small deformation of a fixed ellipse $\mathcal E_e$. Then, the following holds:
 		\begin{equation}
 			|L_{p, q}(\Omega) - T_{p, q}(\mathcal E_e)| = o_{p, q, e}(1), \varepsilon \rightarrow 0
 		\end{equation}
 		\label{lemmaz4}
 	\end{lemma}

 	\begin{lemma}
 		Assume $p < p_0$ and $q = 2p$. Let $L_{p, q}(\Omega)$ be any obit of type $(p, q)$ for $\Omega$, where $\mu$ is an $\varepsilon$ small deformation of a fixed ellipse $\mathcal E_e$. Then, there exists a length of an orbit $L_{p, q}(\mathcal E_e)$ of type $(p, q)$, such that the following holds:
 		\begin{equation}
 			|L_{p, q}(\Omega) - L_{p, q}(\mathcal E_e)| = o_{p, q, e}(1), \varepsilon \rightarrow 0
 		\end{equation}
 		\label{lemmaz5}
 	\end{lemma}
 	
 	The main idea of the proof is to say that the deformed dynamics are close to the original ones, so we can get a nearly-periodic orbit in the ellipse by picking the same starting point in the phase space. Then, we will use compactness of the phase space to find a true periodic orbit in the neighborhood.
 	
	Let $\left( s_i^0, \varphi_i^0\right)_{i = 0}^q$ be a periodic orbit of type $(p, q)$ of the deformation. Right now assume each point in the orbit is $2\delta$ away from the boundary in the sense of the previous lemma. We talk why we can assume this right after the proof of Lemma \ref{lemmaz2}. We will prove that there exists a periodic orbit in $\mathcal E_e$ of the same type nearby. 
	
	Define $\left(s_i^j, \varphi_i^j \right)$ for $0 \le j \le i$ as a point we get by first iterating $\left( s_0^0, \varphi_0^0 \right) $ $i -j $ times using $B_\mu$ and then $j$ times using $B_0$. Note that we have the following:
	
	\begin{equation}
		\left| \left(s_i^1, \varphi_i^1 \right) -  \left(s_i^0, \varphi_i^0 \right) \right| = O(\varepsilon). 
	\end{equation}
	
	Then, since the billiard map inside an ellipse is smooth over $(s, \varphi)$ and $e$, we can iterate the previous bound $q-i$ times over $B_0$. Then, we get a bound on the final points:
	
	\begin{equation}
		\left| \left(s_q^j, \varphi_q^j \right) -  \left(s_q^{j-1}, \varphi_q^{j-1} \right) \right| = C^{j-1}O(\varepsilon). 
	\end{equation}
	
	Using triangle inequality, we get that
	
	\begin{equation}
		\left| \left(s_q^q, \varphi_q^q \right) -  \left(s_q^0, \varphi_q^0 \right) \right| = \left| \left(s_q^q, \varphi_q^q \right) -  \left(s_0^0, \varphi_0^0 \right) \right| = qC^{q-1}O(\varepsilon). 
		\label{distdiff}
	\end{equation}
	
	Since the length is also a smooth function in an ellipse, we get that the lengths of the periodic orbit $l_\mu$ and of the iterated $B_0$ orbit $\tilde{l}_0$ may differ slightly:
	
	\begin{equation}
		\left|l_\mu - \tilde{l}_0\right| = qC^{q-1}O(\varepsilon). 
		\label{lengthdiff}
	\end{equation}
	
	However, the new orbit in an ellipse is not necessarily periodic. We want to prove there is a periodic orbit with the of the same type nearby. For this we will use the following lemma:
	
	\begin{lemma}
		Consider a $q$-iterate of the billiard map of the ellipse $\mathcal E_e$ on the universal cover of the cylinder $\tilde{B}_0^q$. Then, for every $\tilde{\varepsilon} > 0$ there exists $\varepsilon > 0$, such that if for some $(s, \varphi)$
		
		\begin{equation}
			\left|\tilde{B}_0^q (s, \varphi) - (s + 2\pi p, \varphi) \right| < \varepsilon,
		\end{equation} 
		
		then there exists a periodic orbit of type $(p, q)$ starting in a point $\left( \hat{s}, \hat{\varphi}\right) $ for an ellipse, such that
		
		\begin{equation}
			\left|\left( \hat{s}, \hat{\varphi}\right) -  (s, \varphi) \right|  < \tilde{\varepsilon}.
		\end{equation}
		\label{lemmaz2n}
	\end{lemma}
	\begin{proof}
		Assume such $\varepsilon$ does not exist. Then, we can obtain a sequence of counter-examples with $\varepsilon$ going to zero, while $\tilde{\varepsilon}$ stays constant. However, since the phase space is compact, there would be some limit point of this sequence. Since the billiard map is continuous, the limit point would be a periodic orbit of type $(p, q)$. This leads to contradiction, since this orbit would be $\tilde{\varepsilon}$ - close to some of the elements of the sequence.
	\end{proof}
	
	Now, we can prove the following two Lemmata \ref{lemmaz4} and \ref{lemmaz5}:
	
	\begin{proof}
		We will proof these facts together. For every $\tilde{\varepsilon} > 0$ we follow these steps. First, using Lemma \ref{lemmaz2n} for $\tilde{\varepsilon}$, we obtain $\varepsilon_1$ (it is called $\varepsilon$ in the lemma). Then, for small enough $\varepsilon$ the term $qC^{q-1}O(\varepsilon)$ from \eqref{distdiff} and \eqref{lengthdiff} will get smaller than $\tilde{\varepsilon}$ and $\varepsilon_1$. Then, due to Lemma \ref{lemmaz2n} and \eqref{distdiff}, there will exist a $(p, q)$ periodic point of an ellipse $\tilde{\varepsilon}$ close to $ \left(s_0^0, \varphi_0^0 \right)$ with length $l_0$. Then, 
		
		\begin{equation}
			\left| l_\mu - l_0\right| \le \left| l_\mu - \tilde{l}_0\right| +\left| \tilde{l}_0 - l_0\right| = qC^{q-1}O(\varepsilon) + qC^{q-1}O(\tilde{\varepsilon}) \rightarrow 0 , \tilde{\varepsilon} \rightarrow 0
		\end{equation}
	\end{proof}
	
	\subsection{KAM - theory and large $p$ orbits}
	
	We have proven continuity for each type of orbits. The problem, however, is that the bounds are not uniform over the type. Since there are infinite amount of types, it is a problem. So, we will only use these lemmata for small $p$ and $q$. Other orbits can be divided into $2$ classes. The first class has unbounded $p$ and $q \ge 2p$, while the second has bounded $p$ and $q \rightarrow \infty$. The lengths of the second class tend to the multiples of the lengths of the boundary, as we will see later. First type orbits wind around the boundary many times, so one would assume their journey to be quite long. To have a rigorous proof, we have to use some invariant curves and KAM - theory. 
	
	Particularly, we will prove the following bound:
	
	\begin{lemma}
		There exists $p_0 \in \mathbb{N}$, such that for every $\mathcal E_e$ and $\Omega$: 
		\begin{equation}
			t_{p, q} > 16 \pi
		\end{equation}
		for every $p \ge p_0$  and $q \ge 2p$.
		\label{lemmaz2}
	\end{lemma}
	
	\begin{proof}
		First, we only consider one ellipse $\mathcal E_e$ and prove the existence of such $p_0$ that may depend on $e$. The idea of the proof is to establish an invariant KAM curve that won't be destroyed by the deformation. Then, we will separately study orbits above and below this curve.

		We need to use action-angle coordinates for an ellipse. We define action-angle map $\Phi$:
		
		\begin{equation}
			\Phi(s, \varphi) = (\theta, I).
		\end{equation}
		
		It is correctly defined for small enough $\varphi$, it is a symplectic map, that has the following property:
		
		\begin{equation}
			\Phi \circ B_0 \circ \Phi^{-1} (\theta, I) = \left(\theta + \alpha(I), I \right) 
		\end{equation}
		
		This map is smooth for $\varphi > 0$, although it ceases to be so at $\varphi = 0$. We will be considering a strip of a cylinder $(0, 2\pi)\times (\varphi_{min}, \varphi_{max})$. We want it to be close to the boundary, so that $\Phi$ is well defined, but not touching it, so that $\Phi$ would be smooth. We also demand $\varphi_{min}$ to be small enough, so that the image has an open strip $(0, 2\pi)\times(I_{min}, I_{max})$ contained in it. By decreasing $\varphi_{min}$, we make $I_{min}$ as small as we wish. Particularly, $I$-interval should contain a neighborhood of some Diophantine number $\omega$. 
		
		Since $\Phi$ is smooth and symplectic inside the strip, the map in a deformation can be considered the following way
		
		\begin{equation}
			\Phi \circ B_\mu \circ \Phi^{-1} (\theta, I) = \Phi \circ B_0 \circ \Phi^{-1} (\theta, I) + \left(P(\theta, I), Q(\theta, I)\right).
		\end{equation}    
		
		Here, the norms of $P$ and $Q$ are small, and the map is also symplectic. Now, we are going to use KAM theory. We know that the unperturbed map has an invariant curve $I = \omega$. Then, we can use the main result from \cite{Zehnder1989}. 
		
		It says that if the starting system had an invariant curve with rotation number $\omega$, so that
		
		\begin{equation}
			\left|\omega - m/n \right| \ge \frac{\gamma}{n^\tau}
		\end{equation}
		
		(we can assure this holds for $\tau = 2.1$ by selecting needed $\omega$) and if $P$ and $Q$ have small $C^8$-norm, then there exists some functions $p(\theta)$ and $q(\theta)$ with small enough $C^1$-norm, so that the deformed map has an invariant curve 
		
		\begin{equation}
			\theta = \theta' + p(\theta'); \; \; I = \omega + q(\theta').
		\end{equation} 
		
		$P$ and $Q$ have small $C^8$ norms, when $B_\mu -B_0$ has a small $C^8$ norm, since $\Phi$ is smooth. 
		
		From the existence of such a curve we get a very important corollary. Either a deformed orbit has all the $I$ larger than $I_{min}$ or smaller than $I_{max}$. Going back to the arc-length coordinates we get that either an orbit has all the $\varphi > \varphi_{min}$ or all the $\varphi < \varphi_{max}$. 
		
		Let's consider the first case. We note that the same thing happens on the upper half of the cylinder, meaning we can assume $\varphi$ is also bounded from above from $\pi$. Then, we can say that the length of every segment of the orbit is bounded away from zero. The fact that the length of a chord inside an ellipse, not forming small angles with the boundary can't be small is true. Then, we can just use \eqref{lengthdeform} to prove it for the deformation. So, the length of the whole orbit is at least $ql_{min}$. Then,
		
		\begin{equation}
			16 \pi \ge l_\mu \ge ql_{min} \ge pl_{min}.
		\end{equation}
		
		So, $p$ is bounded. Next, we consider the second case: $\varphi < \varphi_{max}$. Then, we can bound the ratio between the arc length difference $s' - s$ and the segment length for each reflection. We have the following trivial bound:
		
		\begin{equation}
			s' - s \le |p' - p| \left( \frac{1}{\cos \varphi} + \frac{1}{\cos \varphi'} \right)   \le \frac{2|p' - p|}{\cos \varphi_{max}}.
		\end{equation}
		
		Here, $|p'-p|$ is a length of a segment. From here we observe, that
		
		\begin{equation}
			2p\pi \le \frac{2l_\mu}{\cos\varphi_{max}} \le \frac{32\pi}{\cos \varphi_{max}}.
		\end{equation}
		
		This of course places bounds on $p$. We have proven the lemma for a given ellipse. Now we just need to say that ellipses with similar eccentricities can also be counted as the small deformations. So, we have proven the lemma for some small interval of eccentricities. Since the needed interval $(0, e_{\max})$ can be made compact, this finishes the proof. 
		
	\end{proof}

	Now we can explain the assumption in the proof of Lemma \ref{lemmaz4}. We mentioned every $(p, q)$ orbit should be bounded away from the boundary. It is true, since otherwise we can take some small KAM rotation number and say that it has a persistent KAM curve for every deformation. No $(p, q)$ orbit can go below this curve, otherwise $\varphi$ will stay small and we won't be able to rotate around the boundary $p$ times. Hence, all orbits stay outside of the KAM curve, so $\varphi$ should be bounded from below by $2\delta$. Of course this $\delta$ depends on $p$ and $q$, but this is okay for us.  
	
	\begin{lemma}
		The following estimates hold for every $\mathcal E_e$ and $\Omega$:
		\begin{equation}
			T_{p, q} < 15\pi
		\end{equation}
		for $q \in \mathbb{N}, p\le 7$.
		
		\label{lemmaz1}
	\end{lemma}

	So, since we will only use caustics with $p \le 7$ to prove the result, this means that orbits with $p\ge p_0$ won't make any difference in the proof, so we will not study them.
	
	\subsection{Expansion for bounded $p$}
	
	Now we deal with the second class of orbits. Since here $p$ is bounded, and $q$ is large, the orbits are very close to the boundary. So it makes sense to study them in Lazutkin coordinates. This allows us to get estimates for their dynamics and lengths and get expansions for them. These are studied in \cite{mm} a well as in \cite{hezzel} and \cite{dkw}, where quantitative versions were obtained.
	
	\begin{lemma}
		Uniformly for all $p < p_0$, $e<e_{\max}$ and $\mu$ with small $C^{10}$ norm we have the following:
		\begin{equation}
			L_{p, q} = p\ell(\Omega) - c_{2, p}(\mu) q^{-2} + O(q^{-4}), \; q\rightarrow \infty.
			\label{lengthseries}
		\end{equation}
		Here, $L_{p, q}$ is the length of any orbit of type $(p, q)$. Particularly we have the following:
		\begin{equation}
			T_{p, q} - t_{p, q} =  O(q^{-4}), \; q\rightarrow \infty.
		\end{equation}
		Moreover,
		\begin{equation}
			c_{2, p} = \frac{p^3}{24}\left( \int_0^\ell \kappa^{2/3}(s)ds\right)^3 
		\end{equation}
	\end{lemma}
	
	\begin{proof}
		The idea of he proof is similar to Lemma 4.3 of \cite{hezzel}. We cannot directly use it, since in the case of non-nearly circular domains the term with $q^{-3}$ has a non-small coefficient, so this may lead intervals to overlap, since the distance between them is also of order $q^{-3}$. We just need to go one step further and remove this term altogether by using higher order Lazutkin coordinates (\cite{hezzel} used an order $5$). This will lead us to increased smoothness requirements.
		
		We will use Lazutkin coordinates $(u, v)$ of order $6$ in the proof. The dynamics in these coordinates for small rotation numbers is given by the following:
		\begin{equation}
			B_\mu(u, v) = (u+v+v^6a(u, v), v+v^7b(u, v)),
		\end{equation}
		where smooth functions $a$ and $b$ are bounded by $O(\left\| \frac{1}{\kappa}\right\|_{C^{5}})$.
		
		Now assume that $(u_0, v_0)$ is a starting point of periodic orbit of type $(p, q)$ with $p \le p_0$. Then,
		
		\begin{equation}
			u_0+qv_0+qO(|v_0|^6\left\| 1/\kappa\right\|_{C^{5}}) = u_0 + p.
		\end{equation}
		
		We have a following bound on $v_0$:
		
		\begin{equation}
			v_0 = \frac{p}{q} + O(q^{-6}\left\| 1/\kappa\right\|_{C^{5}})
		\end{equation}
		
		By iterating the starting map $j \le p_0q$ times, we get that
		
		\begin{equation}
			u_j = u_0 + \frac{pj}{q} + \frac{O(\left\| 1/\kappa\right\|_{C^{5}})}{q^5}, \; \; v_j = \frac{p}{q} + \frac{O(\left\| 1/\kappa\right\|_{C^{5}})}{q^5}.	
		\end{equation}
		Now we go back to regular Lazutkin $\vartheta, \eta$ (We scale $\vartheta$ from $0$ to $1$ here). They are related to $u, v$ as
		
		\begin{equation}
			(\vartheta, \eta) = \left( u+v^2A(u,v), v+v^3B(u,v)\right) 
		\end{equation}
	
		with norms $C^k$ of $A$ and $B$ being bound by $C^{k+2}$ and $C^{k+3}$ norm of curvature respectively. Particularly,
		
		\begin{equation}
			A(u,v) = A_0(u)+A_1(u)v+A_2(u)v^2+O\left( \left\| 1/\kappa\right\|_{C^{5}}\right) |v|^3.
		\end{equation}
		So,
		
		\begin{equation}
			\vartheta_j = u_0 + \frac{pj}{q} + \frac{p^2A_0(u_0+pj/q)}{q^2} + \frac{p^3A_1(u_0+pj/q)}{q^3}+\frac{p^4A_2(u_0+pj/q)}{q^4}+ \frac{O(\left\| 1/\kappa\right\|_{C^{5}})}{q^5}
		\end{equation}
	
		After writing $u_0$ in terms of $\vartheta_0$ and $\eta_0$, we get
		
		\begin{equation}
			\vartheta_j = \vartheta_0 + \frac{pj}{q} + \frac{p^2\alpha_1(\vartheta_0+pj/q)}{q^2} + \frac{p^3\alpha_2(\vartheta_0+pj/q)}{q^3}+\frac{p^4\alpha_3(\vartheta_0+pj/q)}{q^4}+ \frac{O(\left\| 1/\kappa\right\|_{C^{5}})}{q^5},
		\end{equation}
	
	    with $\left\| \alpha_j(\vartheta)\right\|_{C^m} =  O(\left\| 1/\kappa\right\|_{C^{m+j+1}})$. 
	    
	    Then we do the rest of the proof the same way as in \cite{hezzel}. At the end we will get that
	    
	    \begin{equation}
	    	T = a_0 + \frac{a_1}{q} + \frac{a_2}{q^2}+\frac{a_3}{q^3} + \frac{O(1 + \left\| \mu \right\|_{C^{10}} )}{q^4}
	    \end{equation}
    	From the expansion in \cite{mm} and \cite{sorrent}, we get that $a_1=a_3=0$, $a_0 = p\ell$ as well as a formula on $c_{2,p} = -a_2$
	\end{proof}

	\subsection{The rest of the proof}
	
	Now we have proven all the preliminary results and we will use them to prove the theorem.
	
	Let us select $q_0$ for further use. We choose it in a way, such that for all ellipses $\mathcal E_e$ with $e < e_{\max}$, all their small deformations $\Omega$, all $p$ from $1$ to $7$ and all $q > q_0$ there would exist a smooth $q$-loop function that gives the length of an orbit that makes $p$ windings around the boundary of $\Omega$, defined for every point on the boundary of $\Omega$. The existence of such $q_0$ is true, since we can use Lazutkin coordinates near the boundary. We can essentially use ideas, similar to the previous lemma and an implicit function theorem. This parametrizations are obtained using curvature and its derivatives. Since we have bounded the eccentricity, we have bounded this value for the ellipses and small deformations cannot greatly influence this value, since it's norm in $C^{10}$ is small. 
	
	We also choose $q_0$ large enough, so that for all ellipses $\mathcal E_e$ and their small $\varepsilon$ deformations $\Omega$ for each $1 \le p \le 7$ and $q\ge q_0$ the following result holds:
	
	\begin{equation}
		t_{p, q}(\Omega) \ge \frac{2p-1}{p}\ell(\Omega).
		\label{q0select}
	\end{equation}
	
	This is possible due to \eqref{lengthseries} and that $c_{2, p}$ is bounded for deformations, since it is expressed in curvature. We add this requirement to avoid problems when studying the non-incidence condition.
	
	After we have chosen $q_0$, our caustic part gives us a family of caustics $\mathcal F$ that we need to preserve. If we preserve all the caustics from $\mathcal F$, that would mean that our deformation is an ellipse. So, $\mathcal F$ consists of a finite family of caustics coming from conditions $(p_i, q_i)$ in Lemma \ref{mainfieldlemma}, as well as $1/q$ for $q > q_1$. We break $\mathcal F$ up into two parts. The first part $\mathcal F_1$ consists of all the caustics with $p = 1$ in $\mathcal F$. It is an infinite set and in it all $q > q_0$. The second part $\mathcal F_2$ consists of all the other caustics with $2 \le p \le 7$. It is a finite set. 
	
	We will assume that $e \notin \mathcal Z_e$, otherwise we cannot obtain any result. We will also assume a non-incidence condition for $\mathcal E_e$, defined as the following
	
	\begin{definition}
		We say that $e$ satisfies non-incidence condition, if the lengths corresponding to caustics in $\mathcal F$ for the ellipse $\mathcal E_e$ are realized in the length spectrum only as the length corresponding to the respective caustic, not as any other orbit. We also demand that the lengths corresponding to the elements of $\mathcal F$ for $\mathcal E_e$ do not coincide with multiples of length of the boundary and that the length of the boundary is not realized in the length spectrum of $\mathcal E_e$ as a length of any periodic orbit.
		\label{noninc}
	\end{definition}
	
	Let us now assume for the moment that $e$ satisfies both conditions. Then, we will show that rigidity holds. 
	
	We can treat $\mathcal F$ as the set of types of orbits $(p, q)$. We can propose the latter lemma, related to the wave trace singularities:

	\begin{lemma}
		Assume that $(p, q) \in \mathcal F$. Also assume that for $\Omega\;$  $t_{p,q}$ and $T_{p, q}$ are not realized in the length spectrum through other types of orbits and that they are also not multiples of the length of the boundary. Then, the following holds:
		\begin{equation}
			t_{p, q}, T_{p, q} \in \mathrm{sing supp} \; w_\Omega(t)
		\end{equation}
		\label{lemmaz6}
	\end{lemma}
	
	This lemma follows form \cite{mm} and from the choice of $q_0$. Similar relations were studied in \cite{mm} and \cite{hezzel}. However, in their works only orbits with $p = 1$ are studied. In our case, we need to have a result for $p \le 7$ and for large $q$. The main problem with other orbits was that the generating function may not exist. We present some basic remarks about the idea of the lemma.
	
	Wave trace of the domain can locally be decomposed into a sum of distributions, each corresponding to their own type of periodic orbit, up to $C^{\infty}$ smooth error, that does not influence singularities. For example, each of these distributions has  its singularities contained in the respective part of the length spectrum. We mention that we need to be away from the length of the boundary for this to hold. Our non-incidence condition forces every singularity we need to be away from multiples of the length of the boundary. Due to the restriction on $q_0$, our $(p, q)$ orbits have generating functions. Because of this, these distributions can be expressed as an oscillatory integral with exponent of form $i\xi(t - L_{p, q}(s))$, where $L_{p, q}(s)$ is a respective generating function, evaluated at $s' = s$, once again, up to a smooth error. 
	
	This holds, because the generating function forces the singularities of the solution kernel of the wave equation to propagate along the Lagrangian submanifold of $T^*(\Omega \times \mathbb{R})$ that is defined by the generating function. So, the needed part of the wave trace, corresponding to orbits in $\mathcal F$, can be studied as oscillatory integrals.  
	
	To prove that there is a singularity, we decompose the wave trace and multiply it by a smooth cutoff function, supported in the neighborhood of $t_{p, q}$ or $T_{p, q}$ for $(p, q)$ in $\mathcal F$. Since lengths of orbits of other types are away from this point due to lemma assumptions, distributions of other types will have no singularities in the neighborhood, so will be smooth when multiplied by the cutoff. Hence, the study of whether $w_\Omega(t)$ has a singularity is equivalent of studying if the $(p, q)$ distribution has a singularity at this point. Since it can be expressed as an oscillatory integral, we can use similar techniques, mentioned in \cite{hezzel}, and used there argument of Soga to prove that it has a singularity at this point.  
	 
	Now, let's prove that if $e$ satisfies both conditions of non-incidence, there exists $\varepsilon$ small enough, so there is rigidity for small deformations. First, we will assume that 
	
	\begin{equation}
		c_{2, 1}(\Omega) = c_{2, 1}(\mathcal E_e). \label{c2inci}
	\end{equation}
	
	We introduce the following definition:
	
	\begin{definition}
		An interval $(\alpha, \beta)$ is called a $(p, q, \varepsilon)$ interval if for any $\varepsilon$ small deformation $\mu$ of $\mathcal E_e$, satisfying \eqref{c2inci}, $t_{p, q}, T_{p, q} \in (\alpha, \beta)$, while the length of orbits of different types and the multiples of the length of the boundary are not present in this interval. 
	\end{definition}
	
	Note that if we decrease $\varepsilon$, an interval continues to be a $(p, q, \varepsilon)$. 
	
	Now assume we have constructed $(p, q, \varepsilon)$ intervals for every $(p, q) \in \mathcal F$ with some uniform $\varepsilon$. Then all the $(p, q)$ satisfy Lemma \ref{lemmaz6} for $\Omega$. Also, for an ellipse $\mathrm{singsupp}_\Omega(t) \cap (a, b) = T_{p, q}$. That means that $t_{p, q}(\Omega) = T_{p, q}(\Omega) = T_{p, q}(\mathcal E_e)$. So, $\Omega$ preserves $p/q$ caustic for every $(p, q) \in \mathcal F$. Now we can use our caustic result (maybe for smaller $\varepsilon$) and prove that $\Omega$ is itself an ellipse. 
	
	Note that if $\Omega$ is an ellipse, then it is isometric with $\mathcal E_e$. It follows from the fact that $\mathcal E_e$ and $\Omega$ have the same perimeter and $c_{2,1}$. That corresponds to them having the same $\beta_1$ and $\beta_3$ in a sense of \cite{sorrent}. According to Proposition $1$ from there, the ellipses should be isometric.
	
	So, our goal is to construct $(p, q, \varepsilon)$ intervals for every element of $\mathcal F$. 
	
	We start with $\mathcal F_2$. There are only a finite number of elements in $\mathcal F_2$, so we do not care about uniformity of $\varepsilon$. 
	
	\begin{lemma}
		Assuming non-incidence, there exists a $(p, q, \varepsilon)$ interval for every $(p, q) \in \mathcal F_2$.
		\label{lemmaz7}
	\end{lemma}
	\begin{proof}
		First of all, we consider $T_{p, q}(\mathcal E_e)$. Due to the lemma \ref{lemmaz4}, for every neighborhood $(\alpha, \beta)$ of  $T_{p, q}(\mathcal E_e)$ there exists $\varepsilon$, such that $L_{p, q}(\Omega) \in (\alpha, \beta)$ for every $(p, q)$ orbit. The problem is to prove that there would be no other lengths in this interval. Then we note that for an ellipse, the only limit points of the length spectrum are the multiples of the length of the boundary. This also means that the length of the boundary is preserved. Due to the non-incidence condition,  $T_{p, q}(\mathcal E_e)$ is not a limit point of the spectrum. Due to the same condition, it does not coincide with the lengths of orbits of other types. Hence $T_{p, q}$ is isolated from the rest of the spectrum and the multiples of the lengths of the boundary by some neighborhood. We can choose $(a, b)$ in this way. Now we just need to prove that for $\Omega$ other periodic orbits cannot enter this interval. Lets discuss the types of these orbits one by one.
		
		The orbits with $\bar{p} \ge p_0$ cannot enter the interval due to Lemmata \ref{lemmaz1} and \ref{lemmaz2}. Orbits with small $\bar{p}$, but large $\bar{q}$ also cannot enter, because they are close to the multiples of the lengths of the boundary due to \eqref{lengthseries}. So, there are only a finite number of $(\bar{p}, \bar{q})$ types we have to deal with. If $\bar{p}/\bar{q} \ne 1/2$, then we use Lemma \ref{lemmaz4} and say that the length of every $(\bar{p}, \bar{q})$ orbit is close to the respective length of an orbit for an ellipse and thus is outside of the interval. Of course, we may decrease $\varepsilon$ to obtain this. For $2\bar{p} = \bar{q}$ we use Lemma \ref{lemmaz5}. The multiples of the length of the boundary also cannot enter this interval, since they are constant.
		
		So, no other length of an orbit can enter the interval, so it is $(p, q, \varepsilon)$.
	\end{proof}
	
	We are now only left with $\mathcal F_1$. To prove the similar result for them, we need the following lemma:
	
	\begin{lemma}
		The length of the boundary of an ellipse $E_e$ is not approached by the lengths of the orbits with $p>1$. 
		\label{lemmaz8}
	\end{lemma}
	
	\begin{proof}
		First, we can not consider orbits with $p \ge p_0$ due to Lemma \ref{lemmaz2}. Also, we do not consider orbits with $2 \le p < p_0$ and large $q_0$, due to \eqref{lengthseries}, since they are all close to $p\ell$. Then, we are only left with finite amount of types. Types with $p/q \ne 1/2$ only have one orbit length each due to the caustic, so there won't be any approach. We are only left with orbits with rotation number $1/2$. They break up into $2$ bouncing ball trajectories, that go along the axes of an ellipse and into orbits that stay tangent to hyperbolae. Overall, there would be only a finite amount of lengths in this class, so they won't approach $\ell$. 
	\end{proof}
	
	\begin{lemma}
		Assuming non-incidence, here exists a $(p, q, \varepsilon)$ interval for every $(p, q) \in \mathcal F_1$ with uniform $\varepsilon$.
		\label{lemmaz9}
	\end{lemma}
	
	\begin{proof}
		We know that there are infinte amount of orbit types inside $\mathcal F_1$ and that their lengths approach the length of the boundary due to \eqref{lengthseries}. First, we take an ellipse $\mathcal E_e$. We use Lemma \ref{lemmaz8} and non-incidence and obtain a neighborhood of the length of the boundary without orbits with $p>1$. Due to the ideas in Lemma \ref{lemmaz8}, using Lemma \ref{lemmaz4} and \ref{lemmaz5} and decreasing the neighborhood a little we get that it is free of orbits with $p>2$ for any $\varepsilon$ deformation of an ellipse. 
		
		Now due to \eqref{lengthseries}, for large enough $q \ge \hat{q}$ (independent of $\mu$), all the lengths of orbits of $(1, q)$ type are guaranteed to remain in this neighborhood for every small $\varepsilon$ deformation. There are some $q$ that are not guaranteed to be there, but there are only finitely many of them, so we use the same approach for them, as in Lemma \ref{lemmaz7}. 
		
		We may also assume that $\hat{q}$ is so large and $\varepsilon$ is so small, so that 
		
		\begin{equation}
			\left| q^{-3} O(\left\| \mu \right\| ) + O(q^{-4})\right| \le \frac{c_{2, 1}}{100} q^{-3} 
		\end{equation}
		
		in \eqref{lengthseries} for $p = 1$ and $q \ge \hat{q}$. Now we can show that
		
		\begin{equation}
			\left(\ell - c_{2, 1}q^{-2} -  \frac{c_{2, 1}}{10} q^{-3},  \ell - c_{2, 1}q^{-2} +  \frac{c_{2, 1}}{10} q^{-3} \right) 
			\label{lieinter}
		\end{equation}
		
		are $(1, q, \varepsilon)$ intervals (maybe for smaller, but uniform $\varepsilon$) for $q > \hat{q}$. We may assume these intervals fully lie within this neighborhood of the boundary, otherwise we increase $\hat{q}$ a little. First of all, every $(1, q)$ orbit should be inside of the interval for the deformation. Since the interval is inside the neighborhood, orbits with $p>1$ cannot enter the interval. Orbits of type $(1, q)$ with $q>\hat{q}$ also cannot enter, since the intervals do not intersect:
		
		\begin{equation}
			\ell - c_{2, 1}q^{-2} +  \frac{c_{2, 1}}{10} q^{-3} < \ell - c_{2, 1}(q+1)^{-2} -  \frac{c_{2, 1}}{10} (q+1)^{-3}.
		\end{equation}
		
		We are only left with a finite amount of orbits of type $(1, q)$ for $q \le \hat{q}$. Since for an ellipse,
		
		\begin{equation}
			t_{1, 2} < T_{1, 2} < T_{1, 3} < \ldots < T_{1, \hat{q}} <  \ell - c_{2, 1}\hat{q}^{-2} +  \frac{c_{2, 1}}{10}\hat{q}^{-3} < \ell - c_{2, 1}(\hat{q}+1)^{-2} -  \frac{c_{2, 1}}{10} (\hat{q}+1)^{-3}  < T_{1, \hat{q}+1} 
		\end{equation}
		
		these types' lengths should be changed at least by
		
		\begin{equation}
			\ell - c_{2, 1}(\hat{q}+1)^{-2} -  \frac{c_{2, 1}}{10} (\hat{q}+1)^{-3} - \ell + c_{2, 1}\hat{q}^{-2} - \frac{c_{2, 1}}{10}\hat{q}^{-3} > 0
		\end{equation}
		
		to get us into a problem. However, since there are only a finite amount of types, we can use Lemmata \ref{lemmaz4} and \ref{lemmaz5} and choose small enough $\varepsilon$ to avoid this. So, no length can enter these intervals and we have proven the lemma.
		
	\end{proof}
	
	\subsection{Proving an assumption on $c_{2, 1}$ }
	
	Now we need to prove our assumption  \eqref{c2inci}. We propose the following lemma:
	
	\begin{lemma}
		Assume $\mathcal E_e$ satisfies non-incidence condition. Then, there exists $\varepsilon$ such that every $\varepsilon$ small deformation $\Omega$, that preserves wave trace singularities, also preserves $c_{2, 1}$.  
	\end{lemma} 
	\begin{proof}
		Firstly, we say that $c_{2, 1}(\Omega)$ may be only of order $\varepsilon$ different from $c_{2, 1}(\mathcal E_e)$, since it depends on curvature. Since we have non-incidence condition, we may follow the ideas of Lemmata \ref{lemmaz8} and \ref{lemmaz9} and consider the situation only in the neighborhood of the boundary. For both domains only the lengths of $(1, q)$ orbits for large $q$ will be present there and those would be forced to lie in \eqref{lieinter}. 
		
		We propose the following map $g: \mathbb{R} \rightarrow \mathbb{R}$:
		
		\begin{equation}
			g(\ell - x) = \frac{1}{\sqrt{c_{2, 1}(\mathcal E_e)}}x^{-1/2}
		\end{equation}
		
		Now we use this function to map the neighborhood of the boundary onto the real line. We see, that wave trace of $\mathcal E_e$ has singularities at 
		
		\begin{equation}
			g\left( \ell - c_{2, 1}(\mathcal E_e) q^{-2} + O(q^{-4})\right) = q + O(1/q), q\rightarrow \infty.
		\end{equation}
		
		Meanwhile, the intervals with singularities for $\Omega$ are contained in:
		
		\begin{equation}
			\sqrt{\frac{c_{2, 1}(\Omega)}{c_{2, 1}(\mathcal E_e)}}\left(q - 1/5, q+ 1/5 \right) 
		\end{equation}
		
		Of course, if the root is not equal to $1$ (it is close to $1$), then there would be an interval that has no singularities of an ellipse lying within it. Since there would be a singularity for $\Omega$ inside an interval, the singularities of the wave trace would not match, giving us a contradiction. So, the root is $1$, so $c_{2, 1}$ coincide.
		
	\end{proof}
	\subsection{Eccentricities with non-incidence condition}
	
	Now we should ask a question: for which $e$ does the non-incidence relation hold? The relation has several requirements, one of which (that elements of $\mathcal F$ are not incident to the multiples of perimeter) holds due to \eqref{q0select}. We only have to check when the lengths of orbits in $\mathcal F$ or the perimeter of an ellipse are realized in the length spectrum using anther orbit.
	
	When we talk of a length of an orbit, we mean the function, depending on $e$, that gives the length of orbits of specific type. There are $4$ sets of types: tangent to caustics (type given by $(p, q)$), to hyperbolae ($(\tilde{p}, q)$) and minor and major axes bouncing balls (type given by $p$).
	
	Since we have earlier proven the lengths of orbits (and perimeter) to be holomorphic over $e$, we have $2$ possibilities. The first is that for some pair of lengths the incidence happens as an identity (in this pair we'll call the element of $\mathcal F$ the first, and another one - the second). Alternatively, for each pair of lengths the incidence happens only a finite amount of times for $e < e_{\max}$. Let's rule out the first option.
	
	\begin{lemma}
		This identity cannot happen for large enough $q_0$.
	\end{lemma}
	
	\begin{proof}

	Assume this identity holds. Then, the second orbit cannot be a minor axis one, since the length goes to $0$ as $e \rightarrow 1$. If the incident orbit is a major axis or tangent to caustic one, then the first and second share the same $p$, as we also can take $e \rightarrow 1$ (if the first one is a perimeter, we count $p=1$). But then we have a contradiction, since $T_{p, q}$ increases in $q$. 
	
	The last possibility is that the second one is tangent to hyperbola. Since it has a rotation number $1/2$, $q$ is bounded by $2p_0$, and its short axis libration number $\tilde{p}$ is also bounded. So, there are only a finite amount of these types. So, by making $q_0$ large enough, we can ensure that the first is just a perimeter of an ellipse, so $l_{\tilde{p}, q}^2(e) \equiv \ell(\mathcal E_e)$. By using \eqref{lenghyper} and letting $e \rightarrow \cos\frac{\pi \tilde{p}}{q}$ from the right, we get:
	
	\begin{equation}
		2q\sqrt{1 - e^2} = 4E(e)
	\end{equation}
	
	or 
	
	\begin{equation}
		q\sin\frac{\pi \tilde{p}}{q} = 2E\left( \cos\frac{\pi \tilde{p}}{q}\right) 
	\end{equation}
	
	If $\tilde{p} \ge 2$, then the left part is at least $4$, while the right one is less that $\pi$. So, $\tilde{p} = 1$. After that, the left side increases in $q$, while the right one - decreases. Since $q \ge 4$, we can check that
	
	\begin{equation}
		E\left( \sqrt{2}/2\right) < \sqrt{2}, 
	\end{equation}

	so there is no identical incidence.
	
	\end{proof}  

	Now each pair only gives us a finite amount of incidences. What pairs can even give incidences? Orbits with $p \ge p_0$ are too long for us, so we do not consider them. Pick an orbit in $\mathcal F_2$. Since there are only a finite amount of bouncing balls and tangent to hyperbolae types left, they in total give a finite amount of points. Due to \eqref{q0select}, it won't be incident to the multiples of the perimeter and with caustic orbits with other $p$ and large $q$. They won't be incident to the orbits with the same $p$, as already mentioned. So we are left with a finite amount of orbits, so a finite amount of incidences. Finally, $\mathcal F_2$ is finite, so together they also provide a finite amount.
	
	Now we study incidences of $\mathcal F_1$ and the perimeter. They are not incident to the major axis bouncing ball and caustic orbits with $p=1$, while these orbits for $p>1$ are too long (length at least $8$, while ours have $\le 2\pi)$. So we are only left with minor bouncing balls and orbits tangent to hyperbolae. 
	
	Some of those really generate an incidence with perimeter, but we want to further restrict those orbits. Specifically, we say that orbits, tangent to hyperbolae, with short axis libration number $p>1$ are too long to produce an incidence with perimeter. We claim that they also have a length of at least $8$. 
	
	First, we can prove that as $e$ increases, the lengths of these types don't increase. Since their length is the maximum of the lengths functional on certain set, and because we decreased all the chord lengths by increasing $e$ with fixed semi-major axis (essentially contracting it vertically), the maximum would not increase. Then we are only left to prove the inequality as $e\rightarrow 1$. 
	
	One can see from \eqref{hyperb} that as $e \rightarrow 1$ for fixed $\omega$, we have that $k^{-1} \rightarrow 1$. That means that the eccentricity of hyperbola goes to $1$, so the distance between its $2$ components approaches focal distance $2e$. Since orbits with $\tilde{p}>1$ go between these components at least $4$ times, their lengths should approach no less than $8$. So, they cannot be incident to the perimeter (or to anyone in $\mathcal F_1$).
	
	So, perimeter incidences can only be generated by minor bouncing balls and tangent to hyperbolae with short axis libration number $\tilde{p} = 1$. Denote this set of $e$ as $\mathcal A_e$. Since for any $e_{\max}$ there is only a finite amount of incidences, $\mathcal A_e$ is a locally finite set. 
	
	Denote $\mathcal I_e$ as a set of all $e$ with incidence. Look at all the accumulation points. Assume we have a sequence of elements of $\mathcal I_e$ approaching some value. Since the set of incidences of $\mathcal F_2$ and $\mathcal A_e$ are locally finite, we are not considering those. So, all the incident orbits in our sequence have $\tilde{p}=1$. Moreover, if elements are not approaching $1$, then $e$ in the sequence is bounded. Since each type can only be incident finitely many times if $e$ is bounded, $q \rightarrow \infty$. That means that orbit lengths approach the perimeter of an ellipse. So, at the limit point there should be incidence with perimeter, hence it is in $\mathcal A_e$. So, $\mathcal I_e$ is a small set.
	
	\begin{figure}[t]
		\includegraphics[width=18cm]{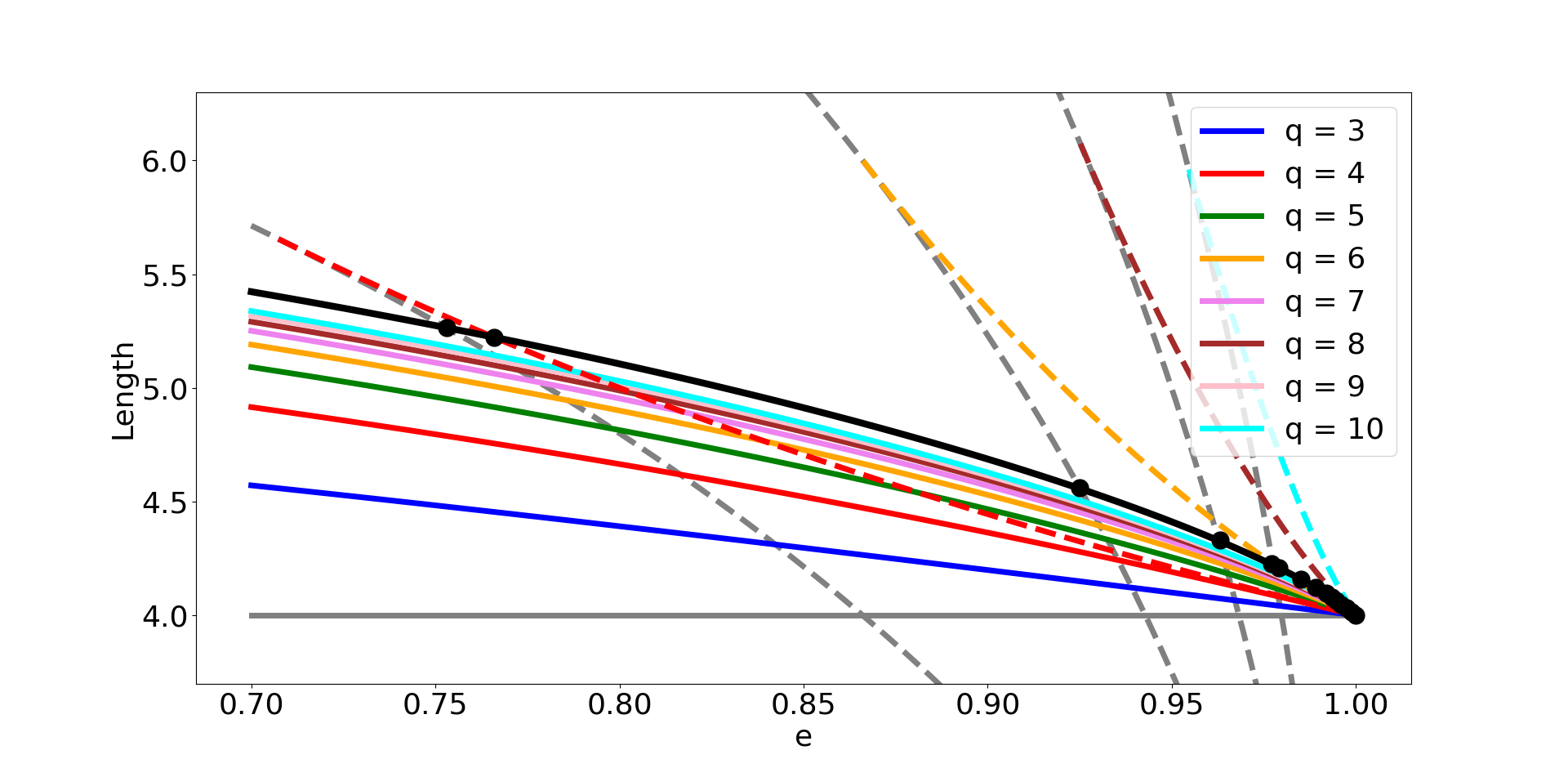}
		\centering
		\caption{Plot of lengths of periodic orbits in an ellipse for large values of $e$ near the perimeter. The perimeter is plotted with a black line, major bouncing ball is plotted with a solid gray line, while minor are plotted with a dashed one. Colored lines correspond to orbits tangent to ellipses (solid), and hyperbolae (dashed), with colors corresponding to their $q$ (they all have $p$ and $\tilde{p}$ equal to $1$, since other orbits are too long to be seen). Elements of $\mathcal A_e$ are plotted as black dots on the perimeter curve. Some elements of $\mathcal I_e$ correspond to intersections between dashed and colored curves.}
		\label{incigrpah}
	\end{figure}
	
	One can compute elements of $\mathcal A_e$ numerically to see how are they located. Then, the first few elements take the following form: $e \approx 0.753$ (incidence of perimeter to minor bouncing ball of period $2$), $e \approx 0.766$ (incidence to hyperbolic $(1, 4)$), $e \approx 0.925, 0.963, 0.978$ (bouncing ball with period $3$, $4$, $5$). The next few hyperbolic incidences: $e \approx 0.979, 0.998$ (of types $(1, 6)$ and $(1, 8)$).

	\appendix
	
	\section{Length spectrum and Birkhoff conjecture}
	
	When dealing with the Laplace spectrum and its relation with a dynamical picture, one should always be careful with possible cancellations. Hence, it can be challenging to extend some of the results to its rigidity. Length spectrum is much easier in this regard, since there are no cancellations. Hence, a lot of results on Birkhoff conjecture can be applied in studying local or global spectral rigidity of ellipses. 
	
	One example of it is our paper, another one is \cite{ks}. From the latter, it also follows that all ellipses are locally length spectrally determined. We extended this rigidity to length spectrum near the multiples of perimeter.
	
	Currently, we have only applied local Birkhoff conjecture results to this problem. There are, however, global results, like \cite{bialy} and \cite{bialymironov}. These works are hard to use in a rigidity problem for three reasons. First, they require a foliation of a phase space by invariant curves, meaning not only rational caustics. That may be challenging to get from the length spectrum, since it only deals with periodic points. Secondly, even the needed rational caustics can have arbitrary large $p$, so one has to study very large elements of the spectrum. Lastly, they still require $1/4$ caustic for example, and globally it is hard to estimate the length of $4$-periodic orbits, since $4$ is not large enough. 
	
	In this section, we will assume that a strong version of the global Birkhoff conjecture holds. Then, we will prove that all ellipses are uniquely determined by their length spectrum. We assume the following holds:
	
	\begin{conjecture}
		There exists $p_0 \ge 1$, such that for all $q_0 \ge 2p_0+1$ the following holds. Let $\Omega$ be a $C^\infty$ smooth convex domain that has a $p/q$ caustic for any $1 \le p \le p_0$ and $q \ge q_0$. Then, $\Omega$ is an ellipse. 
		\label{aconjecture}
	\end{conjecture}
	
	We will prove the following:
	
	\begin{proposition}
		Suppose Conjecture \ref{aconjecture} holds. Let $\Omega$ be a $C^\infty$ smooth strictly convex domain. Assume the length spectrum of $\Omega$ coincides with the spectrum of some ellipse $\mathcal E$. Then, $\Omega$ and $\mathcal E$ are isometric.  
	\end{proposition}
	
	\begin{proof}
		We know, that every length of the periodic point in the ellipse can be explicitly described. Hence, the only accumulation points of the length spectrum are multiples of the perimeter. The perimeter of $\Omega$ is also an accumulation point of the length spectrum, hence
		
		\begin{equation}
			\ell(\Omega) = m \cdot \ell(\mathcal E)
			\label{multperimeter}
		\end{equation}
		
		for some $m \in \mathbb{N}$. We will prove that $\Omega$ has a $1/q$ caustic for every $q$ large enough. For other $p>1$ the proof is identical. According to \cite{mm}, the elements of the length spectrum with large $q$ allow for asymptotic expansion, similar to \eqref{lengthseries}:
		
		\begin{equation}
			L_{1, q} = \ell(\Omega) - c_{2, 1}(\Omega) q^{-2} + O(q^{-4}), \; q\rightarrow \infty.
		\end{equation}
		Moreover, $t_{1, q}$ and $T_{1, q}$ are very close:
		
		\begin{equation}
			T_{1, q} - t_{1, q} =  O(q^{-k}), \; q\rightarrow \infty.
		\end{equation}
		
		for any $k$. Since $c_{2,1} > 0$, if there are arbitrary large $q$ with $t_{1, q} < T_{1, q}$, the length spectrum would have $2$ points of order $q^{-2}$-close to $m \cdot \ell(\mathcal E)$ that are $O(q^{-k})$ close to each other. These points should be somehow realized in the spectrum of an ellipse. Assume they are realized by orbits of type $(m, q_1)$ and $(m, q_2)$. Since $T_{m, q}$ is increasing in $q$, $q_1$ and $q_2$ should be large enough for orbit length to be near $m \cdot \ell(\mathcal E)$.
		
		Since ellipses allow $m/{q_1}$ and $m/{q_2}$ caustic, $q_1 \ne q_2$. However, we can apply \eqref{lengthseries} for ellipse now and obtain that even if $q_2 = q_1 + 1$, $T_{m, q_1}$ and $T_{m, q_2}$ should be no less than of order $q^{-3}$ close to each other if they are of order $q^{-2}$ close to the multiple of the perimeter. So, points in the length spectrum are too close to each other to be realized by orbits with $p = m$. So, one of those orbits has $p \ne m$. Since this happens for arbitrary large $q$, we get that $m\cdot\ell(\mathcal E)$ is accumulated by the lengths of orbits with $p \ne m$. This cannot happen for an ellipse. So, for large enough $q$ in $\Omega$ we have $t_{1, q} = T_{1, q}$.
		
		If we let $q$ be large enough, there would be a smooth generating $q$-loop function in $\Omega$. Then, since the minimal and maximal length are equal, we have that $\Omega$ allows $1/q$ caustic. Having done this for all needed $p \le p_0$, we can apply Conjecture \eqref{aconjecture} and get that $\Omega$ is also an ellipse. 
		
		If we know that $\Omega$ is an ellipse, we can prove that it is isometric to $\mathcal E$. First of all, they should have the same length, since we can apply \eqref{multperimeter} in reverse. Secondly, their $c_{2, 1}$ should coincide. To see this, we can apply \eqref{lengthseries} for both domains. Using the same arguments, as in \eqref{ctohypo}, we see that if $c_{2, 1}$ differ, lengths of some orbits with $p = 1$ and large $q$ in $\Omega$ cannot be realized as lengths of orbits with $p = 1$ in an ellipse, or vice versa. Then, they should be realized by $p>1$. So, we get that the perimeter of $\Omega$ (or $\mathcal E$) is an accumulation point of lengths of orbits with $p>1$. This cannot happen, so $c_{2,1}$ coincide. Then, once again, we get that $\mathcal E$ and $\Omega$ are isometric by Proposition $1$ of \cite{sorrent}.
		
	\end{proof}

	\medskip
	
	\printbibliography

@article{hks,
author = {Huang, Guan and Sorrentino, Alfonso},
year = {2018},
month = {04},
pages = {},
title = {Nearly Circular Domains Which Are Integrable Close to the Boundary Are Ellipses},
volume = {28},
journal = {Geometric and Functional Analysis}
}

@article{stirling, title={Calculus of Finite Differences. By C. Jordan. Second edition. Pp. xxii, 654. 1950. (Chelsea Publishing Co., New York)}, volume={35},  number={314}, journal={The Mathematical Gazette}, publisher={Cambridge University Press}, author={Jones, C. W.}, year={1951}, pages={145–146}}

@misc{ks,
      title={On the Local Birkhoff Conjecture for Convex Billiards}, 
      author={Vadim Kaloshin and Alfonso Sorrentino},
      year={2018},
      archivePrefix={arXiv},
      primaryClass={math.DS}
}

@article{adsk,
 ISSN = {0003486X},
 author = {Artur Avila and Jacopo De Simoi and Vadim Kaloshin},
 journal = {Annals of Mathematics},
 number = {2},
 pages = {527--558},
 publisher = {Annals of Mathematics},
 title = {An integrable deformation of an ellipse of small eccentricity is an ellipse},
 volume = {184},
 year = {2016}
}

@article{birk,
author = {George D. Birkhoff},
title = {{On the periodic motions of dynamical systems}},
volume = {50},
journal = {Acta Mathematica},
number = {none},
publisher = {Institut Mittag-Leffler},
pages = {359 -- 379},
year = {1927}
}

@article{pori,
 ISSN = {0003486X},
 author = {Hillel Poritsky},
 journal = {Annals of Mathematics},
 number = {2},
 pages = {446--470},
 publisher = {Annals of Mathematics},
 title = {The Billard Ball Problem on a Table With a Convex Boundary--An Illustrative Dynamical Problem},
 volume = {51},
 year = {1950}
}

@article{math, title={Glancing billiards}, volume={2}, number={3-4}, journal={Ergodic Theory and Dynamical Systems}, publisher={Cambridge University Press}, author={Mather, John N.}, year={1982}, pages={397–403}}

@article{bialy,
author = {Bialy, Misha},
journal = {Mathematische Zeitschrift},
number = {1},
pages = {147-154},
title = {Convex billiards and a theorem by E. Hops.},
volume = {214},
year = {1993},
}

@article{gutkin,
  title={Billiard dynamics: an updated survey with the emphasis on open problems.},
  author={Eug{\`e}ne Gutkin},
  journal={Chaos},
  year={2012},
  volume={22 2},
  pages={
          026116
        }
}

@article{kac,
 ISSN = {00029890, 19300972},
 author = {Mark Kac},
 journal = {The American Mathematical Monthly},
 number = {4},
 pages = {1--23},
 publisher = {Mathematical Association of America},
 title = {Can One Hear the Shape of a Drum?},
 volume = {73},
 year = {1966}
}

@article{Zelditch2004TheIS,
  title={The inverse spectral problem},
  author={Steve Zelditch},
  journal={Surveys in differential geometry},
  year={2004},
  volume={9},
  pages={401-467}
}

@article{artin,
    author = {Heath-Brown, D. R.},
    title = "{Artin's conjecture for primitive roots}",
    journal = {The Quarterly Journal of Mathematics},
    volume = {37},
    number = {1},
    pages = {27-38},
    year = {1986},
    month = {03},
    issn = {0033-5606},
    doi = {10.1093/qmath/37.1.27},
    url = {https://doi.org/10.1093/qmath/37.1.27}
}

@article{brown,
 ISSN = {00029947},
 URL = {http://www.jstor.org/stable/2154304},
 author = {Russell M. Brown},
 journal = {Transactions of the American Mathematical Society},
 number = {2},
 pages = {889--900},
 publisher = {American Mathematical Society},
 title = {The Trace of the Heat Kernel in Lipschitz Domains},
 urldate = {2022-10-20},
 volume = {339},
 year = {1993}
}

@article{mm,
author = {Shahla Marvizi and Richard Melrose},
title = {{Spectral invariants of convex planar regions}},
volume = {17},
journal = {Journal of Differential Geometry},
number = {3},
publisher = {Lehigh University},
pages = {475 -- 502},
year = {1982},
doi = {10.4310/jdg/1214437138},
URL = {https://doi.org/10.4310/jdg/1214437138}
}

@misc{hezzel,
  doi = {10.48550/ARXIV.1907.03882},
  
  url = {https://arxiv.org/abs/1907.03882},
  
  author = {Hezari, Hamid and Zelditch, Steve},
  
  title = {One can hear the shape of ellipses of small eccentricity},
  
  publisher = {arXiv},
  
  year = {2019},
  
  copyright = {arXiv.org perpetual, non-exclusive license}
}

@article{hezzel12,
title = "$C^{\infty}$ spectral rigidity of the ellipse",
author = "Hamid Hezari and Steve Zelditch",
year = "2012",
doi = "10.2140/apde.2012.5.1105",
language = "English (US)",
volume = "5",
pages = "1105--1132",
journal = "Analysis and PDE",
issn = "2157-5045",
publisher = "Mathematical Sciences Publishers",
number = "5",

}

@article{vi21,
  title={Robin Spectral Rigidity of the Ellipse},
  author={Amir Vig},
  journal={The Journal of Geometric Analysis},
  year={2020},
  pages={1-58}
}

@Inbook{kato,
author="Kato, Tosio",
title="Analytic perturbation theory",
bookTitle="Perturbation Theory for Linear Operators",
year="1995",
publisher="Springer Berlin Heidelberg",
address="Berlin, Heidelberg",
pages="364--426",
isbn="978-3-642-66282-9",
}

@book{abse,
author="M. Abramowitz and I. A. Segun: editors",
title="Handbook of Mathematical Functions.",
year="1965",
address="Dover, New York",
}

@article{sieber,
doi = {10.1088/0305-4470/30/13/011},
url = {https://dx.doi.org/10.1088/0305-4470/30/13/011},
year = {1997},
month = {jul},
publisher = {},
volume = {30},
number = {13},
pages = {4563},
author = {Martin Sieber},
title = {Semiclassical transition from an elliptical to an oval billiard},
journal = {Journal of Physics A: Mathematical and General}
}

@article {zel,
    AUTHOR = {Zelditch, Steve},
     TITLE = {Inverse spectral problem for analytic domains. {II}. {$Z_2$}-symmetric domains},
   JOURNAL = {Ann. of Math. (2)},
  FJOURNAL = {Annals of Mathematics. Second Series},
    VOLUME = {170},
      YEAR = {2009},
    NUMBER = {1},
     PAGES = {205--269},
      ISSN = {0003-486X},
   MRCLASS = {58J53 (35P20 35R30 37D50)},
  MRNUMBER = {2521115},
MRREVIEWER = {David Borthwick},
       DOI = {10.4007/annals.2009.170.205},
       URL = {https://doi.org/10.4007/annals.2009.170.205},
}

@article{Zehnder1989,
author = {Zehnder, Eduard, Salamon, Dietmar},
journal = {Commentarii mathematici Helvetici},
number = {1},
pages = {84-132},
title = {KAM theory in configuration space.},
url = {http://eudml.org/doc/140144},
volume = {64},
year = {1989},
}

@article{dkw,
author = {Jacopo De Simoi and Vadim Kaloshin and Qiaoling Wei and Hamid Hezari},
title = {{Dynamical spectral rigidity among $Z_2$-symmetric strictly convex domains close to a circle (Appendix B coauthored with H. Hezari)}},
volume = {186},
journal = {Annals of Mathematics},
number = {1},
publisher = {Department of Mathematics of Princeton University},
pages = {277 -- 314},
year = {2017},
doi = {10.4007/annals.2017.186.1.7},
URL = {https://doi.org/10.4007/annals.2017.186.1.7}
}

@articleInfo{sorrent,
title = {Computing Mather's $\beta$-function for Birkhoff billiards},
journal = {Discrete and Continuous Dynamical Systems},
volume = {35},number = {10},pages = {5055-5082},
year = {2015},
issn = {1078-0947},
doi = {10.3934/dcds.2015.35.5055},
url = {/article/id/9ae4f4b9-c3cd-4fff-a7e5-0750dc17f126},
author = {Alfonso Sorrentino}
}

@unpublished{popovtopalov,
  TITLE = {{From KAM Tori to Isospectral Invariants and Spectral Rigidity of Billiard Tables}},
  AUTHOR = {Popov, G and Topalov, P},
  URL = {https://hal.archives-ouvertes.fr/hal-02012786},
  NOTE = {working paper or preprint},
  YEAR = {2019},
  MONTH = Feb
}

@article{chang,
author = {Chang,Shau‐Jin  and Friedberg,Richard },
title = {Elliptical billiards and Poncelet’s theorem},
journal = {Journal of Mathematical Physics},
volume = {29},
number = {7},
pages = {1537-1550},
year = {1988},
doi = {10.1063/1.527900},
}

@article{bialymironov,
author = {Misha Bialy and Andrey E. Mironov},
title = {{The Birkhoff-Poritsky conjecture for centrally-symmetric billiard tables}},
volume = {196},
journal = {Annals of Mathematics},
number = {1},
publisher = {Department of Mathematics of Princeton University},
pages = {389 -- 413},
year = {2022},
doi = {10.4007/annals.2022.196.1.2},
URL = {https://doi.org/10.4007/annals.2022.196.1.2}
}

@article{keagan,
  doi = {10.48550/ARXIV.2209.11721},
  
  url = {https://arxiv.org/abs/2209.11721},
  
  author = {Callis, Keagan G.},
  
  title = {Absolutely Periodic Billiard Orbits of Arbitrarily High Order},
  
  publisher = {arXiv},
  
  year = {2022},
  
  copyright = {Creative Commons Attribution 4.0 International}
}

@article{treschev,
title = {Billiard map and rigid rotation},
journal = {Physica D: Nonlinear Phenomena},
volume = {255},
pages = {31-34},
year = {2013},
issn = {0167-2789},
doi = {https://doi.org/10.1016/j.physd.2013.04.003},
url = {https://www.sciencedirect.com/science/article/pii/S0167278913001103},
author = {D. Treschev}
}

@article{wangzhang,
  doi = {10.48550/ARXIV.2211.03182},
  url = {https://arxiv.org/abs/2211.03182},
  author = {Wang, Qun and Zhang, Ke},
  title = {Gevrey regularity for the formally linearizable billiard of Treschev},
  publisher = {arXiv},
  year = {2022},
  copyright = {Creative Commons Attribution 4.0 International}
}

@article{iantchenko,
  title={Birkhoff normal forms in semi-classical inverse problems},
  author={Alexei Iantchenko and Johannes Sjoestrand and Maciej Zworski},
  journal={Mathematical Research Letters},
  year={2002},
  volume={9},
  pages={337-362}
}

\end{document}